\documentclass[oneside,english]{amsart}
\usepackage[T1]{fontenc}
\usepackage[latin9]{inputenc}
\usepackage{amsbsy}
\usepackage{amssymb}

\makeatletter
\usepackage[latin9]{inputenc}
\usepackage{color}
\usepackage{amsthm}
\usepackage{amssymb}
\usepackage{hyperref}
\makeatletter
\numberwithin{equation}{section}
\numberwithin{figure}{section}

\usepackage{times}
\usepackage{amsfonts}\usepackage{mathabx}\usepackage{amsrefs}\usepackage{mathrsfs}\usepackage{tikz}\usetikzlibrary{decorations,decorations.pathmorphing}

\newtheorem{thm}{Theorem}[section]
\newtheorem{prop}[thm]{Proposition}
\newtheorem{cor}[thm]{Corollary}
\newtheorem{lem}[thm]{Lemma}
\newtheorem{fact}[thm]{Fact}
\theoremstyle{definition}
\newtheorem{rem}[thm]{Remark}
\newtheorem{exa}[thm]{Example}
\newtheorem{ques}[thm]{Question}
\newtheorem{notation}[thm]{Notation}
\newtheorem{theorem}{Theorem}

\keywords{}
\date{}

\makeatother

\usepackage{babel}
\begin{document}
\title{Growth of periodic Grigorchuk groups }
\author{Anna Erschler and Tianyi Zheng}
\thanks{The work of the authors is supported by the ERC grant GroIsRan. The
first named author also thanks the support of the ANR grant MALIN. }
\begin{abstract}
On torsion Grigorchuk groups we construct random walks of finite entropy
and power-law tail decay with non-trivial Poisson boundary. Such random
walks provide near optimal volume lower estimates for these groups.
In particular, for the first Grigorchuk group $G$ we show that its
growth $v_{G,S}(n)$ satisfies $\lim_{n\to\infty}\log\log v_{G,S}(n)/\log n=\alpha_{0}$,
where $\alpha_{0}=\frac{\log2}{\log\lambda_{0}}\approx0.7674$, $\lambda_{0}$
is the positive root of the polynomial $X^{3}-X^{2}-2X-4$.
\end{abstract}

\maketitle

\section{Introduction}

Let $\Gamma$ be a finitely generated group and $S$ be a finite generating
set of $\Gamma$. The word length $l_{S}(\gamma)$ of an element $\gamma\in\Gamma$
is the smallest integer $n\ge0$ for which there exists $s_{1},\ldots,s_{n}\in S\cup S^{-1}$
such that $\gamma=s_{1}\ldots s_{n}$. The growth function $v_{\Gamma,S}(n)$
counts the number of elements with word length $l_{S}(\gamma)\le n$,
that is 
\[
v_{\Gamma,S}(n)=\#\{\gamma\in\Gamma:\ l_{S}(\gamma)\le n\}.
\]
A finitely generated group $\Gamma$ is of polynomial growth if there
exists $0<D<\infty$ and a constant $C>0$ such that $v_{\Gamma,S}(n)\le Cn^{D}$.
It is of exponential growth if there exists $\lambda>1$ and $c>0$
such that $v_{\Gamma,S}(n)\ge c\lambda^{n}$. If $v_{\Gamma,S}(n)$
is sub-exponential but $\Gamma$ is not of polynomial growth, we say
$\Gamma$ is a group of intermediate growth. Any group of polynomial
growth is virtually nilpotent by Gromov's theorem \cite{gromov}.
For some classes of groups the growth is either polynomial or exponential
(for example, for solvable groups by results of Milnor and Wolf \cite{Milnor}, \cite{Wolf},
and for linear groups as it follows from Tits alternative \cite{Tits}).
The first examples of groups of intermediate growth are constructed
by Grigorchuk in \cite{Grigorchuk84}. 

Grigorchuk groups are indexed by infinite strings $\omega$ in $\{{\bf 0},{\bf 1},{\bf 2}\}^{\infty}$
(whose definitions are recalled in Subsection \ref{subsec:Grigorchuk-def}).
Following \cite{Grigorchuk84}, for any such $\omega$, we associate
automorphisms $a,b_{\omega},c_{\omega},d_{\omega}$ of the rooted
binary tree. Consider the group $G_{\omega}$ generated by $S=\{a,b_{\omega},c_{\omega},d_{\omega}\}$.
For the string $\omega=({\bf 012})^{\infty}$, the corresponding group
is called the \emph{first Grigorchuk group}, which was introduced
in Grigorchuk \cite{Grigorchuk80}. By \cite{Grigorchuk84}, if $\omega$
is eventually constant then $G_{\omega}$ is virtually abelian (hence
of polynomial growth); otherwise $G_{\omega}$ is of intermediate
growth. The group $G_{\omega}$ is periodic (also called torsion)
if and only if $\omega$ contains all three letters ${\bf 0},{\bf 1},{\bf 2}$
infinitely often.

A long-standing open question, since the construction of Grigorchuk
groups of intermediate growth in \cite{Grigorchuk84}, is to find
the asymptotics of their growth functions. Among them the first Grigorchuk
group $G=G_{012}$ has been the most investigated. In Grigorchuk \cite{Grigorchuk84},
it was shown that 
\[
\exp(cn^{1/2})\le v_{G,S}(n)\le\exp\left(Cn^{\beta}\right),
\]
where $\beta=\log_{32}31$. The upper bound above which shows that
$G$ is of sub-exponential growth is based on norm contracting property.
The property considered in \cite{Grigorchuk84} is as follows. There
is an integer $d\ge2$ and an injective homomorphism $\psi=(\psi_{i})_{i=1}^{d}$
from a finite index subgroup $H$ of $G$ into $G^{d}$ such that
for any $h\in H$, with respect to word length $\left|\cdot\right|=l_{S}$
on $G$,
\begin{equation}
\sum_{i=1}^{d}\left|\psi_{i}(h)\right|\le\eta|h|+C\label{eq:contraction}
\end{equation}
for some constant $C>0$ and contraction coefficient $\eta\in(0,1)$.
Later Bartholdi in \cite{Bartholdi98} has shown that there exist
choices of norm that achieve better contraction coefficient than the
word length $l_{S}$, and therefore a better upper bound for growth;
in \cite{MuchnikPak} Muchnik and Pak show similar growth upper bounds
for more general $\omega$ including $({\bf 012})^{\infty}$ and for
further generalizations of $G_{\omega}$.  The upper bound in \cite{Bartholdi98,MuchnikPak}
states that
\[
v_{G,S}(n)\le\exp\left(Cn^{\alpha_{0}}\right),
\]
where $\alpha_{0}=\frac{\log2}{\log\lambda_{0}}\approx0.7674$, where
$\lambda_{0}$ is the positive root of the polynomial $X^{3}-X^{2}-2X-4$. 

The lower bound $e^{n^{1/2}}$ proved by Grigorchuk uses what he calls
anti-contraction properties on $G$. It remained open whether $v_{G,S}$
grows strictly faster than $e^{n^{1/2}}$ until Leonov \cite{Leonov}
introduced an algorithm which led to a better anti-contraction coefficient
and showed that $v_{G,S}(n)\ge\exp\left(cn^{0.504}\right)$. Bartholdi
\cite{Bartholdi01} obtained that $v_{G,S}(n)\ge\exp\left(cn^{0.5157}\right)$.
In the thesis of Brieussel \cite{brieussel}, Corollary 1.4.2 states
that $v_{G,S}(n)\ge\exp\left(cn^{0.5207}\right)$. 

Let $\mu$ be a probability measure on a group $\Gamma$. A function
$f:\Gamma\to\mathbb{R}$ is \emph{$\mu$-harmonic} if $f(x)=\sum_{y\in G}f(xy)\mu(y)$
for all $x\in\Gamma$. Among several definitions of the Poisson boundary
of $(\Gamma,\mu)$, also called the Poisson-Furstenberg boundary,
one is formulated in terms of the measure theoretic boundary that
represents all bounded $\mu$-harmonic functions on $\Gamma$, see
\cite{KV83}. In particular the Poisson boundary of $(\Gamma,\mu)$
is trivial if and only if all bounded $\mu$-harmonic functions on
$\Gamma$ are constant. The (Shannon) entropy of a measure $\mu$
is defined as $H(\mu)=-\sum_{g\in\Gamma}\mu(g)\log\mu(g)$. When $\Gamma$
is finitely generated, the entropy criterion of Kaimanovich-Vershik
\cite{KV83} and Derriennic \cite{derriennic} provides a quantitative
relation between the growth function of $\Gamma$ and the tail decay
of a measure $\mu$ of finite entropy and non-trivial Poisson boundary
on $\Gamma$, see more details reviewed in Subsection \ref{subsec:growthlemma}. 

In this paper we construct probability measures with non-trivial Poisson
boundary on torsion Grigorchuk groups. The measures are chosen to
have finite entropy and explicit control over their tail decay, that
is, with an upper bound on $\mu\left(\left\{ g:l_{S}(g)\ge n\right\} \right).$
The main contribution of the paper is that on a collection of Grigorchuk
groups, including these indexed by periodic strings containing three
letters ${\bf 0},{\bf 1},{\bf 2}$, we construct such measures with
good tail decay that provide near optimal volume lower estimates.
Our main result Theorem \ref{periodic1} specialized to the first
Grigorchuk group $G_{012}$ is the following.

\begin{theorem}\label{bound1st}

Let $\alpha_{0}=\frac{\log2}{\log\lambda_{0}}\approx0.7674$, where
$\lambda_{0}$ is the positive root of the polynomial $X^{3}-X^{2}-2X-4$.
For any $\epsilon>0$, there exists a constant $C_{\epsilon}>0$ and
a non-degenerate symmetric probability measure $\mu$ on $G=G_{012}$
of finite entropy and nontrivial Poisson boundary, where the tail
decay of $\mu$ satisfies that for all $r\ge1$,
\[
\mu\left(\left\{ g:l_{S}(g)\ge r\right\} \right)\le C_{\epsilon}r^{-\alpha_{0}+\epsilon}.
\]
As a consequence, for any $\epsilon>0$, there exists a constant $c_{\epsilon}>0$
such that for all $n\ge1$,
\[
v_{G,S}(n)\ge\exp\left(c_{\epsilon}n^{\alpha_{0}-\epsilon}\right).
\]

\end{theorem}

We construct measures of non-trivial Poisson boundary needed for the
proof of Theorem \ref{bound1st} by introducing and studying some
asymptotic invariants of the group, related to the diameter and shape
of what we call \textquotedbl quasi-cubic sets\textquotedbl{} and
\textquotedbl uniformised quasi-cubic sets\textquotedbl{} (see Subsection
\ref{subsec:sketch} where we sketch the proof). The main problem
we address is how to use the geometric and algebraic structure of
the groups to construct measures with good control over the tail decay
such that careful analysis of the induced random walks on the Schreier
graph can be performed to show non-triviality of the Poisson boundary. 

The volume lower bound in Theorem \ref{bound1st} matches up in exponent
with Bartholdi's upper bound in \cite{Bartholdi98}. In particular,
combined with the upper bound we conclude that the volume exponent
of the first Grigorchuk group $G$ exists and is equal to $\alpha_{0}$,
that is 
\[
\lim_{n\to\infty}\frac{\log\log v_{G,S}(n)}{\log n}=\alpha_{0}.
\]
It was open whether the limit exists. Our argument applies to a large
class of $\omega$ satisfying certain assumptions, see Theorem \ref{periodic1}.
In particular the assumption of the theorem is verified by all periodic
strings containing $3$ letters ${\bf 0},{\bf 1},{\bf 2}$, or strings
that are products of ${\bf 201}$ and ${\bf 211}.$ For general $\omega$,
volume upper bounds on $G_{\omega}$ are based on the norm contracting
property (\ref{eq:contraction}) first shown in Grigorchuk \cite{Grigorchuk84},
and suitable choices of norms are given in the work of Muchnik and
Pak \cite{MuchnikPak}, Bartholdi and the first named author \cite{BE2}.
Combining with upper bounds from \cite{BE2}, we derive the following
two corollaries from Theorem \ref{periodic1}. 

For a periodic string $\omega$ that contains all three letters ${\bf 0},{\bf 1},{\bf 2}$,
we show that the volume exponent of $G_{\omega}$ exists. Let $M_{i}$,
$i\in\{{\bf 0},{\bf 1},{\bf 2}\}$, be matrices defined as

\[
M_{{\bf 0}}=\left[\begin{array}{ccc}
2 & 0 & 1\\
0 & 2 & 1\\
0 & 0 & 1
\end{array}\right],M_{{\bf 1}}=\left[\begin{array}{ccc}
2 & 1 & 0\\
0 & 1 & 0\\
0 & 1 & 2
\end{array}\right],M_{{\bf 2}}=\left[\begin{array}{ccc}
1 & 0 & 0\\
1 & 2 & 0\\
1 & 0 & 2
\end{array}\right].
\]

\begin{theorem}\label{periodic}

Let $\omega\in\{{\bf 0},{\bf 1},{\bf 2}\}^{\infty}$ be a periodic
string $\omega=(\omega_{0}\ldots\omega_{q-1})^{\infty}$ of period
$q$ which contains all three letters ${\bf 0},{\bf 1},{\bf 2}$.
Then the growth exponent of the Grigorchuk group $G_{\omega}$ exists
and 
\[
\lim_{n\to\infty}\frac{\log\log v_{G_{\omega},S}(n)}{\log n}=\alpha_{\omega},
\]
where $\alpha_{\omega}=\frac{q\log2}{\log\lambda}$ and $\lambda=\lambda_{\omega_{0}\ldots\omega_{q-1}}$
is the spectral radius of the matrix $M_{\omega_{0}}\ldots M_{\omega_{q-1}}$. 

\end{theorem}

By \cite[Theorem 7.2]{Grigorchuk84} the family $\{G_{\omega}\}$
provides a continuum of mutually non-equivalent growth functions.
When $\omega$ is not periodic, Grigorchuk shows in \cite{Grigorchuk84}
that for some choices of $\omega$, the growth of $G_{\omega}$ exhibits
oscillating behavior: $\lim_{n\to\infty}\log\log v_{G_{\omega}}(n)/\log n$
does not exist. Together with upper bounds from \cite{BE2}, our lower
bounds show that given numbers $\alpha\le\beta$ in $[\alpha_{0},1]$,
there exists a $G_{\omega}$ which realizes these prescribed numbers
as lower and upper volume exponents.

\begin{theorem}\label{liminf}

For any $\alpha,\beta$ such that $\alpha_{0}\le\alpha\le\beta\le1$,
where $\alpha_{0}\approx0.7674$ is the volume exponent of $G_{012}$,
there exists $\omega\in\{{\bf 0},{\bf 1},{\bf 2}\}^{\infty}$ such
that the Grigorchuk group $G_{\omega}$ satisfies
\begin{align*}
\liminf_{n\to\infty}\frac{\log\log v_{G_{\omega},S}(n)}{\log n} & =\alpha,\\
\limsup_{n\to\infty}\frac{\log\log v_{G_{\omega},S}(n)}{\log n} & =\beta.
\end{align*}

\end{theorem}

Theorem \ref{liminf} is a special case of the more precise statement
in Theorem \ref{prescribe}. We mention that for some groups of intermediate
growth there is no way to obtain a near optimal lower bound of the
growth function using tail decay of measures of non-trivial Poisson
boundary. For instance, there are extensions of the first Grigorchuk
group that are of intermediate growth, where the growth function can
be arbitrarily close to exponential but any random walk with finite
$\alpha_{0}$-moment has trivial Poisson boundary, see Remark \ref{piecewise}
at the end of the paper. 

\subsection{Sketch of the arguments for the first Grigorchuk group\label{subsec:sketch}}

For a nilpotent group $\Gamma$ and any probability measure $\mu$
on $\Gamma$, it is known that the Poisson boundary of $(\Gamma,\mu)$
is trivial. For non-degenerate measure $\mu$ on a finitely generated
nilpotent group this result is due to Dynkin and Maljutov \cite{DynMal},
for a general case see Azencott \cite[Proposition 4.10]{azencott}.
If we assume that the growth function is sub-exponential, then as
a corollary of the entropy criterion, any finitely supported measure
$\mu$ on $\Gamma$ (more generally, any measure $\mu$ on $\Gamma$
of finite first moment $\sum_{g\in\Gamma}l_{S}(g)\mu(g)<\infty$)
has trivial Poisson boundary. If the growth of $\Gamma$ is bounded
by $v_{\Gamma,S}(n)\le\exp(Cn^{\alpha})$ for some $0<\alpha<1$,
then all measures with finite $\alpha$-moment, that is $\sum_{g\in\Gamma}l_{S}(g)^{\alpha}\mu(g)<\infty$,
have trivial Poisson boundary (see Corollary \ref{trivialmoment}).
Since the volume growth of the first Grigorchuk group $G=G_{012}$
satisfies $v_{G,S}(n)\le\exp(Cn^{\alpha_{0}})$ by \cite{Bartholdi98,MuchnikPak},
we have to look for measures with infinite $\alpha_{0}$-moment in
order to have non-trivial Poisson boundary. 

\begin{figure}\label{schreiergrig} \includegraphics[scale=0.8]{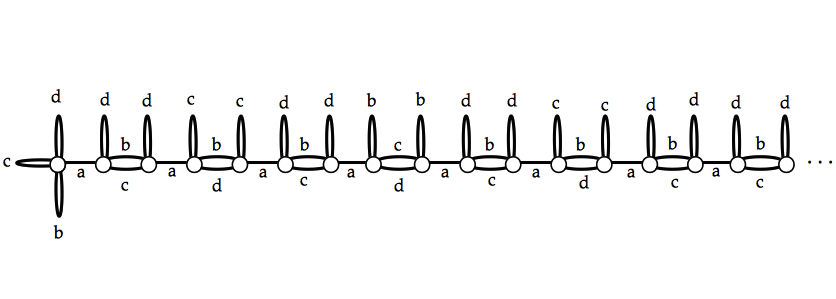} \caption{Orbital Schreier graph of the right most ray $o=1^{\infty}$ under the action of the first Grigorchuk group.} \end{figure}

The first Grigorchuk group $G$ is generated by four automorphisms
$a,b,c,d$ of the rooted binary tree (in Subsection \ref{subsec:Grigorchuk-def}
we recall the definition of these automorphisms, see also Figure \ref{bcd}).
The group $G$ acts on the boundary of the tree $\partial\mathsf{T}$
by homeomorphisms. The orbital Schreier graph $\mathcal{S}$ of $o$
is a half line as shown in Figure \ref{schreiergrig}, where $o=1^{\infty}$
is the right most ray of the tree.

In \cite{Erschler04} random walks with non-trivial Poisson boundary
are used to show that for $G_{\omega}$, where $\omega$ is not eventually
constant and contains only finite number of ${\bf 2}$ (for example
$\omega=({\bf 01})^{\infty}$), its volume growth satisfies
\[
v_{G_{\omega},S}(n)\ge\exp\left(c_{\epsilon}n/\log^{1+\epsilon}n\right).
\]
The measures with non-trivial Poisson boundary constructed in \cite{Erschler04}
come from transient symmetric random walks on $\mathbb{Z}$, where
it is well-known and easy to see what is the best tail condition on
such measures on $\mathbb{Z}$. The situation remained completely
unclear for other Grigorchuk groups with $\omega$ containing all
three letters infinitely often. Since for such $\omega$ the group
$G_{\omega}$ is torsion, the argument in \cite{Erschler04} relying
on an infinite order element cannot be applied. Here we need to introduce
a new construction of measures for the first Grigorchuk group. 

Consider the groupoid of germs of the action $G\curvearrowright\partial\mathsf{T}$.
The isotropy group $\mathcal{G}_{x}$ of the groupoid of germs of
$G$ at $x\in\partial\mathsf{T}$ is the quotient of the stabilizer
$G_{x}$ of $x$ by the subgroup of elements of $G$ that act trivially
in a neighborhood of $x$. For the first Grigorchuk group $G$, at
$x$ in the orbit $o\cdot G$, $\mathcal{G}_{x}$ is isomorphic to
$\left(\mathbb{Z}/2\mathbb{Z}\right)\times\left(\mathbb{Z}/2\mathbb{Z}\right)$
and its elements can be listed as $\{id,b,c,d\}$. See more details
in Section \ref{sec:Transience}. Given a group element $g$, we denote
its germ at $x$ by $\Phi_{g}(x)$, which is the germ of $g\tau^{-1}$
in the neighborhood of $x$ where $\tau$ is a finitary automorphism
of the tree such that $x\cdot g=x\cdot\tau$. This does not depend
on the choice of $\tau$, see precise definition in formula (\ref{eq:confiPhi})).
Our goal is to construct a non-degenerate probability measure $\mu$
on $G$ such that along the $\mu$-random walk trajectory $(W_{n})_{n=1}^{\infty}$,
the coset of $\Phi_{W_{n}}(o)$ in $\mathcal{G}_{o}/\left\langle b\right\rangle $
stabilizes with probability $1$. Such almost sure stabilization provides
non-trivial events in the tail $\sigma$-field of the random walk,
which implies the Poisson boundary of $(G,\mu)$ is non-trivial. 

Our first construction of a measure with controlled (but not near
optimal) tail decay and non-trivial Poisson boundary on $G$ is of
the form $\mu=\frac{1}{2}({\bf u}_{S}+\eta)$, where ${\bf u}_{S}$
is the uniform measure on $S=\{a,b,c,d\}$ and $\eta$ is a symmetric
measure supported on the subgroup ${\rm H}^{b}$ of $G$. Here ${\rm H}^{b}$
is the subgroup of $G$ that consists of $g\in G$ such that germs
$\Phi_{g}(x)$ are in $\left\langle b\right\rangle $ for all $x\in o\cdot G$.
For such a measure $\mu$, if the $\eta$-induced random walk on the
orbit of $o$ is transient, then the Poisson boundary of $\mu$ is
non-trivial, see Remark \ref{small}. In the situation of $\omega$
containing only ${\bf 0}$ and ${\bf 1}$ infinitely often, the subgroup
${\rm H}^{b}$ contains an infinite order element $ab$, thus in this
case one can apply well known results for random walks on $\left\langle ab\right\rangle =\mathbb{Z}$.
For the first Grigorchuk group, the subgroup ${\rm H}^{b}$ is locally
finite and the orbit of $o$ under ${\rm H}^{b}$ is not the full
orbit $o\cdot G$, but rather rays cofinal with $1^{\infty}$ such
that at locations $3k+1$, $k\ge1$, the digits are $1$ (see Lemma
\ref{Horbit}). 

Since the subgroup ${\rm H}^{b}$ is locally finite, unlike $\mathbb{Z}$
there is no \textquotedbl simple random walk\textquotedbl{} on it
that one can take convex combinations of convolution powers of the
corresponding measure to produce a measure with transient induced
random walk on the orbit of $o$. Rather, we build the measures from
a sequence of group elements satisfying an injectivity property for
their action on the orbit $o\cdot G$. For $g_{1},g_{2},\ldots\in G$,
we say the sequence satisfies the \emph{cube independence property}
on $o\cdot G$ if for any $n\ge1$ and $x\in o\cdot G$, the map from
the hypercube to the orbit $\{0,1\}^{n}\to o\cdot G$
\[
(\varepsilon_{1},\ldots\varepsilon_{n})\mapsto x\cdot g_{n}^{\varepsilon_{n}}\ldots g_{1}^{\varepsilon_{1}}
\]
is injective. We call the set $F_{n}=\left\{ g_{n}^{\varepsilon_{n}}\ldots g_{1}^{\varepsilon_{1}}:(\varepsilon_{1},\ldots\varepsilon_{n})\in\{0,1\}^{n}\right\} $
an $n$-\emph{quasi-cubic set} and the uniform measure ${\bf u}_{F_{n}}$
on $F_{n}$ an $n$-\emph{quasi-cubic measure}. Given such a sequence
$\left(g_{n}\right)$, we take convex combination over all $n$-quasi-cubic
measures and their inverses. The cube independence property guarantees
that for such measures, the induced random walks on the orbit of $o$
satisfy good isoperimetric inequalities. We prove a sufficient criterion
for such induced random walks to be transient in Proposition \ref{indep_tran}. 

\begin{figure}[!htb] \centering \includegraphics[scale=.6]{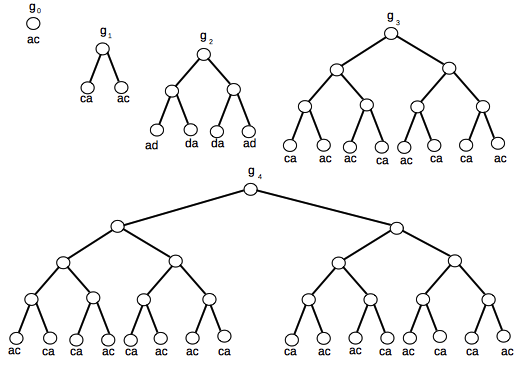}
\caption{Portrtraits of $g_j,\ j=0,\dots,4$ in the first Grigorchuk group.  More generally, words $\zeta^{n}(ax)$, where $x=c$ or $x=d$, represents an element in the stabilizer of level $n$,  and the labels (sections) at the vertices of level $n$  are $ax$ or $xa$.   }  
\label{fig:portaritzeta3AC} \end{figure} Our choice of elements satisfying the cube independence property on
$G$ is closely related to words that are known to have what is called
asymptotically maximal inverted orbit growth, see \cite{BE1}. Let
$\zeta$ be the substitution on $\{ab,ac,ad\}^{\ast}$ defined as
\begin{equation}
\zeta:ab\mapsto abadac,\ ac\mapsto abab,\ ad\mapsto acac.\label{eq:subzeta}
\end{equation}
Then $(\zeta^{n}(ac))_{n=1}^{\infty}$ achieves asymptotic maximal
inverted orbit growth, see the proof of \cite[Proposition 4.7]{BE1}.
With considerations of germs in mind, we take the sequence to be $g_{n}=\zeta^{n}(ac)$
if $n\equiv0,1\mod3$ and $g_{n}=\zeta^{n}(ad)$ if $n\equiv2\mod3$,
see Figure \ref{fig:portaritzeta3AC}. Then $g_{n}$ has $\left\langle b\right\rangle $-germs
for $n\equiv1,2\mod3$ and has $\left\langle c\right\rangle $-germs
for $n\equiv0\mod3$. For our construction it is important that this
sequence $\left(g_{n}\right)$ satisfies the cube independence property.
As an illustration of Proposition \ref{indep_tran}, in Subsection
\ref{subsec:critical} we find an optimal moment condition for measures
on $G$ which induce transient random walks on the orbit of $o=1^{\infty}$.
Such a measure can be taken to be of the form
\[
\eta_{0}=\sum_{n=1}^{\infty}c_{n}\left({\bf u}_{F_{n}}+{\bf u}_{F_{n}^{-1}}\right),
\]
where $c_{n}=\frac{C_{\epsilon}n^{1+\epsilon}}{2^{n}}$ , $F_{n}$
is the $n$-quasi-cubic set formed by $g_{1},\ldots,g_{n}$ and ${\bf u}_{F}$
is the uniform measure on $F$. 

In Section \ref{sec: growth-first} we explain how to use the subsequence
$g_{1},g_{2},g_{4},g_{5},g_{7},g_{8},\ldots$ of only $\left\langle b\right\rangle $-germs
and a related sequence that is more adapted to ${\rm H}^{b}$ to produce
finite entropy symmetric measures supported on ${\rm H}^{b}$ such
that the induced random walks on the orbit of $o$ are transient.
In particular it already provides a volume lower bound $\exp\left(cn^{0.57}\right)$
for $G$ better than previously known bounds, see Corollary \ref{Hbound}. 

However taking a measure supported on ${\rm H}^{b}$ is too restrictive,
see the remark at the end of Section \ref{sec: growth-first}. In
Section \ref{sec:Main} we carry out the main construction which removes
this restriction and produces measures with non-trivial Poisson boundary
and near optimal tail decay.

We now give an informal sketch of the construction. Return to the
\textquotedbl full\textquotedbl{} cube independent sequence $(g_{n})_{n=1}^{\infty}$.
The measure $\eta_{0}$ built from this sequence has near optimal
tail decay among measures that induce transient random walks on the
Schreier graph. However in the sequence we have germs of both types:
$g_{n}$ has $\left\langle b\right\rangle $-germs for $n\equiv1,2\mod3$
and has $\left\langle c\right\rangle $-germs for $n\equiv0\mod3$.
One can check that $\left\langle \gamma\right\rangle $-coset of $\Phi_{W_{n}}(o)$
does not stabilize for any $\gamma\in\{b,c,d\}$, where $W_{n}$ is
the random walk driven by $\frac{1}{2}({\bf u}_{S}+\eta_{0})$. With
the goal to achieve almost sure stabilization of $\left\langle b\right\rangle $-coset
of $\Phi_{W_{n}}(o)$ along the random walk trajectory, we want to
perform modifications to the elements used in the construction of
$\eta_{0}$ that prevent such stabilization, namely these $g_{n}$
with $\left\langle c\right\rangle $-germs.

It turns out for such considerations, transience of the induced random
walks on the Schreier graph is not sufficient. To observe almost sure
stabilization of $\left\langle b\right\rangle $-coset of $\Phi_{W_{n}}(o)$
for a $\mu$-random walk $W_{n}$, we need to show a weighted sum
of the Green function $\mathbf{G}_{P_{\mu}}$ over the Schreier graph
is finite, see the criterion in Proposition \ref{transience}:
\[
\sum_{x\in o\cdot G}\mathbf{G}_{P_{\mu}}(o,x)\mu\left(\left\{ g\in G:\ (g,x)\mbox{ is not a }\left\langle b\right\rangle \mbox{-germ}\right\} \right)<\infty.
\]
Thus the behavior of the Green function of the induced random walk
is essential. Informally we call a ray $x\in\partial\mathsf{T}$ that
carries a non-$\left\langle b\right\rangle $ germ a bad location.
Assume that the Green function $\mathbf{G}_{P_{\mu}}(o,x)$ decays
with distance $d_{\mathcal{S}}(o,x)$, then to control the effect
of $\left\langle c\right\rangle $-germ elements, it is to our advantage
to push the bad locations to be further away from the ray $o$. Note
that when we perform modifications to the measure, all terms in the
summation are affected: both the Green function of $P_{\mu}$ and
the $\mu$-weights on the bad locations. 

Recall that for $n$ divisible by $3$, $g_{n}=\zeta^{n}(ac)$, it
is in the $n$-level stabilizer and the section at any vertex $v$
in level $n$ is in $\{ac,ca\}$. Therefore $g_{n}$ has $2^{n}$
bad locations on the Schreier graph and the distance of these locations
to $o$ is bounded by $C2^{n}$. To move these locations further from
$o$, without causing disturbances globally, we rely on certain branching
properties of the group. Let $k_{n}$ be an integer divisible by $3$,
we introduce a conjugation of the generator $c$ that swaps on level
$k_{n}+2$ the section at $1^{k_{n}+2}$ with the section at its sibling.
Explicitly, take the element $h_{k_{n}}$ which is in the rigid stabilizer
of $1^{k_{n}}$ and acts as $baba$ in the subtree rooted at $1^{k_{n}}$.
Denote by $\mathfrak{c}_{k_{n}}$ the conjugate $\mathfrak{c}_{k_{n}}=h_{k_{n}}^{-1}ch_{k_{n}}$.
The portrait of $\mathfrak{c}_{k_{n}}$ is illustrated in Figure \ref{fig:portaitck}.
For $n$ and $k_{n}$ both divisible by $3$, take $\tilde{g}_{n}=\zeta^{n}\left(a\mathfrak{c}_{k_{n}}\right)$.
Then $\tilde{g}_{n}$ still has $2^{n}$ locations on the Schreier
graph such that $(\tilde{g}_{n},x)$ is a $c$-germ, but now the distance
of these locations to $o$ is comparable to $2^{n+k_{n}}$. 

\begin{figure}[!htb] \centering \includegraphics[scale=.6]{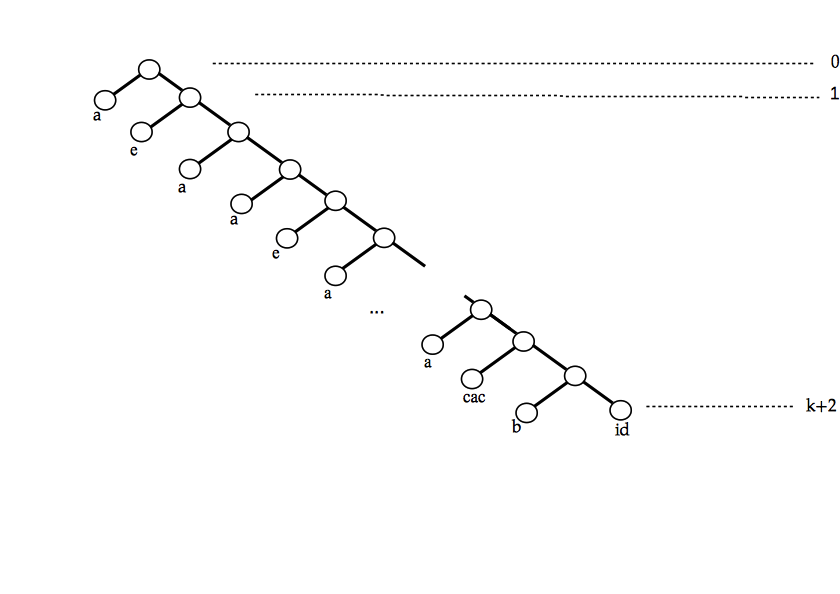} \caption{A portrtrait of $\mathfrak{c}_k$, for $k \equiv 0 \mod 3$. There are  $k+3$ levels of the tree shown on the picture. The portrait is the same as the portrait of $c$ up to level $k$, while at level $k+2$ left and right branches are permuted. } \label{fig:portaitck} \end{figure} 

Take a parameter sequence $(k_{n})$, in the cube independent sequence
$(g_{n})$, keep these $g_{n}$ such that $n\equiv1,2\mod3$ and replace
these $g_{n}$ with $n\equiv0\mod3$ by the modified element $\tilde{g}_{n}$.
This modified sequence satisfies the cube independence property, we
can form the $n$-quasi-cubic sets and define a measure $\eta$ of
the same form as $\eta_{0}$ using this sequence. It is possible that
with some choice like $c_{n}=2^{-n\beta}$, $\beta$ close to $1$
and $k_{n}=A\log n$, one can observe almost sure stabilization for
the resulting measure $\frac{1}{2}({\bf u}_{S}+\eta)$. However for
such a measure $\mu$ we didn't manage to obtain good enough estimates
on the induced Green function to verify the summability condition.
Here the issue is that because of the modifications, the induced transition
kernel on the Schreier graph becomes more singular (compared to the
one induced by $\eta_{0}$ ).

\begin{figure} \includegraphics[scale=1.0]{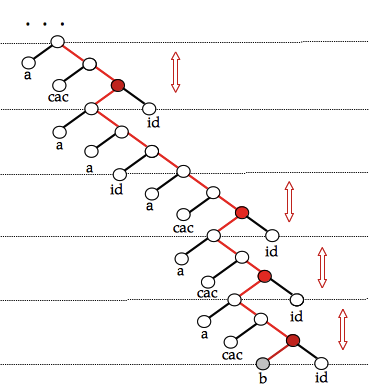} \caption{A conjugated element $\mathfrak{c}^{v}$,  in usual notation on the first Grigorchuk group. The levels of the form $3j-1$ are traced with dotted lines. There are two possiblities for the portraits between $3j-1$ and $3j+2$. Here the picture is shown for indexing vertex $v =11(11\ast)\dots(11\ast)(110)(111)(110)(110)110$.} \label{fig:big}\end{figure} 

To overcome this difficulty, we introduce an extra layer of randomness
to produce less singular induced transition kernels. We take a collection
of conjugates $\mathfrak{c}^{v}$ of $c$, indexed by vertices $v$
of the form $11\underbrace{11\ast11\ast\ldots11\ast}_{k_{n}\ {\rm digits}}110$,
where at $\ast$ the digit can be $0$ or $1$, see Figure \ref{fig:big}.
Given such a conjugate $\mathfrak{c}^{v}$, as before apply the substitution
$\zeta$ to obtain $\tilde{g}_{n}^{v}=\zeta^{n}\left(a\mathfrak{c}^{v}\right)$.
We call an $n$-quasi-cubic set formed by $\left(\tilde{g}_{j}^{v}\right)$
with each $v$ a random vertex chosen uniformly from the corresponding
indexing set a uniformised quasi-cubic set. The final measure we take
to prove Theorem \ref{bound1st} is of a similar form to $\eta$ discussed
in the previous paragraph, but involves what we call uniformised quasi-cubic
measures. For more detailed description of these measures see Example
\ref{firstmain}. We remark that the parameters $\left(k_{n}\right)$
need to be chosen carefully: the word length of element $\tilde{g}_{n}$
is approximately $2^{k_{n}}(2/\eta_{0})^{n}$, therefore if $k_{n}$
is too large, the tail decay wouldn't be near $n^{-\alpha_{0}}$;
on the other hand if $k_{n}$ is not sufficiently large then the effect
it introduces is not strong enough to achieve the goal of stabilization.
In the end we will take $k_{n}\approx A\log n$ with constant $A$
large enough. 

After the extra randomization, the resulting measure on $G$ induces
a transition kernel on the Schreier graph to which known heat kernel
estimates can be applied. More precisely, we apply techniques from
Barlow-Grigor'yan-Kumagai \cite{BGK} based on what is called the
Meyer's construction and the Davies method for off-diagonal heat kernel
estimates of jumping processes. In this way we obtain good enough
decay estimate for the induced Green function in terms of the Schreier
graph distance on the orbit of $o$. Together with explicit estimates
of the $\mu$-weights at bad locations we show the weighted sum of
Green function over the Schreier graph is finite with appropriate
choice of parameters. The construction provides finite entropy measures
on $G$ with near optimal tail decay and non-trivial Poisson boundary.

\subsection{Organization of the paper}

Section \ref{sec:Preliminaries} reviews background on Poisson boundary
and its connection to growth, groups acting on rooted trees and some
basic properties of Grigorchuk groups. In Section \ref{sec:Transience}
we formulate the criterion for stabilization of cosets of germ configurations.
Section \ref{sec:existence} provides quick examples of measures (without
explicit tail control) with non-trivial Poisson boundary on the first
Grigorchuk group. In Section \ref{sec:construction} we define the
cube independence property and show how to build explicit measures
with transient induced random walks from sequences satisfying this
property. Section \ref{sec: growth-first} discusses measures with
restricted germs on the first Grigorchuk group that induce transient
random walks on the orbit of $1^{\infty}$, in particular volume lower
bounds such measures provide and their limitations. Section \ref{sec:Main}
contains the main result of the paper: we construct finite entropy
measures on $G_{\omega}$ with near optimal tail decay and non-trivial
Poisson boundary, where $\omega$ satisfies what we call Assumption
$({\rm Fr}(D))$. In Section \ref{sec:maingrowth} we derive volume
lower estimates for $G_{\omega}$ from measures constructed in Section
\ref{sec:Main}. We close the paper with a few remarks and questions. 

\medskip{}

\emph{Conventions that are used throughout the paper:}
\begin{itemize}
\item We consider right group actions $\Gamma\curvearrowright X$ and write
the action of group element $g$ on $x\in X$ as $x\cdot g$. 
\item Given a probability measure $\mu$ on $\Gamma$, we say $\mu$ is
non-degenerate if its support generates $G$ as a semigroup. Let $X_{1},X_{2},\ldots$
be i.i.d. random variables with distribution $\mu$ on $\Gamma$ and
$W_{n}=X_{1}\ldots X_{n}$. We refer to $(W_{n})_{n=1}^{\infty}$
as a $\mu$-random walk.
\item On a regular rooted tree $\mathsf{T}$, given a vertex $u\in\mathsf{T}$
and $v\in\mathsf{T}\cup\partial\mathsf{T}$, $uv$ denotes the concatenation
of the strings, that is the ray with prefix $u$ followed by $v$.
Given $u,v\in\mathsf{T}\cup\partial\mathsf{T}$, $u\wedge v$ denotes
the longest common prefix of $u$ and $v$. 
\item Given a finite set $F$, ${\bf u}_{F}$ denotes the uniform distribution
on $F$. 
\item Let $v_{1},v_{2}$ two non-decreasing functions on $\mathbb{N}$ or
$\mathbb{R}$, $v_{1}\lesssim v_{2}$ if there is a constant $C>0$
such that $v_{1}(x)\le v_{2}(Cx)$ for all $x$ in the domain. Let
$\simeq$ be the equivalence relation that $v_{1}\simeq v_{2}$ if
$v_{1}\lesssim v_{2}$ and $v_{2}\lesssim v_{1}$. Note that for finitely
generated infinite groups, the $\simeq$ equivalence class of the
growth function does not depend on the choice of generating set. 
\end{itemize}

\subsection*{Acknowledgment }

We are grateful to the anonymous referee whose comments and suggestions
improved the exposition of the paper. We thank J\'er\'emie Brieussel
for helpful comments on the preliminary version of the paper.

\section{Preliminaries\label{sec:Preliminaries}}

\subsection{Boundary behavior of random walks and connection to growth\label{subsec:growthlemma}}

We recall the notion of the invariant $\sigma$-field and the tail
$\sigma$-field of a random walk. Let $\mu$ be a probability measure
on a countable group $\Gamma$. For $x\in\Gamma$, let $\mathbb{P}_{x}$
be the law of the random walk trajectory $\left(W_{0},W_{1},\ldots\right)$
starting at $W_{0}=x$. Endow the space of infinite trajectories $G^{\mathbb{N}}$
with the usual Borel $\sigma$-field $\mathcal{A}$. The shift $\mathfrak{s}:\Gamma^{\mathbb{N}}\to\Gamma^{\mathbb{N}}$
acts by $\mathfrak{s}(g_{1},g_{2},\ldots)=(g_{2},g_{3},\ldots)$.
The \emph{invariant $\sigma$-field} $\mathcal{I}$ is defined as
$\mathcal{I}=\{A\in\mathcal{A:\ \mathfrak{s}}^{-1}(A)=A\}$, it is
also referred to as the \emph{stationary $\sigma$-field}. There is
a correspondence between $\mathcal{I}$ and the space of bounded $\mu$-harmonic
functions on $\Gamma$. In particular, the invariant $\sigma$-field
$\mathcal{I}$ is trivial (i.e. any event $A\in\mathcal{I}$ satisfies
$\mathbb{P}_{x}(A)=0$ or $1$) if and only if all bounded $\mu$-harmonic
functions are constant, see \cite{KV83}. 

The intersection $\mathcal{A}_{\infty}=\cap_{n=0}^{\infty}\mathfrak{s}^{-n}\mathcal{A}$
is called the \emph{tail $\sigma$-field} or \emph{exit $\sigma$-field}
of the random walk. When $\mu$ is aperiodic, the $\mathbb{P}_{x}$-completion
of $\mathcal{A}_{\infty}$ and the $\mathbb{P}_{x}$-completion of
$\mathcal{I}$ coincide. Therefore for $\mu$ aperiodic, all bounded
$\mu$-harmonic functions are constant if and only if the tail $\sigma$-field
$\mathcal{A}_{\infty}$ is $\mathbb{P}_{x}$-trivial for some (hence
all ) $x\in\Gamma$.

The (Shannon) \emph{entropy} of a probability measure $\mu$ on $\Gamma$
is 
\[
H(\mu)=-\sum_{g\in\Gamma}\mu(g)\log\mu(g).
\]
In the case that the measure $\mu$ has finite entropy $H(\mu)<\infty$,
the (Avez) \emph{asymptotic entropy} of $\mu$ is defined as 
\[
\mathbf{h}_{\mu}=\lim_{n\to\infty}\frac{H\left(\mu^{(n)}\right)}{n}.
\]
The entropy criterion of Kaimanovich-Vershik \cite{KV83} and Derriennic
\cite{derriennic}, states that for $\mu$ of finite entropy, all
bounded $\mu$-harmonic functions are constant if and only if the
Avez asymptotic entropy $\mathbf{h}_{\mu}=0$. 

Given a probability measure of finite entropy and non-trivial Poisson
boundary on a finitely generated group $\Gamma$, by the entropy criterion
and the Shannon theorem, one can derive lower estimate for the volume
growth of $\Gamma$ from the tail decay of the measure. The following
lemma extends \cite[Lemma 6.1]{Erschler04}. 

\begin{lem}\label{moment_growth}

Let $\Gamma$ be a finitely generated group and $\left|\cdot\right|=l_{S}$
a word distance on $\Gamma$. Suppose $\Gamma$ admits a probability
measure $\mu$ of finite entropy such that the Poisson boundary $(\Gamma,\mu)$
is nontrivial. Let 
\[
\varrho_{n}=\inf\left\{ r>0:\ \mu\left(\left\{ g:|g|\ge r\right\} \right)<1/n\right\} \mbox{ and }\ \phi(r)=\sum_{g:|g|\le r}|g|\mu(g).
\]
Then there exists a constant $c=c(\Gamma,\mu)>0$ such that 
\[
v_{\Gamma,S}(R_{n})\ge\exp\left(cn\right)\ \mbox{where }R_{n}=n\phi(\varrho_{n}).
\]

\end{lem}

\begin{proof}

If the measure $\mu$ has finite first moment, in other words the
function $\phi(r)$ is bounded, then $\Gamma$ has exponential growth.
In this case the claim is true. In what follows we assume that $\phi(r)\to\infty$
as $r\to\infty$. 

Let $\mu_{\varrho}$ be the truncation of measure $\mu$ at radius
$\varrho$, that is $\mu_{\varrho}(g)=\mu(g)\mathbf{1}_{\{|g|\le\varrho\}}$.
Because of the triangle inequality $|gh|\le|g|+|h|$, we have 
\[
\sum_{g}|g|\mu_{\varrho}^{(n)}(g)\le n\sum_{g}|g|\mu_{\varrho}(g)=n\phi(\varrho).
\]
It follows from Markov inequality that 
\[
\mu_{\varrho}^{(n)}(\{g:\ |g|\ge R\})\le\frac{n\phi(\varrho)}{R}.
\]
For all $n\ge2$ , the total mass of the convolution power $\mu_{\varrho_{n}}^{(n)}$
satisfies
\[
\sum_{g}\mu_{\varrho_{n}}^{(n)}(g)=\left(1-\mu(\{g:|g|>\varrho_{n}\})\right)^{n}\ge(1-1/n)^{n}\ge\frac{1}{4}.
\]
Take $R_{n}=8n\phi(\varrho_{n})$, then 
\[
\mu^{(n)}(B(e,R_{n}))\ge\mu_{\varrho_{n}}^{(n)}(B(e,R_{n}))\ge\frac{1}{4}-\frac{n\phi(\varrho_{n})}{R_{n}}=\frac{1}{8}.
\]
Denote by $\mathbf{h}$ the asymptotic entropy of $\mu$. By the entropy
criterion, non-triviality of Poisson boundary of $(G,\mu)$ is equivalent
to $\mathbf{h}>0$. By Shannon's theorem, for any $\epsilon>0$, there
exists a constant $N_{\epsilon}$ such that for any $n>N_{\epsilon}$,
there is a finite set $V_{n}\subset G$ such that $\mu^{(n)}(V_{n})\ge1-\epsilon$
and 
\[
e^{-n(\mathbf{h}+\epsilon)}\le\mu^{(n)}(x)\le e^{-n(\mathbf{h}-\epsilon)}\mbox{ for }x\in V_{n}.
\]
Consider the intersection $V_{n}'=\{g:\ |g|\le R_{n}\}\cap V_{n}$,
we have 
\[
\frac{1}{8}-\epsilon\le\mu^{(n)}(V_{n}')\le|V_{n}'|\max_{g\in V_{n}'}\mu^{(n)}(g)\le|V_{n}'|e^{-n(\mathbf{h}-\epsilon)}.
\]
Choose $\epsilon>0$ to be a constant less than $\frac{1}{2}\min\{1/8,\mathbf{h}\}$.
We obtain that for all $n$ sufficiently large,
\[
v_{\Gamma,S}(R_{n})\ge|V_{n}'|\ge\frac{1}{16}e^{\frac{1}{2}n\mathbf{h}}.
\]

\end{proof}

We record the following two corollaries of Lemma \ref{moment_growth}.

\begin{cor}\label{tailgrowth}

Suppose that $\Gamma$ admits a probability measure $\mu$ of finite
entropy with nontrivial Poisson boundary. Suppose $\mu$ satisfies
\[
\mu\left(\{g:|g|\ge r\}\right)\le\frac{C}{r^{\alpha}\ell(r)},
\]
where $\alpha\in(0,1)$ and $\ell$ is a slowly varying function such
that $\ell(x)\sim\ell(x^{b})$ as $x\to\infty$ for some $b>0$. Then
there exists a constant $c=c(\Gamma,\mu)>0$ such that
\[
v_{\Gamma,S}(n)\ge\exp(cn^{\alpha}\ell(n)).
\]

\end{cor}

\begin{proof}Let $f$ be the asymptotic inverse of $r^{\alpha}\ell(r)$,
then $f(r)\sim(r\ell^{\#}(r))^{1/\alpha}$, where $\ell^{\#}$ is
the de Bruijn conjugate of $\ell,$ see \cite[Proposition 1.5.15.]{RV}.
Let 
\[
\tilde{\phi}(r)=\int_{1}^{r}\frac{dx}{x^{\alpha}\ell(x)}.
\]
Then $\tilde{\phi}(r)\sim\frac{1}{1-\alpha}r^{1-\alpha}/\ell(r)$,
see \cite[Proposition 1.5.8.]{RV}. It follows that 
\[
R_{n}=4n\phi(\varrho_{n})\le C_{1}n\tilde{\phi}(f(n))\le C_{2}n^{1/\alpha}\ell^{\#}(n)^{-1+1/\alpha}/\ell(n^{1/\alpha}\ell^{\#}(n)^{1/\alpha}).
\]
Under the assumption that $\ell(x)\sim\ell(x^{b})$ for some $b>0$,
we have $\ell^{\#}(x)\sim1/\ell(x)$ and the bound simplifies to 
\[
R_{n}\le C(n/\ell(n))^{1/\alpha}.
\]
By Lemma \ref{moment_growth}, there exists constants $C,c>0$ which
depend on $\mu$ such that 
\[
v_{\Gamma,S}\left(C(n/\ell(n))^{1/\alpha}\right)\ge\exp(cn).
\]
Since the asymptotic inverse of $(x/\ell(x))^{1/\alpha}$ is $x^{\alpha}\ell(x)$
in this case, we conclude that there exists a constant $c=c(\Gamma,\mu)>0$
such that 
\[
v_{\Gamma,S}(n)\ge\exp(cn^{\alpha}\ell(n)).
\]

\end{proof}

\begin{cor}\label{trivialmoment}

Let $\Gamma$ be a finitely generated group such that $v_{\Gamma}(r)\lesssim\exp(r^{\alpha})$,
where $\alpha\in(0,1)$. Then any probability measure $\mu$ on $\Gamma$
with finite $\alpha$-moment
\[
\sum_{g\in\Gamma}|g|^{\alpha}\mu(g)<\infty
\]
has finite entropy and trivial Poisson boundary. 

\end{cor}

\begin{proof}

We first show that $\mu$ has finite entropy. Write $M=\sum_{g\in\Gamma}|g|^{\alpha}\mu(g)$.
Denote by $\mu_{k}$ the restriction of $\mu$ to $B\left(e,2^{k}\right)\setminus B(e,2^{k-1})$,
normalized to be a probability measure, that is $\mu_{k}(g)=\frac{1}{\lambda_{k}}\mu(g)\mathbf{1}_{B\left(e,2^{k}\right)\setminus B(e,2^{k-1})}$,
where $\lambda_{k}=\mu\left(B\left(e,2^{k}\right)\setminus B(e,2^{k-1}\right)$.
It follows that $\mu=\sum_{k=0}^{\infty}\lambda_{k}\mu_{k}$. Recall
the fact that the entropy of any probability measure supported by
a finite set is maximal for the uniform measure. Let $X_{1}$ be a
random variable with distribution $\mu$ and $Z=\sum_{k=1}^{\infty}k\mathbf{1}_{B\left(e,2^{k}\right)\setminus B(e,2^{k-1})}(X_{1})$.
Then 
\begin{align*}
H(\mu) & \le\sum_{k=0}^{\infty}\lambda_{k}H(\mu_{k})+\sum_{k=0}^{\infty}-\lambda_{k}\log\lambda_{k}\\
 & \le\sum_{k=0}^{\infty}\lambda_{k}\log\left|B(e,2^{k})\right|+\sum_{k=0}^{\infty}-\lambda_{k}\log\lambda_{k}.
\end{align*}
Now we plug in the assumption on the volume function of $\Gamma$
that $\left|B(e,2^{k})\right|\le\exp(Cr^{\alpha})$ for some constant
$C$, then 
\[
\sum_{k=0}^{\infty}\lambda_{k}\log\left|B(e,2^{k})\right|\le C\sum_{k=0}^{\infty}\lambda_{k}2^{k\alpha}\le2^{\alpha}C\sum_{g\in\Gamma}|g|^{\alpha}\mu(g)=2^{\alpha}CM.
\]
By the Markov inequality we have $\lambda_{k}\le\frac{M}{2^{(k-1)\alpha}}$.
Then for the second term, by the inequality $-a\log a\le2e^{-1}\sqrt{a}$,
we have 
\[
\sum_{k=0}^{\infty}-\lambda_{k}\log\lambda_{k}\le\sum_{k=0}^{\infty}2e^{-1}\sqrt{\lambda_{k}}\le2e^{-1}\sqrt{M}\sum_{k=0}^{\infty}2^{-\frac{(k-1)\alpha}{2}}<\infty.
\]
We conclude that $\mu$ has finite entropy. 

Suppose on the contrary $\mu$ has non-trivial Poisson boundary. Since
$\mu$ has finite $\alpha$-moment, we have 
\begin{align*}
\varrho_{n} & =\inf\{r:\ \mu(\{g:|g|\ge r\})\le1/n\}\le n^{1/\alpha},\\
\phi(r) & =\sum_{g:|g|\le r}|g|\mu(g)=o(r^{1-\alpha})\mbox{ as }r\to\infty.
\end{align*}
Therefore by Corollary \ref{tailgrowth}, we have that $v_{\Gamma}(R_{n})\gtrsim e^{n}$
where $R_{n}/n^{1/\alpha}\to0$ as $n\to\infty$. This contradicts
with the assumption that $v_{\Gamma}(r)\lesssim\exp(r^{\alpha})$. 

\end{proof}

\subsection{Permutational wreath products}

Let $\Gamma$ be a group acting on a set $X$ from the right. Let
$A$ be a group, the \emph{permutational wreath product} $W=A\wr_{X}\Gamma$
is defined as the semi-direct product $\sum_{X}A\rtimes\Gamma$, where
$\Gamma$ acts on $\sum_{X}A$ by permuting coordinates. The \emph{support
${\rm supp}f$ }of $f:X\to A$ consists of points of $x\in X$ such
that $f(x)\neq id_{A}$. Elements of $\sum_{X}A$ are viewed as finitely
supported functions $X\to A$. The action of $\Gamma$ on $\sum_{X}A$
is 
\[
\tau_{g}f(x)=f(x\cdot g).
\]
It follows that ${\rm supp}(\tau_{g}f)=({\rm supp}f)\cdot g^{-1}$.
Elements in $W=A\wr_{X}\Gamma$ are recorded as pairs $(f,g)$ where
$f\in\sum_{X}A$ and $g\in\Gamma$. Multiplication in $W$ is given
by 
\[
(f_{1},g_{1})(f_{2},g_{2})=(f_{1}\tau_{g_{1}}f_{2},g_{1}g_{2}).
\]
When $\Gamma$ and $A$ are finitely generated, $W=A\wr_{X}\Gamma$
is finitely generated as well. For more information about Cayley graphs
of permutational wreath products, see \cite[Section 2]{BE1}. 

\subsection{Groups acting on rooted trees }

We review some basics for groups acting on rooted trees. For more
background and details, we refer to \cite[Part 1]{BGShandbook} or
\cite[Chapter 1]{NekraBook}. Let $\mathbf{d}=(d_{j})_{j\in\mathbb{N}}$
be a sequence of integers, $d_{j}\ge2$ for all $j\in\mathbb{N}$. 

The spherically symmetric rooted tree $\mathsf{T}_{\mathbf{d}}$ is
the tree with vertices $v=v_{1}\ldots v_{n}$ with each $v_{j}\in\{0,1,\ldots,d_{j}-1\}$.
The root is denoted by the empty sequence $\emptyset$. Edge set of
the tree is $\left\{ \left(v_{1}\ldots v_{n},v_{1}\ldots v_{n}v_{n+1}\right)\right\} $.
The index $n$ is called the depth or level of $v$, denoted $\left|v\right|=n$.
Denote by $\mathsf{T}_{\mathbf{d}}^{n}$ the finite subtree of vertices
up to depth $n$ and $\mathsf{L}_{n}$ the vertices of level $n$.
The boundary $\partial\mathsf{T}_{\mathbf{d}}$ of the tree $\mathsf{T}_{\mathbf{d}}$
is the set of infinite rays $x=v_{1}v_{2}\ldots$ with $v_{j}\in\{0,1,\ldots,d_{j}-1\}$
for each $j\in\mathbb{N}$. 

The group of automorphisms ${\rm {Aut}}(\mathsf{T}_{\mathbf{d}})$
of the rooted tree is the group of tree automorphisms that fix the
root $\emptyset$. The group ${\rm {Aut}}(\mathsf{T}_{\mathbf{d}})$
admits the following canonical isomorphism 
\[
\psi:{\rm Aut}(\mathsf{T}_{\mathbf{d}})\to{\rm Aut}(\mathsf{T}_{\mathbf{\mathfrak{s}d}})\wr_{\{0,1,\ldots,d_{1}-1\}}\mathfrak{S}_{d_{1}},
\]
where $\mathfrak{S}_{d}$ is the permutation group on the set $\{0,1,\ldots,d-1\}$,
and $\mathfrak{s}$ is the shift on the sequences, $\mathfrak{s}(d_{1},d_{2},\ldots)=(d_{2},d_{3},\ldots)$.
We refer to the isomorphism $\psi$ as the \emph{wreath recursion}.
As the isomorphism is canonical, we omit $\psi$ and identify an automorphism
and its image under $\psi$: for $g\in{\rm Aut}(\mathsf{T}_{\mathbf{d}})$,
write 
\[
g=\left(g_{0},\ldots,g_{d_{1}-1}\right)\tau_{1}(g)
\]
where the sections $g_{j}\in{\rm Aut}(\mathsf{T}_{\mathfrak{s}\mathbf{d}})$
and $\tau_{1}(g)\in S_{d_{1}}$. The section $g_{j}$ represents the
action of $g$ on the subtree $\mathsf{T}_{j}$ rooted at the vertex
$j$. The permutation $\tau_{1}(g)$ describes how the subtrees rooted
at $j$, $j\in\{0,1,\ldots,d_{1}-1\}$ are permuted. We refer to $\tau_{1}(g)$
as the \emph{root permutation} of $g$. 

Apply the wreath recursion $n$ times, we have 
\begin{align*}
\psi^{n}:{\rm Aut}(\mathsf{T}_{d}) & \to{\rm Aut}(\mathsf{T}_{\mathfrak{s}^{n}{\bf d}})\wr_{\mathsf{L}_{n}}{\rm Aut}(\mathsf{T}_{\mathbf{d}}^{n}),\\
g & \mapsto(g_{v})_{v\in\mathsf{L}_{n}}\tau_{n}(g),
\end{align*}
where finitary permutation $\tau_{n}(g)$ is the projection of $g$
to the automorphism group of the finite tree $\mathsf{T}_{\mathbf{d}}^{n}$,
and $g_{v}$ is called the \emph{section} of $g$ at vertex $v$.
Since the isomorphism $\psi^{n}$ is canonical, we write $g=(g_{v})_{v\in\mathsf{L}_{n}}\tau_{n}(g)$. 

Let $G$ be a subgroup of ${\rm {Aut}}(\mathsf{T}_{\mathbf{d}})$.
Given a vertex $u\in\mathsf{T}_{\mathbf{d}}$, the \emph{vertex stabilizer}
${\rm St}_{G}(u)$ consists of these elements in $G$ that fix $u$,
that is ${\rm St}_{G}(u)=\{g\in G:\ u\cdot g=u\}$. The $n$\emph{-th
level stabilizer} ${\rm St}_{G}(\mathsf{L}_{n})$ consists of the
automorphisms in $G$ that fix all the vertices on level $n$:
\[
{\rm St}_{G}(\mathsf{L}_{n}):=\cap_{v\in\mathsf{L}_{n}}{\rm St}_{G}(v).
\]

The \emph{rigid vertex stabilizer} of $u$ in $G$, denoted by ${\rm Rist}_{G}(u)$,
consists of the automorphisms in $G$ that fix all vertices not having
$u$ as a prefix. An element $g\in{\rm Rist}_{G}(u)$ only acts on
the subtree rooted at $u$: if $|u|=n$, then under the wreath recursion
$\psi^{n}$, the finitary permutation $\tau_{n}(g)$ is equal to identity,
and the section $g_{v}$ is equal to identity for any $v\neq u$.
Since automorphisms in different rigid vertex stabilizers on the same
level act on disjoint subtrees, they commute with other other. The
rigid $n$-th level stabilizer is defined as 
\[
{\rm Rist}_{G}(\mathsf{L}_{n}):=\prod_{u\in\mathsf{L}_{n}}{\rm Rist}_{G}(u).
\]

We now recall the definitions of branch groups. Let $\mathsf{T}$
be a rooted tree of constant valency $d$ and $G$ be a subgroup of
${\rm {Aut}}(\mathsf{T})$ acting level transitively. 
\begin{enumerate}
\item If all rigid $n$-th level stabilizers ${\rm Rist}_{G}(\mathsf{L}_{n})$,
$n\in\mathbb{N}$, have finite index in $G$, then we say $G$ is
a \emph{branch group}. 
\item We say $G$ is a \emph{regular branch group} if it acts on a regular
tree $\mathsf{T}$ of constant valency $d$, and there exists a finite
index subgroup $K<G$ such that $K^{d}$ is contained in $\psi(K)$
as a finite index subgroup, that is $(\underbrace{K,\ldots,K}_{d\ {\rm copies}})$
is a finite index subgroup of $\psi(K)$. In this case we say $G$
is branching over $K$. 
\end{enumerate}
The group of automorphisms of the subtree rooted at $u$ is embedded
in ${\rm Aut}(\mathsf{T}_{d})$ as rigid vertex stabilizer of $u$
by 
\begin{align}
\iota(\cdot:u):{\rm Aut}(\mathsf{T}_{\mathfrak{s}^{n}{\bf d}}^{u}) & \to{\rm Aut}(\mathsf{T}_{\mathfrak{s}^{n}{\bf d}})\wr_{\mathsf{L}_{n}}{\rm Aut}(\mathsf{T}_{\mathbf{d}}^{n})\label{eq:embed-rigid}\\
h & \mapsto(g_{v}),\ \mbox{where }g_{u}=h\mbox{ and }g_{v}=e\mbox{ for any }v\neq u.\nonumber 
\end{align}
It follows from the definition that if $G$ is branching over $K$,
then for any vertex $u\in\mathsf{T}$, ${\rm Rist}_{G}(u)$ contains
$K$ as a subgroup. In particular, we have that $\iota(h,u)\in K$
for any $h\in K$, $u\in\mathsf{T}$. 

\subsection{Grigorchuk groups\label{subsec:Grigorchuk-def}}

The groups $\left\{ G_{\omega}\right\} $ were introduced by Grigorchuk
in \cite{Grigorchuk84} as first examples of groups of intermediate
growth. We recall below the definition of these groups. See also exposition
in the books \cite[Chapter VIII]{delaHarpeBook} and \cite{Mann}.

Let $\{{\bf 0},{\bf 1},{\bf 2}\}$ be the three non-trivial homomorphisms
from the $4$-group $(\mathbb{Z}/2\mathbb{Z})\times(\mathbb{Z}/2\mathbb{Z})=\{id,b,c,d\}$
to the $2$-group $\mathbb{Z}/2\mathbb{Z}=\{id,a\}$. They are ordered
in such a way that ${\bf 0},{\bf 1},{\bf 2}$ vanish on $d,c,b$ respectively.
For example, ${\bf 0}$ maps $id,d$ to $id$ and $b,c$ to $a$.
Let $\Omega=\left\{ {\bf 0},{\bf 1},{\bf 2}\right\} ^{\mathbb{\infty}}$
be the space of infinite sequences over letters $\left\{ {\bf 0},{\bf 1},{\bf 2}\right\} $.
The space $\Omega$ is endowed with the shift map $\mathfrak{s}:\Omega\to\Omega$,
$\mathfrak{s}(\omega_{0}\omega_{1}\ldots)=\omega_{1}\omega_{2}\ldots$.

Given an $\omega\in\Omega$, the Grigorchuk group $G_{\omega}$ acting
on the rooted binary tree is generated by $\left\{ a,b_{\omega},c_{\omega},d_{\omega}\right\} $,
where $a=(id,id)\varepsilon$, $\varepsilon$ transposes $0$ and
$1$, and the automorphisms $b_{\omega},c_{\omega},d_{\omega}$ are
defined recursively according to $\omega$ as follows. The wreath
recursion sends
\[
\psi_{n}:G_{\mathfrak{s}^{n}\omega}\to G_{\mathfrak{s}^{n+1}\omega}\wr_{\{0,1\}}\mathfrak{S}_{2}
\]
by
\begin{align*}
\psi_{n}\left(b_{\mathfrak{s}^{n}\omega}\right) & =\left(\omega_{n}(b),b_{\mathfrak{s}^{n+1}\omega}\right),\\
\psi_{n}\left(c_{\mathfrak{s}^{n}\omega}\right) & =\left(\omega_{n}(c),c_{\mathfrak{s}^{n+1}\omega}\right),\\
\psi_{n}\left(d_{\mathfrak{s}^{n}\omega}\right) & =\left(\omega_{n}(d),d_{\mathfrak{s}^{n+1}\omega}\right).
\end{align*}
The string $\omega$ determines the portrait of the automorphisms
$b_{\omega},c_{\omega},d_{\omega}$ by the recursive definition above.
The \emph{first Grigorchuk group} corresponds to the periodic sequence
$\omega=({\bf 012})^{\infty}$ and is often denoted as $G_{012}$.
See Figure \ref{bcd} for the portraits of these tree automorphisms.
The construction of $G_{012}$ is given by Grigorchuk in \cite{Grigorchuk80},
where he shows that $G_{012}$ is a torsion group. Merzlyakov has
shown later in \cite{Merzlyakov} that the construction of $G_{012}$
is closely related to an earlier construction of Aleshin \cite{Aleshin}.
Indeed it can be shown that $G_{012}$ and Aleshin's group in \cite{Aleshin}
are commensurable. 

\begin{figure}[!htb] \centering \includegraphics[scale=0.6]{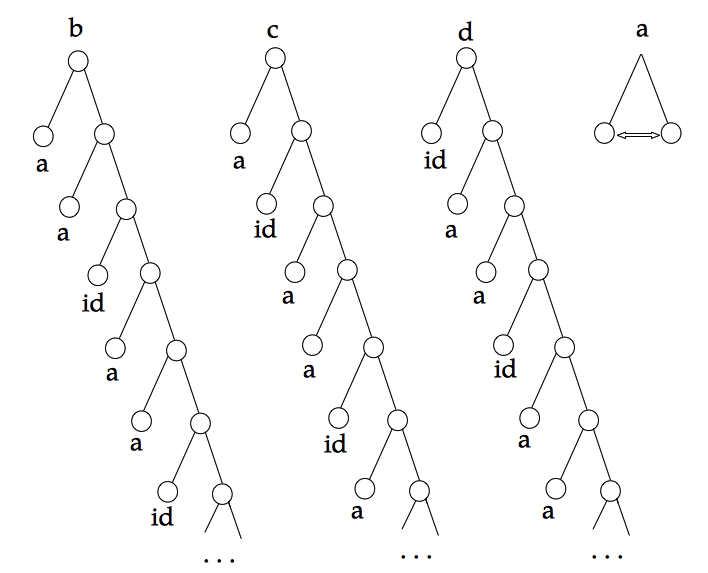} \caption{Portraits of generators of the first Grigorchuk group.}\label{bcd}
\end{figure} 

Main interest in Grigorchuk's construction results from his paper
\cite{Grigorchuk84} where it is shown that $G_{012}$, and more generally
$G_{\omega}$ with $\omega$ not eventually constant, is of intermediate
growth. If $\omega$ is eventually constant, then $G_{\omega}$ has
polynomial growth. When $\omega$ contains infinitely many of each
of the three symbols $\{{\bf 0},{\bf 1},{\bf 2}\}$, $G_{\omega}$
is an infinite torsion group. 

The following substitutions are used to obtain words with asymptotically
maximal inverted orbit growth in \cite[Section 3.2]{BE2}. Define
self-substitutions $\zeta_{i}$, $i\in\{{\bf 0},{\bf 1},{\bf 2}\}$,
of $\{ab,ac,ad\}$ by
\begin{align}
\zeta_{{\bf 0}}:ab\mapsto abab, & \ ac\mapsto acac,\ ad\mapsto abadac,\label{eq:subzetai}\\
\zeta_{{\bf 1}}:ab\mapsto abab,\  & ac\mapsto adacab,\ ad\mapsto adad,\nonumber \\
\zeta_{{\bf 2}}:ab\mapsto acabad,\  & ac\mapsto acac,\ ad\mapsto adad.\nonumber 
\end{align}
It is understood that $\zeta_{\omega_{n}}$ sends $\left\{ ab_{\mathfrak{s}^{n+1}\omega},ac_{\mathfrak{s}^{n+1}\omega},ad_{\mathfrak{s}^{n+1}\omega}\right\} ^{\ast}$
to $\left\{ ab_{\mathfrak{s}^{n}\omega},ac_{\mathfrak{s}^{n}\omega},ad_{\mathfrak{s}^{n}\omega}\right\} ^{\ast}$
by the rules specified above, for example if $\omega_{n}={\bf 0}$,
then $\zeta_{\omega_{n}}\left(ab_{\mathfrak{s}^{n+1}\omega}\right)=ab_{\mathfrak{s}^{n}\omega}ab_{\mathfrak{s}^{n}\omega}$.
One step wreath recursion gives that for $w\in\left\{ ab,ac,ad\right\} ^{\ast}$,
\begin{equation}
\psi_{n}(\zeta_{\omega_{n}}\left(w\right))=\begin{cases}
(awa,w) & \mbox{if }\zeta_{\omega_{n}}(w)\mbox{ contains an even number of }a,\\
(wa,aw)\varepsilon & \mbox{if }\zeta_{\omega_{n}}(w)\mbox{ contains an odd number of }a.
\end{cases}\label{eq:subformula}
\end{equation}
This formula is used in \cite[Section 3.2]{BE2} (where there is a
misprint), see also lecture notes of Bartholdi \cite[Section F]{bartholdinote}.
It follows from (\ref{eq:subformula}) that if $g\in G_{\mathfrak{s}^{n+1}\omega}$
can be represented by a word in $\{ab_{\mathfrak{s}^{n+1}\omega},ac_{\mathfrak{s}^{n+1}\omega},ad_{\mathfrak{s}^{n+1}\omega}\}$,
then the substitution $\zeta_{\omega_{n}}$ can be applied to $g$
and the resulting image in $G_{\mathfrak{s}^{n}\omega}$ does not
depend on the choice of the representing word. The composition $\zeta_{\omega_{0}}\circ\ldots\circ\zeta_{\omega_{n}}$
maps such $g\in G_{\mathfrak{s}^{n+1}\omega}$ into $G_{\omega}$. 

Usually, the recursion on the first Grigorchuk group $G_{012}$ is
recorded as 
\begin{equation}
a=(id,id)\varepsilon,\ b=(a,c),\ c=(a,d),\ d=(id,b).\label{eq:un}
\end{equation}
Compared with the general notation for $G_{\omega}$, the following
identification is made: $\omega=({\bf 012})^{\infty}$ and
\begin{align}
b & =b_{\omega}=c_{\mathfrak{s^{2}\omega}}=d_{\mathfrak{s}\omega},\label{eq:identification}\\
c & =b_{\mathfrak{s}\omega}=c_{\omega}=d_{\mathfrak{s}^{2}\omega},\nonumber \\
d & =b_{\mathfrak{s}^{2}\omega}=c_{\mathfrak{s}\omega}=d_{\omega}.\nonumber 
\end{align}
Note that under this identification, when applying the substitutions
$\zeta_{i}$ along the string $\omega=({\bf 012})^{\infty}$, we have
the substitution $\zeta$ mentioned in the Introduction:
\[
\zeta:ab\mapsto abadac,\ ac\mapsto abab,\ ad\mapsto acac.
\]
For example, for $n\equiv0\mod3$, $\zeta_{\omega_{n-1}}=\zeta_{{\bf 2}}$
and $\zeta_{\omega_{n-1}}\left(ab_{\mathfrak{s}^{n}\omega}\right)=ac_{\mathfrak{s}^{n-1}\omega}ab_{\mathfrak{s}^{n-1}\omega}ad_{\mathfrak{s}^{n-1}\omega}$,
$\zeta_{\omega_{n-1}}\left(ac_{\mathfrak{s}^{n}\omega}\right)=ac_{\mathfrak{s}^{n-1}\omega}ac_{\mathfrak{s}^{n-1}\omega}$,
$\zeta_{\omega_{n-1}}\left(ad_{\mathfrak{s}^{n}\omega}\right)=ad_{\mathfrak{s}^{n-1}\omega}ad_{\mathfrak{s}^{n-1}\omega}$
in general notation. By the identification explained in (\ref{eq:identification}),
$\zeta_{\omega_{n-1}}$ can be written as the substitution $\zeta$. 

We mention another useful substitution on $G_{012}$: the substitution
$\sigma$ on $\{a,b,c,d\}^{\ast}$ such that
\begin{equation}
\sigma(a)=aca,\ \sigma(b)=d,\ \sigma(c)=b,\ \sigma(d)=c.\label{eq:subsigma}
\end{equation}
It gives an homomorphism from $G_{012}$ to itself. The substitution
$\sigma$ is used in the proof of the anti-contraction property which
leads to the lower bound $e^{n^{1/2}}$ by Grigorchuk \cite{Grigorchuk84}.
It also appears in Lysionok's recursive presentation of the group
$G_{012}$ in \cite{Lysionok}. 

In later sections where only the first Grigorchuk group is discussed
we will adopt the usual notation as in (\ref{eq:un}). One should
keep in mind the identification above to avoid possible confusions
with the general notation for $G_{\omega}$. 

\subsection{Sections along a ray\label{subsec:portrait}}

A tree automorphism can be described by its \emph{portrait}, see for
example \cite[Section 1.4]{BGShandbook}. For our purposes, it is
useful to describe an automorphism by its sections along a ray (finite
or infinite) which it fixes. 

Let $\mathsf{T}$ be the rooted binary tree. Let $g\in{\rm Aut}(\mathsf{T})$
and suppose $v\in\mathsf{T}$ is a vertex such that $v\cdot g=v$.
Write $v=v_{1}v_{2}\ldots v_{n}$. Since $v$ is fixed by $g$, $g$
is completely described by its sections at $\check{v}_{1},v_{1}\check{v}_{2},\ldots,v_{1}\ldots v_{n-1}\check{v}_{n}$
and $v$, where $\check{v}_{j}$ denotes the opposite of $v_{j}$,
$\check{v}_{j}=1-v_{j}$. More formally, under the wreath recursion,
$g=(g_{0},g_{1})$ (there is no root permutation because $v$ is fixed
by $g$). Then record $g_{\check{v}_{1}}$ and continue to perform
wreath recursion to $g_{v_{1}}$. Record the section $g_{v_{1}\check{v}_{2}}$
and continue with $g_{v_{1}v_{2}}$. The procedure stops when we reach
level $n$ and record the sections at $v$ and its sibling. 

For an automorphism $g$ which fixes an infinite ray $v=v_{1}v_{2}\ldots\in\partial\mathsf{T}$,
$g$ is described by the collection of sections $\left\{ g_{v_{1}\ldots\check{v}_{n}}\right\} _{n\in\mathbb{N}}$.
When all the sections $g_{v_{1}\ldots\check{v}_{n}}\in\mathfrak{S}_{2}$,
following the terminology of \cite[Definition 1.24]{BGShandbook},
we say that the automorphism $g$ is \emph{directed along the infinite
ray} $v$, . For example, in the Grigorchuk group $G_{\omega}$, the
generators $b_{\omega},c_{\omega},d_{\omega}$ are directed along
the ray $1^{\infty}$. 

When $v\cdot g=v$ and all the sections $g_{v_{1}\ldots\check{v}_{n}}$
are represented by short words from an explicit collection, we draw
the picture of these sections along $v$ to describe the element $g$.
Although this picture is not a portrait in the strict sense, it effectively
describes how $g$ acts on the tree.

\begin{figure} \centering \includegraphics[scale=0.05]{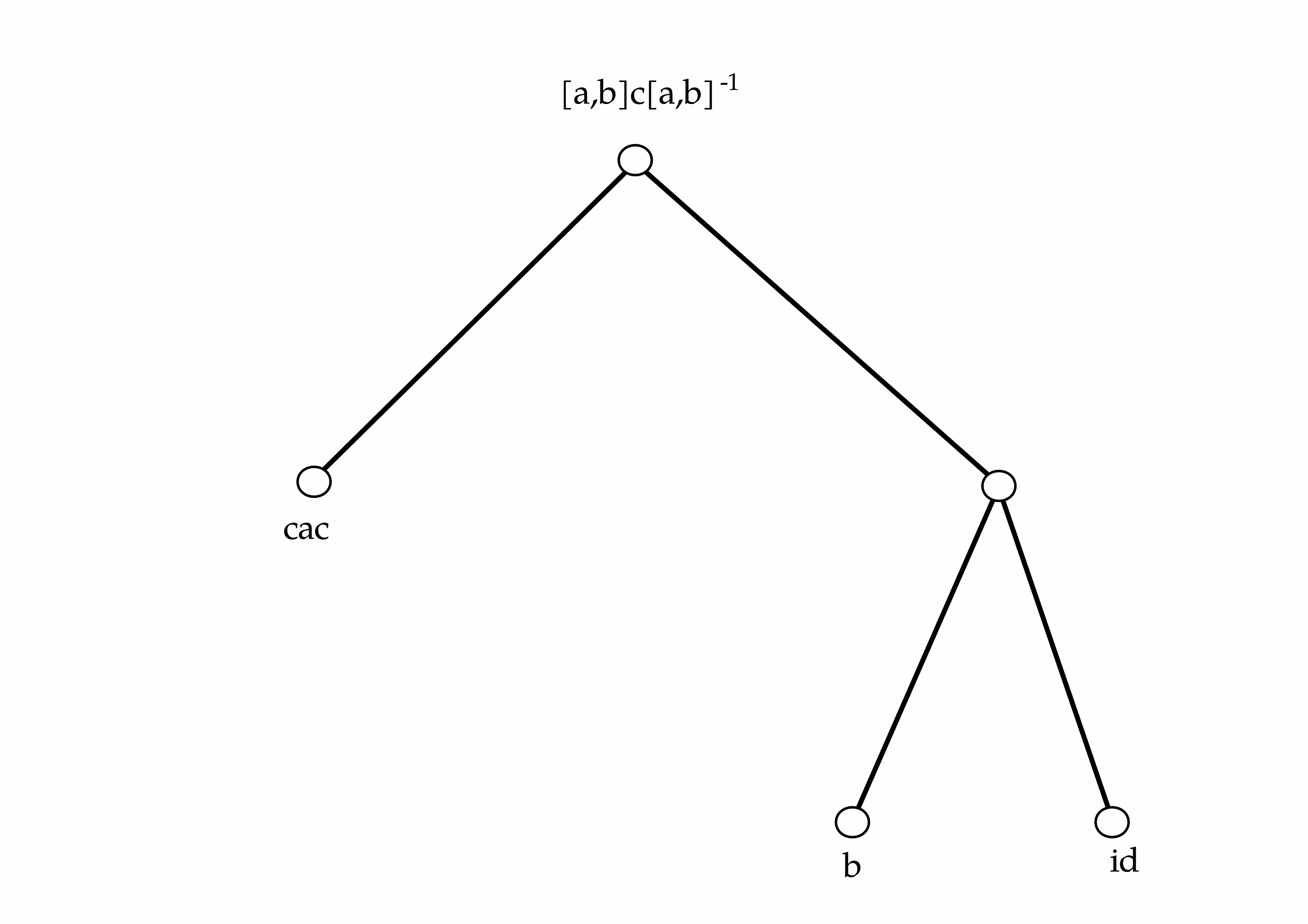} \caption{A portrait of $[a,b]c[a,b]^{-1}$ along $10$}\label{absmall}
\end{figure} 

\begin{exa}

Take in the first Grigorchuk group the element $g=[a,b]c[a,b]^{-1}$.
Then the vertex $10$ is fixed by $g$. The portrait of $g$ along
$10$ is given by $g_{0}=cac$, $g_{11}=id$ and $g_{10}=b$. See
Figure \ref{absmall}. 

\end{exa}

\section{Stabilization of germ configurations \label{sec:Transience}}

The main purpose of this section is to formulate and prove Proposition
\ref{transience} which provides a sufficient condition for stabilization
of cosets of germ configurations. 

We recall the notion of germs of homeomorphisms. Let $\mathcal{X}$
be a topological space. A \emph{germ} of homeomorphism of $\mathcal{X}$
is an equivalence class of pairs $(g,x)$ where $x\in\mathcal{X}$
and $g$ is a homeomorphism between a neighborhood of $x$ and a neighborhood
of $g(x)$; and two germs $(g_{1},x_{1})$ and $(g_{2},x_{2})$ are
equal if $x_{1}=x_{2}$ and $g_{1},g_{2}$ coincide on a neighborhood
of $x_{1}$. A composition $(g_{1},x_{1})(g_{2},x_{2})$ is defined
if and only if $x_{1}\cdot g_{1}=x_{2}$. The inverse of the a germ
$(g,x)$ is $(g,x)^{-1}=(g^{-1},x\cdot g)$. A groupoid of germs of
homeomorphisms on $\mathcal{X}$ is a set of germs of homeomorphisms
of $\mathcal{X}$ that is closed under composition and inverse, and
that contains all identity germs $(id_{\mathcal{X}},x)$, $x\in\mathcal{X}$.

Let $G$ be a group acting by homeomorphisms on $\mathcal{X}$ from
the right. Its \emph{groupoid of germs}, denoted by $\mathcal{G}$,
is the set of germs $\left\{ (g,x):\ g\in G,x\in\mathcal{X}\right\} $.
For $x\in\mathcal{X}$, the isotropy group of $\mathcal{G}$ at $x$,
denoted by $\mathcal{G}_{x}$, is the set of germs $\left\{ (g,x):\ g\in{\rm St}_{G}(x)\right\} $.
In other words, $\mathcal{G}_{x}$ is the quotient of the stabilizer
${\rm St}_{G}(x)$ by the subgroup of $G$ which consists of elements
acting trivially on a neighborhood of $x$. 

Suppose there is a group $L$ acting by homeomorphisms on $\mathcal{X}$
such that for any point $x\in\mathcal{X}$, the orbits $x\cdot L=x\cdot G$
and the isotropy group $\mathcal{L}_{x}$ of its groupoid $\mathcal{L}$
is trivial. We refer to such an $L$ as an \emph{auxiliary group with
trivial isotropy}. Write conjugation of a germ $(g,x)\in\mathcal{G}_{x}$
by $\sigma\in L$ as
\[
(g,x)^{\sigma}:=\left(\sigma^{-1}g\sigma,x\cdot\sigma\right)\in\mathcal{G}_{x\cdot\sigma}.
\]
It's easy to see that since $L$ has trivial isotropy groups, if $x\cdot\sigma_{1}=x\cdot\sigma_{2}$
for $\sigma_{1},\sigma_{2}\in L$, then $(g,x)^{\sigma_{1}}=(g,x)^{\sigma_{2}}$. 

With the auxiliary group $L$ chosen, let $\hat{\mathcal{G}}$ be
the groupoid of germs of the group $\left\langle G,L\right\rangle $.
The isotropy group $\hat{\mathcal{G}}_{x}$ is called the group of
germs in \cite{Erschler04}. 

\begin{notation}\label{Hsubgroupoid}

Let $G\curvearrowright\mathcal{X}$ by homeomorphisms and $L$ be
an auxiliary group with trivial isotropy. Suppose the isotropy group
$\mathcal{\hat{G}}_{o}$ of $\hat{\mathcal{G}}$ is non-trivial at
some point $o\in\mathcal{X}$. Let $H_{o}\lvertneqq\mathcal{\hat{G}}_{o}$
be a proper subgroup of $\mathcal{\hat{G}}_{o}$ . For each point
$x\in o\cdot G$, fix a choice of $\sigma_{x}\in L$ such that $o\cdot\sigma_{x}=x$.
Let
\[
H_{x}:=\left\{ (g,x)\in\mathcal{\hat{G}}_{x}:\ (\sigma_{x}g\sigma_{x}^{-1},o)\in H_{0}\right\} ,
\]
then $H_{x}$ is a proper subgroup of $\hat{\mathcal{G}}_{x}$. Let
$\mathcal{H}=\mathcal{H}(H_{o})$ be the following sub-groupoid of
$\mathcal{G}$:
\begin{equation}
\mathcal{H}:=\{(g,x):\ g\in G,\ x\in o\cdot G,\ (g\sigma^{-1},x)\in H_{x}\mbox{ where }\sigma\in L,\ x\cdot g=x\cdot\sigma\}.\label{eq:H_v}
\end{equation}

\end{notation}

We now discuss some examples of groupoid of germs of groups acting
on rooted trees. Let $G<{\rm Aut}(\mathsf{T}_{\mathbf{d}})$, then
$G$ acts on the boundary of the tree $\mathcal{X}=\partial\mathsf{T}_{\mathbf{d}}$
by homeomorphisms. Let $L$ be the subgroup of ${\rm Aut}(\mathsf{T}_{\mathbf{d}})$
that consists of all finitary automorphisms. In other words, an element
$g\in{\rm Aut}(\mathsf{T}_{\mathbf{d}})$ is in $L$ if there exists
a finite level $n$ such that all sections $g_{v}$ for $v\in\mathsf{L}_{n}$
are trivial. The group $L$ is locally finite. For $G<{\rm Aut}(\mathsf{T}_{\mathbf{d}})$,
we use the group $L_{G}=\left\{ \gamma\in L:\ x\cdot\gamma\in x\cdot G\mbox{ for all }x\in\partial\mathsf{T}\right\} $
as an auxiliary group with trivial isotropy groups. In particular
if $G$ acts level transitively, then $L$ is used as the auxiliary
group. 

The isotropy groups of the groupoid of germs are easy to recognize
in directed groups (called spinal groups in \cite[Chapter 2]{BGShandbook}). 

\begin{exa}\label{ggerms}

Let $G_{\omega}$ be a Grigorchuk group and $o=1^{\infty}$. For any
$\omega$, $G_{\omega}$ acts level transitively. Use the finitary
automorphisms $L$ as the auxiliary group. From the definition of
the group we have that if $x\notin o\cdot G_{\omega}$, that is $x$
is not cofinal with $1^{\infty}$, then the isotropy group $(\mathcal{G}_{\omega})_{x}$
is trivial.

For $\omega$ eventually constant, say eventually constant ${\bf 0}$,
we have that $(d_{\omega},o)=(id,o)$ and $(b_{\omega},o)=(c_{\omega},o)$
in $(\mathcal{G}_{\omega})_{o}$. In this case for $x$ cofinal with
$1^{\infty}$, $(\mathcal{G}_{\omega})_{x}\simeq\mathbb{Z}/2\mathbb{Z}$.
For $\omega$ not eventually constant then $\left(\hat{\mathcal{G}}_{\omega}\right)_{x}=(\mathcal{G}_{\omega})_{x}=\{id,b,c,d\}$
for $x$ cofinal with $1^{\infty}$. In this case a group element
$g$ has $\gamma$-germ at $x$, where $\gamma\in\{b,c,d\}$, if there
is a finite level $n$ such that the section of $g$ at $x_{1}\ldots x_{n}$
is $\gamma_{\mathfrak{s}^{n}\omega}$. 

When $\omega$ is not eventually constant, the subgroup $\left\langle b\right\rangle =\{id,b\}$
is a proper subgroup of $\hat{\mathcal{G}}_{o}$. We refer to the
corresponding sub-groupoid $\mathcal{H}^{b}=\mathcal{H}(\left\langle b\right\rangle )$
of $\mathcal{G}$ as the groupoid of $\left\langle b\right\rangle $-germs.
The groupoid of $\left\langle c\right\rangle $-germs ($\left\langle d\right\rangle $-germs
resp.) is defined in the same way from the subgroup $\{id,c\}$ ($\{id,d\}$
resp.) of $\hat{\mathcal{G}}_{o}$. 

\end{exa}

Proposition \ref{transience} below provides a sufficient condition
for stabilization of $\mathcal{H}$-cosets of germs based on the Green
function of the induced random walk on the orbit. This criterion is
applied to verify non-triviality of Poisson boundary for the measures
with good control over tail decay constructed in Section \ref{sec:Main}
on Grigorchuk groups. Given a probability measure $\mu$ on $G$,
its induced transition kernel $P_{\mu}$ on an orbit $o\cdot G$ is
given by
\[
P_{\mu}(x,y)=\sum_{g\in G}\mu(g)\mathbf{1}_{\{y=x\cdot g\}}.
\]
The Green function $\mathbf{G}_{P_{\mu}}(x,y)$ of the $P_{\mu}$-random
walk is 
\[
\mathbf{G}_{P_{\mu}}(x,y)=\sum_{n=0}^{\infty}P_{\mu}^{n}(x,y).
\]
We refer to the book \cite[Chapter I]{WoessBook} for general background
on random walks on graphs. In many situations, estimates on transition
probabilities of the induced $P_{\mu}$-random walk are available
while the $\mu$-random walk on the group $G$ is hard to understand. 

\begin{prop}\label{transience}

Let $G$ be a countable group acting by homeomorphisms on $\mathcal{X}$
from the right and $L$ be an auxiliary group with trivial isotropy.
Assume that the isotropy group $\mathcal{\hat{G}}_{o}$ is nontrivial
at some point $o\in\mathcal{X}$ and $\hat{\mathcal{G}}_{o}=\mathcal{G}_{o}$. 

Let $\mu$ be a non-degenerate probability measure on $G$. Let $P_{\mu}$
be the induced transition kernel on the orbit $o\cdot G$ and $\mathcal{\mathbf{G}}_{P_{\mu}}$
the Green function of the $P_{\mu}$-random walk. Suppose there exists
a proper subgroup $H_{o}\lvertneqq\mathcal{\hat{G}}_{o}$ such that
\begin{equation}
\sum_{x\in o\cdot G}\mathbf{G}_{P_{\mu}}(o,x)\mu\left(\left\{ g\in G:\ (g,x)\notin\mathcal{H}\right\} \right)<\infty,\label{eq:summable}
\end{equation}
where $\mathcal{H}=\mathcal{H}(H_{0})$ is the sub-groupoid of $\mathcal{G}$
associated with $H_{o}$ defined in (\ref{eq:H_v}). Then the Poisson
boundary of $(G,\mu)$ is non-trivial. 

\end{prop}

\begin{rem}\label{small}

Suppose $\mu$ have the following additional property: there exists
a finite subset $V\subseteq o\cdot G$ such that for any element $g$
in the support of $\mu$, $(g,x)\in\mathcal{H}$ for all $x\notin V$,
then the condition in Proposition \ref{transience} is equivalent
to $G_{P_{\mu}}(o,o)<\infty$. In other words, the induced $P_{\mu}$-random
walk is transient on the orbit $o\cdot G$. In this special case the
claim is given by \cite[Proposition 2]{Erschler04}. For example,
it can be applied to measures of the form $\mu=\frac{1}{2}(\nu+\eta)$,
where $\nu$ is of finite support and for any element $g$ in the
support of $\eta$, $(g,x)\in\mathcal{H}$ for all $x\in o\cdot G$.
In this case we say $\eta$ is supported on a subgroup of restricted
germs.

\end{rem}

It is natural to consider $G$ as an embedded subgroup of the (unrestricted)
permutational wreath product $\mathcal{\hat{G}}_{o}\wr\wr_{o\cdot G}G$
as follows. 

Denote by $\mathcal{W}$ the semi-direct product $\left(\prod_{x\in o\cdot G}\mathcal{\hat{G}}_{x}\right)\rtimes G$.
Record elements of $\mathcal{W}$ as pairs $\left(\Phi,g\right)$,
where $g\in G$ and $\Phi$ assigns each point $x\in o\cdot G$ an
element in the isotropy group $\mathcal{\hat{G}}_{x}$ at $x$. Denote
by ${\rm supp}\Phi$ the set of points $x$ where $\Phi(x)$ is not
equal to identity of $\mathcal{G}_{x}$. The action of $G$ on $\prod_{x\in o\cdot G}\mathcal{\hat{G}}_{x}$
is given by 
\[
(\tau_{g}\Phi)(x)=\sigma\Phi(x\cdot g)\sigma^{-1},
\]
where $\sigma\in L$ satisfies $x\cdot\sigma=x\cdot g$. One readily
checks that $\tau$ is a well-defined (doesn't depend on the choice
of $\sigma$) left action of $G$ on $\prod_{x\in o\cdot G}\mathcal{\hat{G}}_{x}$.
Multiplication in $\mathcal{W}$ is given by $(\Phi,g)(\Phi',g')=\left(\Phi\tau_{g}\Phi',gg'\right)$.
The semi-direct product $\mathcal{W}$ is isomorphic to the permutational
wreath product $\mathcal{\hat{G}}_{x}\wr\wr_{o\cdot G}G$. We now
describe an embedding $\vartheta:G\hookrightarrow\mathcal{W}$. 

\begin{fact}\label{mapgerm}

Let $\vartheta:G\to\mathcal{W}$ by defined as $\vartheta(g)=(\Phi_{g},g)$
such that for $x\in o\cdot G$,
\begin{equation}
\Phi_{g}(x)=(g\sigma^{-1},x)\in\mathcal{G}_{x},\ \mbox{where }\sigma\in L,\ x\cdot g=x\cdot\sigma.\label{eq:confiPhi}
\end{equation}
Then $\vartheta$ is a monomorphism.

\end{fact}

\begin{proof}

Let $g_{1},g_{2}\in G$ and $\sigma_{1},\sigma_{2}\in L$ such that
$x\cdot g_{1}=x\cdot\sigma_{1}$, $x\cdot g_{1}g_{2}=x\cdot g_{1}\sigma_{2}$.
Then
\begin{align*}
\Phi_{g_{1}}(x)\left(\tau_{g_{1}}\Phi_{g_{2}}\right)(x) & =(g_{1}\sigma_{1}^{-1},x)\left(g_{2}\sigma_{2}^{-1},x\cdot g_{1}\right)^{\sigma_{1}^{-1}}\\
 & =(g_{1}\sigma_{1}^{-1},x)(\sigma_{1}g_{2}\sigma_{2}^{-1}\sigma_{1}^{-1},x\cdot g_{1}\sigma_{1}^{-1})\\
 & =(g_{1}g_{2}\sigma_{2}^{-1}\sigma_{1}^{-1},x)=\Phi_{g_{1}g_{2}}(x).
\end{align*}
It follows that $\vartheta$ is homomorphism. It is clearly injective.

\end{proof}

\begin{proof}[Proof of Proposition \ref{transience}]

Let $X_{1},X_{2}\ldots$ be i.i.d. random variables on $G$, $W_{n}=X_{1}\ldots X_{n}$.
Let $\vartheta:G\to\mathcal{W}$, $\vartheta(g)=(\Phi_{g},g)$, be
the embedding as described in Fact \ref{mapgerm}. Along the random
walk trajectory $W_{n}$, consider the $H_{o}$-coset of the germ
at $o$. Denote by $\pi\left(\Phi_{W_{n}}(o)\right)$ the (right)
coset of $\Phi_{W_{n}}(o)$ in $\mathcal{G}_{o}/H_{o}$. By the rule
of multiplication in $\mathcal{W}$,
\[
\Phi_{W_{n+1}}(o)=\Phi_{W_{n}}(o)\left(\Phi_{X_{n+1}}\left(o\cdot W_{n}\right)\right)^{\sigma_{n}^{-1}},\ \sigma_{n}\in L,\ o\cdot W_{n}=o\cdot\sigma_{n}.
\]
Therefore the event
\[
\left\{ \pi\left(\Phi_{W_{n}}(o)\right)\neq\pi\left(\Phi_{W_{n+1}}(o)\right)\right\} =\left\{ \Phi_{X_{n+1}}\left(o\cdot W_{n}\right)\notin H_{o\cdot W_{n}}\right\} .
\]
By definition of the sub-groupoid of germs $\mathcal{H}$, 
\[
\left\{ \Phi_{X_{n+1}}\left(o\cdot W_{n}\right)\notin H_{o\cdot W_{n}}\right\} =\left\{ (X_{n+1},o\cdot W_{n})\notin\mathcal{H}\right\} .
\]
 Then
\begin{align*}
\sum_{n=0}^{\infty}\mathbb{P}\left(\pi\left(\Phi_{W_{n}}(o)\right)\neq\pi\left(\Phi_{W_{n+1}}(o)\right)\right) & =\sum_{n=0}^{\infty}\mathbb{P}\left((X_{n+1},o\cdot W_{n})\notin\mathcal{H}\right)\\
 & =\sum_{n=0}^{\infty}\sum_{x\in o\cdot G}\sum_{g\in G}\mu(g)\mathbf{1}_{\{(g,x)\notin\mathcal{H}\}}\mathbb{P}\left(x=o\cdot W_{n}\right)\\
 & =\sum_{x\in o\cdot G}{\bf G}_{P_{\mu}}(o,x)\mu\left(\left\{ g:\ (g,x)\notin\mathcal{H}\right\} \right).
\end{align*}
If the summation is finite, by the Borel-Cantelli lemma, $\pi\left(\Phi_{W_{n}}(o)\right)$
stabilizes a.s. along the random walk trajectory $\left(W_{n}\right)$.
For each coset $\gamma H_{o}$ in $\hat{\mathcal{G}}_{o}/H_{o}$,
consider the tail event 
\[
A_{\gamma H_{o}}=\left\{ \lim_{n\to\infty}\pi\left(\Phi_{W_{n}}(o)\right)=\gamma H_{o}\right\} .
\]
Given $\gamma\in\hat{\mathcal{G}}_{o}=\mathcal{G}_{o}$, let $g\in G$
be a group element such that $o\cdot g=o$ and $\Phi_{g}(o)=\gamma$.
Note that $\mathbb{P}_{id}\left(A_{H_{o}}\right)=\mathbb{P}_{g}\left(A_{\gamma H_{o}}\right)$
and for any $m\in\mathbb{N}$,
\begin{align*}
\mathbb{P}_{id}\left(A_{H_{o}}\right)=\mathbb{P}_{g}\left(A_{\gamma H_{o}}\right) & \ge\mu^{(m)}(g^{-1})\mathbb{P}_{id}\left(A_{\gamma H_{o}}\right),\\
\mathbb{P}_{id}\left(A_{\gamma H_{o}}\right) & \ge\mu^{(m)}(g)\mathbb{P}_{g}\left(A_{\gamma H_{o}}\right)=\mu^{(m)}(g)\mathbb{P}_{id}\left(A_{H_{o}}\right).
\end{align*}
Since $\mu$ is non-degenerate, there exists a finite $m$ such that
$\mu^{(m)}(g),\mu^{(m)}(g^{-1})>0$. Then $\mathbb{P}_{id}\left(A_{H_{o}}\right)=0$
would imply $\mathbb{P}_{id}\left(A_{\gamma H_{o}}\right)=0$ for
all $\gamma\in\mathcal{G}_{o}$. On the other hand since $\pi\left(\Phi_{W_{n}}(o)\right)$
stabilizes with probability $1$, we have $\mathbb{P}_{id}\left(\cup_{\gamma H_{o}\in\mathcal{G}_{o}/H_{o}}A_{\gamma H_{o}}\right)=1$.
It follows that $\mathbb{P}_{id}\left(A_{H_{o}}\right)>0$. Since
$\mathbb{P}_{id}\left(A_{\gamma H_{o}}\right)\ge\mu^{(m)}(g)\mathbb{P}_{id}\left(A_{H_{o}}\right)>0$,
it follows that $\mathbb{P}_{id}\left(A_{H_{0}}\right)<1$. We conclude
that $A_{H_{o}}$ is a non-trivial tail event and the Poisson boundary
of $(G,\mu)$ is non-trivial. 

\end{proof}

The boundary behavior of random walk in Proposition \ref{transience}
can be viewed as a lamplighter boundary. Under the embedding $\vartheta:g\mapsto(\Phi_{g},g)$,
we refer to $\Phi_{g}$ as the germ configuration of $g$. Let $(W_{n})$
be a random walk on $G$ with a non-degenerate step distribution $\mu$.
If (\ref{eq:summable}) holds, then for every point $x\in o\cdot G$,
the coset of $\Phi_{W_{n}}(x)$ in $\mathcal{\hat{G}}_{x}/H_{x}$
stabilizes with probability one along an infinite trajectory of the
random walk $(W_{n})$. Thus if we view $\prod_{x}\mathcal{\hat{G}}_{x}/H_{x}$
as the space of germ configurations mod $\prod_{x}H_{x}$, then when
stabilization occurs, we have a limit configuration $\Phi_{W_{\infty}}$
in this space of cosets. Endowed with the hitting distribution, it
can be viewed as a $\mu$-boundary. 

\section{measures with nontrivial Poisson boundaries on the first Grigorchuk
group\label{sec:existence}}

In this section, we give quick examples of non-degenerate symmetric
probability measures on the first Grigorchuk group with non-trivial
Poisson boundary. Choices of such measures with better quantitative
control over the tail decay will be discussed in Section \ref{sec: growth-first}
and Section \ref{sec:Main}.

Recall that the first Grigorchuk group $G=G_{012}$ acts on the rooted
binary tree, and it is generated by $\{a,b,c,d\}$ where 
\[
a=(1,1)\varepsilon,\ b=(a,c),\ c=(a,d),\ d=(1,b).
\]
As explained in Example \ref{ggerms}, the generators $b,c,d$ have
nontrivial germs at $o=1^{\infty}$. The isotropy group of $\mathcal{G}$
at $o$ is $\mathcal{G}_{o}\simeq\left(\mathbb{Z}/2\mathbb{Z}\right)\times\left(\mathbb{Z}/2\mathbb{Z}\right)$.
List the elements of $\mathcal{G}_{o}$ as $\left\{ id,b,c,d\right\} $,
where $b$ stands for the germ $(b,o)$, similarly $c$ and $d$ stand
for the corresponding germs. Take the following proper subgroup of
$\mathcal{G}_{o}$:
\[
H_{o}=\{id,b\}.
\]
The group $L$ of finitary automorphisms of the rooted binary tree
plays the role of the auxiliary group. Recall the definition of the
sub-groupoid $\mathcal{H}^{b}=\mathcal{H}(H_{o})$ as in (\ref{eq:H_v}),
which we will refer to as the sub-groupoid of $b$-germs. Now consider
the subgroup $\mathrm{H}^{b}<G$ which consists of group elements
that only have trivial or $b$-germs, that is
\begin{equation}
\mathrm{H}^{b}=\{g\in G:\ (g,x)\in\mathcal{H}^{b}\ \mbox{for all }x\in1^{\infty}\cdot G\}.\label{eq: Hb}
\end{equation}
Given an element in $G$, one can recognize whether it is in ${\rm H}^{b}$
from its sections in a deep enough level.

\begin{fact}\label{bsection}

Let $g\in G$ be a group element. Then $g\in\mathrm{H}^{b}$ if and
only if there exists $k\in\mathbb{N}$ such that the section $g_{v}$
is in the set $\{id,a,b\}$ for all $v\in\mathsf{L}_{3k}$. 

\end{fact}

\begin{proof}

The ``if'' direction follows from definition of $\mathrm{H}^{b}$.

We show the only if direction. Since $G$ is contracting, for any
$g\in G$, there exists a finite level $n$ such that all sections
$g_{v}$, $v\in\mathsf{L}_{n}$, are in the nucleus $\mathcal{N}=\{id,a,b,c,d\}$,
see \cite[Section 2.11]{NekraBook}. Perform the wreath recursion
down further $1$ or $2$ levels such that $n\equiv0\ \mod3$. Suppose
on the contrary there exists a vertex $v\in\mathsf{L}_{n}$ such that
the section is $c$ or $d$. Then $(g,v1^{\infty})$ is a $c$-germ
or $d$-germ, a contradiction with $g\in{\rm H}^{b}$. Thus the level
$n$ satisfies the condition in the statement.

\end{proof}

\begin{figure}[!htb] \centering \includegraphics[scale=.4]{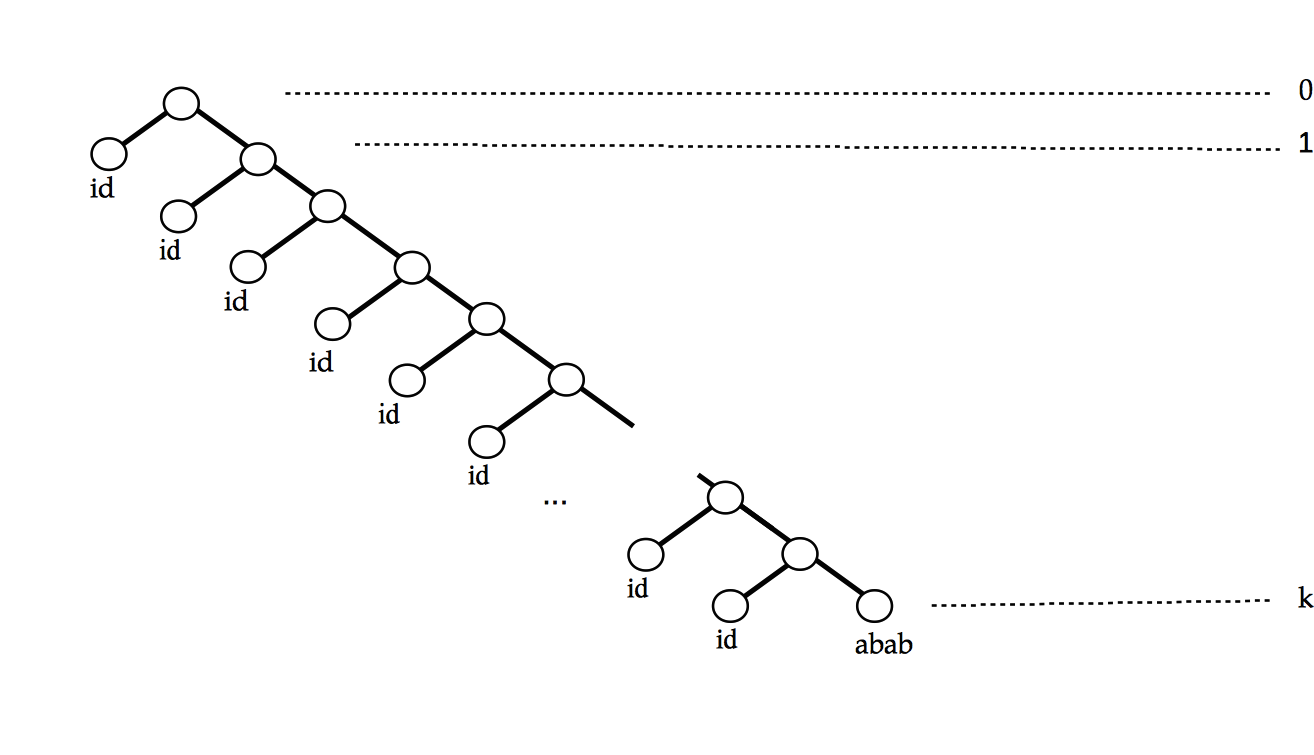} 
\caption{Portrtrait of $\sigma^k(abab)$ in the first Grigorchuk group in usual notation.      If $k$ is divisible by $3$, then the element belongs to $\rm{H}^b$ (the group with germs equal to the germ of $b$ or trivial.)  } \label{fig:portraitsigmaABAB} \end{figure} 

A key property we will use is that the orbit of $1^{\infty}$ under
the action of ${\rm H}^{b}$ is infinite. Recall the notation $\iota(g,u)$
as defined in (\ref{eq:embed-rigid}), which denotes the automorphism
in the rigid stabilizer of $u$ that acts as $g$ in the subtree rooted
at $u$. Note that for any $n\in\mathbb{N}$, $\iota([a,b],1^{n})\in G$.
Explicitly we have $\iota([a,b]:1^{n})=\sigma^{n}(abab)$ where $\sigma$
is the homomorphism given by substitution $a\mapsto cac,b\mapsto d,c\mapsto b,d\mapsto c$.
More generally, $\iota(g,u)\in G$ for any $g\in K=\left\langle \left[a,b\right]\right\rangle ^{G}$
and any vertex $u\in\mathsf{T}$, because $G$ is regularly branching
over the subgroup $K$, see \cite[Proposition 1.25]{BGShandbook}.
The portrait of $\iota\left([a,b],1^{k}\right)$ is drawn in Figure
\ref{fig:portraitsigmaABAB}. The subset $\{\iota\left([a,b],1^{3k}\right)$:
$k\ge0\}$ is contained in ${\rm H}^{b}$. Since $1^{\infty}\cdot[a,b]=1001^{\infty}$,
it follows that 
\[
1^{\infty}\cdot\iota\left([a,b],1^{3k}\right)=1^{3k+1}001^{\infty}.
\]
Therefore the orbit $1^{\infty}\cdot{\rm H}^{b}$ is infinite. 

By Proposition \ref{transience} and Remark \ref{small} we can take
a measure with non-trivial Poisson boundary of the form $\mu=\frac{1}{2}(\nu+\eta)$,
where $\mbox{supp}\nu=\{a,b,c,d\}$ and $\eta$ is a symmetric measure
supported on ${\rm H}^{b}$ with transient induced random walk on
the orbit $1^{\infty}\cdot{\rm H}^{b}$. Existence of such a measure
$\eta$ is deduced from the fact that the orbit $1^{\infty}\cdot{\rm H}^{b}$
is infinite by the following lemma. It is \cite[Lemma 7.1]{Erschler04}
with the assumption of $A$ being finitely generated dropped. Since
the subgroup ${\rm H}^{b}$ is not finitely generated, we include
a proof of the lemma. Let $A$ be a countable group and $B$ a subgroup
of $A$. Given a probability measure $\eta$ on $A$, recall that
the induced transition kernel $P_{\eta}$ on the (left) cosets of
$B$ is 
\[
P_{\eta}(Bx,By)=\sum_{g\in A}\eta(g)\mathbf{1}_{\{Bxg=By\}},
\]
where $x,y\in A$ and $Bx,By$ are cosets of $B$. We say the measure
$\eta$ is transient with respect to $B$ if ${\bf G}_{P_{\eta}}(B,B)=\sum_{n=0}^{\infty}P_{\eta}^{n}(B,B)<\infty$. 

\begin{lem}\label{subordinate-1}

Let $A$ be a countable group and $B<A$ a subgroup of infinite index.
Then there exists a non-degenerate symmetric probability measure $\eta$
on $A$ that is transient with respect to $B$. 

\end{lem}

Apply Lemma \ref{subordinate-1} to the group ${\rm H}^{b}$ and its
subgroup ${\rm H}^{b}\cap{\rm St}_{G}(1^{\infty})$. Since the orbit
$1^{\infty}\cdot{\rm H}^{b}$ is infinite, there is a symmetric measure
$\eta$ on ${\rm H}^{b}$ such that the induced random walk $P_{\eta}$
on $1^{\infty}\cdot{\rm H}^{b}$ is transient. Let ${\bf u}_{S}$
be uniform on the generating set $S=\{a,b,c,d\}$ and take $\mu=\frac{1}{2}({\bf u}_{S}+\eta)$.
Since obviously the Dirichlet forms of $P_{\mu}$ and $P_{\eta}$
satisfy $\mathcal{E}_{P_{\mu}}\ge\frac{1}{2}\mathcal{E}_{P_{\eta}}$,
by comparison principle (see for example \cite[Corollary 2.14]{WoessBook}),
the induced random walk $P_{\mu}$ on $1^{\infty}\cdot G$ is transient
as well. Since $\eta$ is supported on the subgroup ${\rm H}^{b}$
with only $b$-germs and $\nu$ is of finite support, as explained
in Remark \ref{small}, transience of $P_{\mu}$ verifies the condition
(\ref{eq:summable}) in Proposition \ref{transience}. We conclude
that $(G,\mu)$ has non-trivial Poisson boundary. 

\begin{proof}[Proof of Lemma \ref{subordinate-1}]

First take a symmetric measure $\mu$ on $A$ such that $\mu(id)=\frac{1}{2}$
and the support of $\mu$ generates the group $A$. The induced random
walk $P_{\mu}$ on the cosets $B\setminus A$ is irreducible since
${\rm supp}\mu$ generates $A$. Since $B\setminus A$ is infinite,
$P_{\mu}$ admits no non-zero invariant vector in $\ell^{2}(B\setminus A)$,
therefore 
\[
P_{\mu}^{n}(B,B)\searrow0\ \mbox{as }n\to\infty.
\]
Indeed, denote by $E_{s}^{I-P_{\mu}}$ the spectral resolution of
$I-P_{\mu},$
\[
I-P_{\mu}=\int_{0}^{1}sdE_{s}^{I-P_{\mu}}.
\]
Write $N(s)=\left\langle E_{s}^{I-P_{\mu}}\mathbf{1}_{B},\mathbf{1}_{B}\right\rangle $Then
$P_{\mu}^{n}(B,B)=\int_{0}^{1}(1-s)^{n}dN(s)$. Since $(I-P_{\mu})f=0$
has no solution in $\ell^{2}(B\setminus A)$, we have that $N(0)=0$.
Thus $P_{\mu}^{n}(B,B)\searrow0\ \mbox{as }n\to\infty$. 

Next we show that if $P_{\mu}^{n}(B,B)\to0$ as $n\to\infty$, then
some convex linear combinations of convolution powers of $\mu$, that
is measures of the form
\[
\eta=\sum_{n=1}^{\infty}c_{n}\mu^{(n)},\ \sum_{n=1}^{\infty}c_{n}=1,\ c_{n}\ge0,
\]
induce transient random walks on $B\setminus A$. Following \cite{BSC},
we call $\eta$ a discrete subordination of $\mu$. Let $\psi$ be
a smooth strictly increasing function on $(0,\infty)$ such that $\psi(0)=0$,
$\psi(1)=1$. Suppose that we have expansion $\psi(1-s)=1-\sum_{n=1}^{\infty}c_{n}s^{n}$
for $|s|\le1$ with coefficients $c_{n}\ge0$. Then one can take 
\[
\mu_{\psi}:=I-\psi(I-\mu)=\sum_{n=1}^{\infty}c_{n}\mu^{(n)}.
\]
The class of functions $\psi$ that are suitable for this operation
is the Bernstein functions, see the book \cite{SSVbook}.

The induced random walk transition operator $P_{\mu}$ is a self-adjoint
operator acting on $\ell^{2}(B\setminus A)$. Recall that $E_{s}^{I-P_{\mu}}$
denotes the spectral resolution of $I-P_{\mu},$then we have 
\[
I-P_{\mu_{\psi}}=\psi\left(I-P_{\mu}\right)=\int_{0}^{1}\psi(s)dE_{s}^{I-P_{\mu}}.
\]
It follows that 
\[
P_{\mu_{\psi}}^{2n}(B,B)=\int_{0}^{1}(1-\psi(s))^{2n}d\left\langle E_{s}^{I-P_{\mu}}\mathbf{1}_{B},\mathbf{1}_{B}\right\rangle .
\]
Since there is no non-zero $P_{\mu}$-invariant vector in $\ell^{2}(B\setminus A)$,
$N(s)\to0$ as $s\to0+$. Since
\[
\sum_{n=0}^{\infty}P_{\mu_{\psi}}^{2n}(B,B)=\sum_{n=0}^{\infty}\int_{0}^{1}(1-\psi(s))^{2n}dN(s)=\int_{0}^{1}s^{-1}dN\circ\psi^{-1}(s),
\]
to make the Green function finite, it suffices to take some $\psi$
such that $N\circ\psi^{-1}(s)\le cs^{1+\epsilon}$ for some $\epsilon>0$
near $0$. For example one can take the measure $\nu$ on $(0,\infty)$
such that 
\[
\nu((x,\infty))=N(1/x)^{\frac{1}{1+\epsilon}},\ x\ge1,
\]
and $\psi$ to be the Bernstein function with representing measure
$\nu$, $\psi(s)=\int_{(0,\infty)}(1-e^{-st})\nu(dt)$. Note that
the coefficients $c_{n}$ in the expansion $\psi(1-s)=1-\sum_{n=1}^{\infty}c_{n}s^{n}$
are positive. We conclude that for this choice the $P_{\mu_{\psi}}$-random
walk on $B\setminus A$ is transient.

\end{proof}

The measure $\eta$ in Lemma \ref{subordinate-1} with transient induced
random walk on the orbit of $o$ is obtained from discrete subordination
of some non-degenerate measure on ${\rm H}^{b}$. We wish to point
out that in general, the measure $\eta$ does not necessarily have
finite entropy. In Section \ref{sec:construction}, we develop a direct
method to construct measures with transient induced random walks such
that the resulting measures have finite entropy and explicit tail
decay bounds. Unlike the general Lemma \ref{subordinate-1}, this
construction heavily relies on the structure of the groups under consideration.
This will be crucial in applications to growth estimates. 

\section{Further construction of measures with transient induced random walk
\label{sec:construction}}

\subsection{Measures constructed from cube independent elements }

In this subsection we describe a construction of measures with transient
induced random walks. It is particularly useful in groups acting by
homeomorphisms possessing a rich collection of rigid stabilizers.
Transience of the induced random walk is deduced from $\ell^{2}$-isoperimetric
inequalities. 

Let $P$ be a transition kernel on a countable graph $X$. We assume
that $P$ is reversible with respect to $\pi$, that is $\pi(x)P(x,y)=\pi(y)P(y,x)$.
The Dirichlet form of $P$ is defined by 
\[
\mathcal{E}_{P}(f)=\frac{1}{2}\sum_{x,y\in X}(f(x)-f(y))^{2}P(x,y)\pi(x).
\]
Denote by $\left\Vert f\right\Vert _{\ell^{2}(\pi)}$ the $\ell^{2}$-norm
of $f$ with respect to the measure $\pi$. Given a finite subset
$\Omega\subset X$, denote by $\lambda_{1}(\Omega)$ the smallest
Dirichlet eigenvalue:
\[
\lambda_{1}(\Omega)=\inf\left\{ \mathcal{E}_{P}(f):\ {\rm supp}f\subseteq\Omega,\ \left\Vert f\right\Vert _{\ell^{2}(\pi)}=1\right\} .
\]
An inequality of the form $\lambda_{1}(\Omega)\ge\Lambda(\pi(\Omega))$
is referred to as an $\ell^{2}$-isoperimetric inequality. It is also
known as a Faber-Krahn inequality. 

The following isoperimetric test for transience is from \cite[Theorem 6.12]{grigoryanAG}.
It is a consequence of the connection between $\ell^{2}$-isoperimetric
inequalities and on-diagonal heat kernel upper bound, see Coulhon
\cite[Proposition V.1]{coulhon} or Grigor'yan \cite{grigoryan}.

\begin{prop}[Isoperimetric test for transience \cite{grigoryanAG}]\label{test}

Suppose $v_{0}=\inf_{x\in X}\pi(x)>0$. Assume that 
\[
\lambda_{1}(\Omega)\ge\Lambda(\pi(\Omega))
\]
holds for all finite non-empty subset $\Omega\subset X$, where $\Lambda$
is a continuous positive decreasing function on $[v_{0},\infty)$.
If 
\[
\int_{v_{0}}^{\infty}\frac{ds}{s^{2}\Lambda(s)}<\infty,
\]
then the Markov chain $P$ on $X$ is transient. 

\end{prop}

The best possible choice of the function $\Lambda$ is defined as
the $\ell^{2}$-\emph{isoperimetric profile} $\Lambda_{P}$ of $P$,
\[
\Lambda_{P}(v):=\inf\left\{ \lambda_{1}(\Omega):\ \Omega\subset X\mbox{ and }\pi(\Omega)\le v\right\} .
\]
A useful way to obtain lower bounds on $\Lambda_{P}$ is through $\ell^{1}$-isoperimetry.
By Cheeger's inequality (see \cite[Theorem 3.1]{LS}) 
\begin{equation}
\lambda_{1}(\Omega)\ge\frac{1}{2}\left(\frac{\left|\partial_{P}\Omega\right|}{\pi(\Omega)}\right)^{2},\label{eq:cheeger}
\end{equation}
where $\left|\partial_{P}\Omega\right|=\sum_{x\in\Omega,y\in\Omega^{c}}\pi(x)P(x,y)$
is the size of the boundary of $\Omega$ with respect to $(P,\pi)$. 

In what follows we will consider $P_{\mu}$ on $X$ induced by a symmetric
probability measure $\mu$ on $G$, $G\curvearrowright X$. In this
case $P_{\mu}$ is symmetric and reversible with respect to $\pi\equiv1$. 

\begin{lem}\label{CSC}

Let $G$ be a countable group acting on $X$, $o\in X$ be a point
such that the orbit $o\cdot G$ is infinite. Suppose there is a sequence
of finite subsets $F_{n}\subseteq G$ with $|F_{n}|\nearrow\infty$
as $n\to\infty$ and a constant $c_{0}>0$ such that 
\[
\inf_{x\in o\cdot G}\left|x\cdot F_{n}\right|\ge c_{0}|F_{n}|.
\]
Let $(\lambda_{n})$ be a sequence of positive numbers such that $\sum_{n=1}^{\infty}\lambda_{n}=1$
and 
\[
\sum_{n=1}^{\infty}\frac{1}{\lambda_{n}\left(|F_{n}|-|F_{n-1}|\right)}<\infty.
\]
Let 
\[
\mu=\sum_{n=1}^{\infty}\frac{\lambda_{n}}{2}\left(\mathbf{u}_{F_{n}}+\mathbf{u}_{F_{n}^{-1}}\right),
\]
where $\mathbf{u}_{F}$ denotes the uniform measure on the set $F$.
Then the induced $P_{\mu}$-random walk on $o\cdot G$ is transient.

\end{lem}

\begin{proof}

Denote by $P_{n}$ the transition kernel on $o\cdot G$ induced by
the measure $\frac{1}{2}\left(\mathbf{u}_{F_{n}}+\mathbf{u}_{F_{n}^{-1}}\right)$.
Under the assumption that $\inf_{x\in o\cdot H}\left|x\cdot F_{n}\right|\ge c_{0}|F_{n}|$,
we have that for any set $U\subset o\cdot G$ with $|U|\le\frac{c_{0}}{2}|F_{n}|$,
\[
\frac{|\partial_{P_{n}}U|}{|U|}\ge\frac{c_{0}}{4}.
\]
Indeed, the size of the boundary is given by
\[
|\partial_{P_{n}}U|=\sum_{x\in U,y\in U^{c}}P_{n}(x,y)=\sum_{x\in U}\frac{1}{|F_{n}|}\sum_{g\in F_{n}}\frac{1}{2}\left(\mathbf{1}_{U^{c}}(x\cdot g)+\mathbf{1}_{U^{c}}(x\cdot g^{-1})\right).
\]
The assumption that $\inf_{x\in o\cdot H}\left|\{x\cdot F_{n}\}\right|\ge c_{0}|F_{n}|$
implies that for any $x$,
\[
\sum_{g\in F_{n}}\mathbf{1}_{U^{c}}(x\cdot g)=|x\cdot F_{n}|-\sum_{g\in F_{n}}\mathbf{1}_{U}(x\cdot g)\ge c_{0}|F_{n}|-|U|.
\]
It follows that for $|U|\le\frac{c_{0}}{2}|F_{n}|$, 
\[
\frac{|\partial_{P_{n}}U|}{|U|}\ge\frac{1}{2}\frac{c_{0}|F_{n}|-|U|}{|F_{0}|}\ge\frac{c_{0}}{4}.
\]
We mention that the argument above is well-known, it is a step in
the proof of the Coulhon-Saloff-Coste isoperimetric inequality \cite{csc}. 

By Cheeger's inequality (\ref{eq:cheeger}), the $\ell^{2}$-isoperimetric
profile of $P_{n}$ satisfies
\[
\Lambda_{P_{n}}(v)\ge\frac{1}{2}\left(\frac{c_{0}}{4}\right)^{2}\mbox{ for all }v\le\frac{c_{0}}{2}|F_{n}|.
\]
Now the random walk induced by $\mu$ is the convex linear combination
$P_{\mu}=\sum_{n=1}^{\infty}\lambda_{n}P_{n}$. By definition it is
clear that $\mathcal{E}_{P_{\mu}}\ge\lambda_{n}\mathcal{E}_{P_{n}}$.
Therefore we have a piecewise lower bound on $\Lambda_{P_{\mu}}$:
\[
\Lambda_{P_{\mu}}(v)\ge\frac{1}{2}\left(\frac{c_{0}}{4}\right)^{2}\lambda_{n}\ \ \mbox{for }v\in\left(\frac{c_{0}}{2}|F_{n-1}|,\frac{c_{0}}{2}|F_{n}|\right],n\in\mathbb{N}.
\]
Plug the estimate into the isoperimetric test for transience, we have
\[
\int_{1}^{\infty}\frac{ds}{s^{2}\Lambda_{P_{\mu}}(s)}\le C\sum_{n=1}^{\infty}\frac{1}{\lambda_{n}(|F_{n}|-|F_{n-1}|)}.
\]
If the summation on the right hand side of the inequality is finite,
then by Proposition \ref{test}, the induced random walk $P_{\mu}$
is transient.

\end{proof}

\begin{exa}

Suppose there is a sequence of $F_{n}$ satisfying the assumption
of Lemma \ref{CSC} with volume $|F_{n}|\sim\theta^{n}$. Then one
can choose the sequence $(\lambda_{n})$ to be $\lambda_{n}=C_{\epsilon}\theta^{-n}n^{1+\epsilon}$
for any $\epsilon>0$. The resulting convex linear combination $\mu$
as in the statement of the lemma induces transient random walk on
the orbit of $o$. 

\end{exa}

The condition in Lemma \ref{CSC} forces $F_{n}\cap{\rm St}_{G}(x)$
to be small for all $x$ in the orbit. Such sets are not common to
observe, especially in the situation of totally non-free actions.
In applications to Grigorchuk groups, the subsets $F_{n}$ are built
from a sequence of elements which satisfy an injective property which
we refer to as \textquotedbl cube independence\textquotedbl{} defined
below. We have mentioned the definition of cube independence property
for ${\bf k}=1$ in the Introduction, now we formulate a more general
definition. 

For a sequence of elements $g_{1},g_{2},\ldots\in G$, we say the
sequence satisfies \emph{the cube independence property} on $x_{0}\cdot G$
with parameters $\mathbf{k}=\left(k_{n}\right){}_{n=1}^{\infty}$,
$k_{n}\in\mathbb{N}$, if for any $n\ge1$ and $x\in x_{0}\cdot G$,
the map 
\begin{align*}
\{0,1,\ldots,k_{1}\}\times\ldots\times\{0,1,\ldots,k_{n}\} & \to x_{0}\cdot G\\
(\varepsilon_{1},\ldots\varepsilon_{n}) & \mapsto x\cdot g_{n}^{\varepsilon_{n}}\ldots g_{1}^{\varepsilon_{1}}
\end{align*}
is injective. When the parameters $k_{n}$ are constant $1$, we omit
reference to $\mathbf{k}$ and say the sequence has the cube independence
property.

Any other ordering of the elements $g_{1},\ldots,g_{n}$ in the definition
would serve the same purpose, but we choose to formulate it this way
because in examples we consider it is easier to verify inductively.
By definition it is clear that cube independence property is inherited
by subsequences.

\begin{prop}\label{indep_tran}

Suppose $g_{1},g_{2},\ldots\in G$ is a sequence of elements satisfying
the cube independence property on the orbit $x_{0}\cdot G$ with $\mathbf{k}=(k_{n})$,
$k_{n}\in\mathbb{N}$. Let 
\[
F_{n}=\left\{ g_{n}^{\varepsilon_{n}}\ldots g_{1}^{\varepsilon_{1}}:\ \varepsilon_{j}\in\{0,1,\ldots,k_{j}\},1\le j\le n\right\} .
\]
For any $\epsilon>0$, let
\[
\mu=\sum_{n}\frac{C_{\epsilon}n^{1+\epsilon}}{2(1+k_{1})\ldots(1+k_{n})}\left(\mathbf{u}_{F_{n}}+\mathbf{u}_{F_{n}^{-1}}\right),
\]
where $C_{\epsilon}$ is the normalizing constant such that $\mu$
has total mass $1$. Then $\mu$ is a symmetric probability measure
on $G$ of finite entropy such that the induced $P_{\mu}$-random
walk on $x_{0}\cdot G$ is transient.

\end{prop}

\begin{proof}

By definition of cube independence property, we have for all $x\in x_{0}\cdot G$,
$|x\cdot F_{n}|=|F_{n}|=\prod_{j=1}^{n}(1+k_{j})$. Choose $\lambda_{n}=C_{\epsilon}n^{1+\epsilon}\prod_{j=1}^{n}(1+k_{j})^{-1}$,
then
\[
\sum_{n=1}^{\infty}\frac{1}{\lambda_{n}\left(|F_{n}|-|F_{n-1}|\right)}\le2\sum_{n=1}^{\infty}\frac{1}{\lambda_{n}|F_{n}|}=\frac{2}{C_{\epsilon}}\sum_{n=1}^{\infty}n^{-1-\epsilon}<\infty.
\]
Therefore by Lemma \ref{CSC}, the convex combination $\mu=\sum_{n=1}^{\infty}\frac{\lambda_{n}}{2}\left(\mathbf{u}_{F_{n}}+\mathbf{u}_{F_{n}^{-1}}\right)$
induces transient random walk on the orbit $o\cdot G$. The entropy
of $\mu$ is bounded by 
\begin{align*}
H_{\mu}(1) & \le\sum_{n=1}^{\infty}\lambda_{n}\log|F_{n}|\\
 & =C\sum_{n=1}^{\infty}\frac{n^{1+\epsilon}}{\prod_{j=1}^{n}(1+k_{j})}\log\left(\prod_{j=1}^{n}(1+k_{j})\right)<\infty,
\end{align*}
where the series is summable because $\prod_{j=1}^{n}(1+k_{j})\ge2^{n}$. 

\end{proof}

\begin{rem}

As mentioned in the Introduction, we refer to the set $F_{n}$ as
an $n$-quasi-cubic set and the measure ${\bf u}_{F_{n}}$ an $n$-quasi-cubic
measure. By definition bounds on the length of elements $g_{1},g_{2},\ldots$
give explicit estimates of the truncated moments and tail decay of
$\mu$. More precisely, if $\sum_{j=1}^{n}(1+k_{j})|g_{j}|\le\ell_{n}$,
then $\max\{|g|,\ g\in F_{n}\}\le\ell_{n}$ and for the measure $\mu$
in Corollary \ref{indep_tran},
\[
\mu\left(B(e,\ell_{n})^{c}\right)\le\frac{C'_{\epsilon}n^{1+\epsilon}}{(1+k_{1})\ldots(1+k_{n})}.
\]

\end{rem}

One way to produce a cube independent sequence is to select suitable
elements from level stabilizers.

\begin{lem}\label{elem_ind}

Let $G<{\rm Aut}(\mathsf{T}_{d})$ and suppose that $g_{n}\in G$
satisfies that $g_{n}\in{\rm St}_{G}(\mathsf{L}_{n})$ and the root
permutation of the section $(g_{n})_{v}$ has no fixed point for all
$v\in\mathsf{L}_{n}$. Then the sequence $(g_{n})_{n\ge0}$ has the
cube independence property on $1^{\infty}\cdot G$. 

\end{lem}

\begin{proof}

Given $x\in1^{\infty}\cdot G$, we need to show injectivity of the
map
\begin{align*}
\{0,1\}^{n} & \to1^{\infty}\cdot G\\
(\varepsilon_{1},\ldots\varepsilon_{n}) & \mapsto x\cdot g_{n}^{\varepsilon_{n}}\ldots g_{1}^{\varepsilon_{1}}.
\end{align*}
Since $\left\langle g_{2},\ldots,g_{n}\right\rangle $ is in the stabilizer
of level $2$, we have that the first two digits of $y=x\cdot g_{n}^{\varepsilon_{n}}\ldots g_{1}^{\varepsilon_{1}}$
is $x_{1}x_{2}\cdot g_{1}^{\epsilon}$. Combined with the assumption
on the root permutation of $(g_{1})_{v}$, $v\in\mathsf{L}_{1}$,
the second digit of $y$ is different from $x_{2}$ if and only if
$\varepsilon_{1}=1$. Thus $x\cdot g_{n}^{\varepsilon_{n}}\ldots g_{1}^{\varepsilon_{1}}=x\cdot g_{n}^{\varepsilon'_{n}}\ldots g_{1}^{\varepsilon'_{1}}$
implies $\varepsilon_{1}=\varepsilon_{1}'$. The same argument recursively
shows that $\varepsilon_{j}=\varepsilon_{j}'$ for all $1\le j\le n$. 

\end{proof}

\subsection{Critical constant of recurrence\label{subsec:critical}}

In this subsection we consider an explicit sequence of cube independent
elements on the first Grigorchuk group $G$, obtained by applying
certain substitutions. As an application of Proposition \ref{indep_tran},
we evaluate the critical constant for recurrence of $(G,{\rm St}_{G}(1^{\infty}))$. 

The notion of critical constant for recurrence of $(G,H)$, where
$H$ is a subgroup of $G$, is introduced in \cite{Erschler05}. For
a group $G$ equipped with a length function $l$ and a subgroup $H<G$
of infinite index, the \emph{critical constant for recurrence} $c_{{\rm rt}}(G,H)$
with respect to $l$, is defined as $\sup\beta$, where the $\sup$
taken is over all $\beta\ge0$ such that there exists a symmetric
probability measure $\mu$ on $G$ with transient induced random walk
on $H\setminus G$ and $\mu$ has finite $\beta$-moment with respect
to $l$ on $G$:
\[
\sum_{g\in G}l(g)^{\beta}\mu(g)<\infty.
\]
When $G$ is finitely generated and $l$ is the word distance on a
Cayley graph of $G$, we omit reference to $l$. 

For the rest of this subsection, $G$ denotes the first Grigorchuk
group $G_{012}$ and we use the usual notation for $G$ as in (\ref{eq:un}).
By \cite[Theorem 2]{Erschler05}, for $H$ the stabilizer ${\rm St}_{G}(1^{\infty})$,
the critical constant $c_{{\rm rt}}(G,{\rm St}_{G}(1^{\infty}))<1$.
We show that indeed $c_{{\rm rt}}\left(G,{\rm St}_{G}(1^{\infty})\right)$
is equal to the growth exponent of the permutational wreath product
$(\mathbb{Z}/2\mathbb{Z})\wr_{\mathcal{S}}G$, $\mathcal{S}=1^{\infty}\cdot G$. 

\begin{thm}\label{critical}

Let $G$ be the first Grigorchuk group. Then 

\[
c_{{\rm rt}}\left(G,{\rm St}_{G}(1^{\infty})\right)=\alpha_{0}=\frac{\log2}{\log\lambda_{0}},
\]
where $\lambda_{0}$ is the positive root of the polynomial $X^{3}-X^{2}-2X-4$,
$\alpha_{0}\approx0.7674$. 

\end{thm}

The proof of Theorem \ref{critical} consists of two parts. To show
the lower bound, we construct an explicit measure with transient induced
random walks on the orbit of $1^{\infty}$ by applying Proposition
\ref{indep_tran}. 

Recall the substitutions $\zeta$ of $\{ab,ac,ad\}^{\ast}$, which
is used to produce words with asymptotic maximal inverted orbit growth
in \cite[Proposition 4.7]{BE1}, 

\[
\zeta:ab\mapsto abadac,\ ac\mapsto abab,\ ad\mapsto acac.
\]
By direct calculation $\zeta(ab)=(b,aba)\varepsilon$, $\zeta(ac)=(ca,ac)$,
$\zeta(ad)=(da,ad)$. 

\begin{fact}\label{zeta}

Let $w$ be a word in the alphabet $\{ab,ac,ad\}$ such that it has
\emph{even} number of $a$'s. Then $\zeta^{n}(w)$ is in the $n$-th
level stabilizer, and the section at the vertex $v=x_{1}\ldots x_{n}$
is given by 
\[
\left(\zeta^{n}(w)\right)_{v}=\begin{cases}
w & \mbox{if }x_{1}+\ldots+x_{n}\equiv n\mod2\\
awa & \mbox{if }x_{1}+\ldots+x_{n}\equiv n-1\mod2.
\end{cases}
\]

\end{fact}

\begin{proof}

The claim can be verified by induction on $n$. For $w$ with even
number of $a$'s, $\zeta(w)=(awa,w).$ Note that if $w$ has even
number of $a$'s then so does $\zeta(w)$. Applied to $\zeta^{n-1}(w)$,
we have $\zeta^{n}(w)=(a\zeta^{n-1}(w)a,\zeta^{n-1}(w))$. Then the
induction hypothesis on $\zeta^{n-1}(w)$ implies the statement on
$\zeta^{n}(w)$. 

\end{proof}

By Fact \ref{zeta} and Lemma \ref{elem_ind}, we have that $\left(\zeta^{n}(ac)\right)_{n=1}^{\infty}$
form a cube independent sequence. With considerations of germs in
mind (although not relevant in this subsection), we prefer to take
the following sequence.

\begin{exa}\label{g_nexa}

As mentioned in the Introduction, take the following sequence of elements
$(g_{n})_{n=1}^{\infty}$ in $G$: 
\begin{equation}
g_{n}=\begin{cases}
\zeta^{n}(ac) & \mbox{if }n\equiv0,1\mod3\\
\zeta^{n}(ad) & \mbox{if }n\equiv2\mod3.
\end{cases}\label{eq:g_n}
\end{equation}
Then by Fact \ref{zeta} and Lemma \ref{elem_ind}, $(g_{n})$ has
the cube independence property. 

\end{exa}

With this sequence of cube independent elements, take the $n$-quasi-cubic
sets
\[
F_{n}=\left\{ g_{n}^{\varepsilon_{n}}\ldots g_{1}^{\varepsilon_{1}}:\ \varepsilon_{j}\in\{0,1\},1\le j\le n\right\} .
\]
and the convex combination 
\begin{equation}
\eta_{0}=\sum_{n=1}^{\infty}\frac{C_{\epsilon}n^{1+\epsilon}}{2^{n}}\left(\mathbf{u}_{F_{n}}+\mathbf{u}_{F_{n}^{-1}}\right)\label{eq:eta0}
\end{equation}
where $C_{\epsilon}>0$ is the normalization constant such that $\eta_{0}$
is a probability measure. Then by Proposition \ref{indep_tran}, the
$\eta_{0}$-induced random walk on the orbit of $1^{\infty}$ is transient. 

By definition the length of $g_{n}$ can be estimated by eigenvalues
of the matrix associated with the substitution $\zeta$. Record the
number of occurrences of $ab$, $ac$, $ad$ in a word $w$ as a column
vector ${\bf l}(w)$, then by definition of the substitution $\zeta$,
we have ${\bf l}(\zeta(w))=M{\bf l}(\omega)$ where the matrix $M$
is 
\[
M=\left[\begin{array}{ccc}
1 & 2 & 0\\
1 & 0 & 2\\
1 & 0 & 0
\end{array}\right].
\]
It follows that 
\begin{align*}
|\zeta^{n}(ac)| & =\left(\begin{array}{ccc}
2 & 2 & 2\end{array}\right)M^{n}\left(\begin{array}{c}
0\\
1\\
0
\end{array}\right),\\
|\zeta^{n}(ad)| & =\left(\begin{array}{ccc}
2 & 2 & 2\end{array}\right)M^{n}\left(\begin{array}{c}
0\\
0\\
1
\end{array}\right).
\end{align*}
Therefore 
\[
|g_{n}|\le C\lambda_{0}^{n},
\]
where $\lambda_{0}$ is the spectral radius of $M$, that is the positive
root of the characteristic polynomial $X^{3}-X^{2}-2X-4$. 

\begin{proof}[Proof of lower bound in Theorem \ref{critical}]

Consider the measure $\eta_{0}$ defined in (\ref{eq:eta0}) which
induces a transient random walk on the orbit of $1^{\infty}$ by Corollary
\ref{indep_tran}. Since 
\[
\max\{|g|:g\in F_{n}\}\le\sum_{j=1}^{n}|g_{j}|\le C\sum_{j=1}^{n}\lambda_{0}^{n}=C'\lambda_{0}^{n},
\]
it follows that 
\[
\eta_{0}\left(g:\ |g|\ge C'\lambda_{0}^{n}\right)\le\frac{Cn^{1+\epsilon}}{2^{n}}.
\]
This tail estimate implies that for any $\beta<\frac{\log2}{\log\lambda_{0}}$,
$\mu$ has finite $\beta$-moment. Since the $\eta_{0}$-induced random
walk on the orbit of $1^{\infty}$ is transient, it implies 
\[
c_{{\rm rt}}\left(G,{\rm St}_{G}(1^{\infty})\right)\ge\alpha_{0}=\frac{\log2}{\log\lambda_{0}}.
\]

\end{proof}

The upper bound $c_{{\rm rt}}\left(G,{\rm St}_{G}(1^{\infty})\right)\le\alpha_{0}$
is a consequence of the volume growth estimate of $W=(\mathbb{Z}/2\mathbb{Z})\wr_{\mathcal{S}}G$
in \cite{BE1}. 

\begin{proof}[Proof of upper bound in Theorem \ref{critical}]

We show a slightly stronger statement: for any non-degenerate probability
measure $\mu$ on $G$ with finite $\alpha_{0}$-moment, the induced
random walk on the orbit of $1^{\infty}$ is recurrent.

Suppose the claim is not true, let $\mu$ be a probability measure
on $G$ with finite $\alpha_{0}$-moment and transient induced random
walk on the orbit of $1^{\infty}$. On $W$, let $\nu$ be the uniform
measure on the lamp group at $o=1^{\infty}$, that is $\nu$ is uniform
on $\{id,\delta_{o}^{1}\}.$ Since the random walk induced by $\mu$
on the orbit $1^{\infty}\cdot G$ is transient, the measure $\zeta=\frac{1}{2}(\mu+\nu)$
has non-trivial Poisson boundary by \cite[Proposition 3.3]{BE3}.
On the other hand the volume growth functions satisfy
\[
v_{G,S}(r)\le v_{W,T}(r)\lesssim\exp(r^{\alpha_{0}})
\]
by \cite[Lemma 5.1]{BE1}. Since $\zeta$ has finite $\alpha_{0}$-moment,
Corollary \ref{trivialmoment} implies $(W,\zeta)$ has trivial Poisson
boundary, which is a contradiction.

\end{proof}

Note that in the definition of the critical constant $c_{{\rm rt}}(G,H)$
we restrict to symmetric random walks on $G$. A priori the critical
constant might become strictly larger if one includes non-symmetric
random walks. For example, for $G=\mathbb{Z}$ and $H=\{0\}$, $c_{{\rm rt}}(\mathbb{Z},\{0\})=1$
while biased random walk of finite range, for example $\mu(1)=1+p$
and $\mu(-1)=1-p$, $p\in(0,1/2)$, is transient. However the upper
bound proved above shows that the critical constant of $(G,{\rm St}_{G}(1^{\infty}))$
remains $\alpha_{0}$ if we take sup of moment over all measures with
transient induced random walk on the orbit of $1^{\infty}$. 

\begin{rem}

The measure $\eta_{0}$ with transient induced random walk defined
in (\ref{eq:eta0}) can be used to show volume lower estimates for
certain extensions of the first Grigorchuk group. For example, consider
\textquotedbl double Grigorchuk groups\textquotedbl{} as in \cite{Erschler05}
which are defined as follows. Take an $\omega\in\{{\bf 0},{\bf 1},{\bf 2}\}^{\infty}$
such that $\omega$ is not a shift of $({\bf 012})^{\infty}$ and
not eventually constant. Take the directed automorphism $g=b_{\omega}$
as defined in Subsection \ref{subsec:Grigorchuk-def}, that is $b_{\omega}$
is a generator of the group $G_{\omega}$. Let $\Gamma$ be the group
generated by $S'=\{a,b,c,d,g\}$. Since $\Gamma$ has generators from
two strings, it is called a double Grigorchuk group. For example,
for the group $G(m)$ in \cite[Corollary 2]{Erschler05}, the corresponding
string is $\omega=({\bf (012})^{m}{\bf 22})^{\infty}$. Using the
measure $\eta_{0}$ one can show an improvement of \cite[Corollary 2]{Erschler05}.
At $o=1^{\infty}$, the isotropy group of the groupoid of germs of
$\Gamma$ is strictly larger than that of $G$. Take $\mathcal{H}=\mathcal{G}$
and apply Proposition \ref{transience} to the measure $\mu=\frac{1}{2}\left({\bf u}_{S'}+\eta_{0}\right)$,
where $\eta_{0}$ is the measure on $G$ defined in (\ref{eq:eta0}),
then we have that $\mu$ has non-trivial Poisson boundary. It follows
by Corollary \ref{tailgrowth} that for any $\epsilon>0$, there exists
$c>0$ such that for all $n>1$,
\[
v_{\Gamma,S'}(n)\ge\exp\left(cn^{\alpha_{0}}/\log^{1+\epsilon}n\right).
\]

We remark that there is a gap in the proof of \cite[Proposition 4.1]{Erschler05},
in which in order to guarantee that the measure $\nu$ considered
has finite entropy, one needs to strengthen the assumption in the
statement to $v_{\Gamma_{2},T_{2}}(n)\le\exp\left(Cn^{\gamma}\right)$
for \emph{all} $n$, for some $\gamma<1$, $C>0$ (assuming for infinitely
many $n$ is not sufficient). 

\end{rem}

\section{More examples of measures with restricted germs on the first Grigorchuk
group \label{sec: growth-first}}

In this section we exhibit more examples of symmetric measures on
the first Grigorchuk group with non-trivial Poisson boundary of the
form $\frac{1}{2}({\bf u}_{S}+\eta)$, where $\eta$ is supported
on the subgroup ${\rm H}^{b}$ which consists of elements with only
trivial or $b$-germs. Compared to Section \ref{sec:existence}, the
improvement is that the measures have finite entropy and explicit
tail estimates. As a consequence we derive volume lower bounds for
$G_{012}$ from these measures, which can already be better than previously
known estimates, see Corollary \ref{Hbound}. 

Throughout this section denote by $G$ the first Grigorchuk group
and $S$ the standard generating set $\{a,b,c,d\}$. Recall the subgroup
${\rm H}^{b}$ which consists of elements of $G$ that have only trivial
or $b$-germs defined in (\ref{eq: Hb}). We have the following example
of a sequence of elements in ${\rm H}^{b}$ which satisfies the cube
independence property.

\begin{exa}\label{Hfirst}

Take the sequence of elements $(g_{n})$ satisfying the cube independence
property defined in (\ref{eq:g_n}). From the definition, we have
that for $n\equiv1,2\mod3$, $g_{n}\in{\rm H}^{b}$; while for $n\equiv0\mod3$,
$g_{n}$ has $c$-germs. Filter out the elements with $c$-germs and
keep the subsequence $g_{1},g_{2},g_{4},g_{5},g_{7},g_{8}\ldots$
with $b$-germs. For convenience of notation, relabel the elements
as $\hat{h}_{2k-1}=g_{3k-2},$ $\hat{h}_{2k}=g_{3k-1}$, $k\in\mathbb{N}$. 

\end{exa}

With the sequence $(\hat{h}_{n})$ in Example \ref{Hfirst}, take
the corresponding quasi-cubic sets
\[
\hat{H}_{n}=\left\{ \hat{h}_{n}^{\varepsilon_{n}}\ldots\hat{h}_{1}^{\varepsilon_{1}}:\ \varepsilon_{j}\in\{0,1\},1\le j\le n\right\} .
\]
For any $\epsilon>0$, take $c_{n}=\frac{C_{\epsilon}n^{1+\epsilon}}{2^{n}}$
and 
\[
\eta_{1}=\sum_{n=1}^{\infty}\frac{c_{n}}{2}\left(\mathbf{u}_{\hat{H}_{n}}+\mathbf{u}_{\hat{H}_{n}^{-1}}\right).
\]
Then by Proposition \ref{indep_tran}, the induced random walk $P_{\eta_{1}}$
on $1^{\infty}\cdot{\rm H}^{b}$ is transient. Let ${\bf u}_{S}$
be uniform on the generating set $S=\{a,b,c,d\}$ and $\mu_{1}=\frac{1}{2}({\bf u}_{S}+\eta_{1})$.
Then by comparison principle \cite[Corollary 2.14]{WoessBook}, the
random walk $P_{\mu_{1}}$ on $1^{\infty}\cdot G$ is transient. Since
$\eta_{1}$ is supported on the subgroup ${\rm H}^{b}$ and ${\bf u}_{S}$
has finite support, by Proposition \ref{transience} and Remark \ref{small},
$(G,\mu_{1})$ has non-trivial Poisson boundary. 

Since $\eta_{1}$ has finite entropy and ${\bf u}_{S}$ is of finite
support, $\mu_{1}$ also has finite entropy. We now estimate its tail
decay. From the length estimate of the elements $g_{n}$, we have
\[
\max\{|h|:\ h\in\hat{H}_{n}\}\le\sum_{j=1}^{n}\left|\hat{h}_{j}\right|\le C'\lambda_{0}^{\frac{3}{2}n},
\]
where $\lambda_{0}$ is the spectral radius of $M$, the positive
root of the characteristic polynomial $X^{3}-X^{2}-2X-4$. Therefore
as explained in the remark after Corollary \ref{indep_tran},
\[
\mu_{1}\left(B\left(e,C'\lambda_{0}^{\frac{3}{2}n}\right)^{c}\right)\le\frac{Cn^{1+\epsilon}}{2^{n}}.
\]
In other words,
\[
\mu_{1}\left(B\left(e,r\right)^{c}\right)\le Cr^{-\frac{2}{3\log_{2}\lambda_{0}}}\log^{1+\epsilon}r.
\]
By Corollary \ref{tailgrowth}, we conclude that
\[
v_{G}(r)\ge\exp\left(cr^{\frac{2}{3\log_{2}\lambda_{0}}}/\log^{1+\epsilon}r\right).
\]
The exponent in the bound is $\frac{2}{3\log_{2}\lambda_{0}}=\frac{2}{3}\alpha_{0}\approx0.5116$.
This estimate is non-trivial in the sense that the exponent is strictly
larger than $1/2$, however it is worse than the lower bound $\exp(r^{0.5153})$
proved in Bartholdi \cite{Bartholdi01}. 

It is clear that in the construction of $\eta_{1}$, skipping every
third element in the cube independent sequence $(g_{n})$ leads to
a significant loss. We can improve the construction by choosing another
cube independent sequence that is more adapted to ${\rm H}^{b}$.
Note the following property of ${\rm H}^{b}$.

\begin{lem}\label{Horbit}

The subgroup ${\rm H}^{b}$ defined as in (\ref{eq: Hb}) is locally
finite. The orbit $1^{\infty}\cdot{\rm H}^{b}$ is 
\[
1^{\infty}\cdot{\rm H}^{b}=\{x=x_{1}x_{2}\ldots\in\partial\mathsf{T}:\ x\mbox{ is cofinal with }1^{\infty},\ x_{3i+1}=1\mbox{ for all }i\in\mathbb{N}\}.
\]

\end{lem}

\begin{proof}

We first show ${\rm H}^{b}$ is locally finite. Let $h_{1},\ldots,h_{k}$
be a finite collection of elements in ${\rm H}^{b}$. By Fact \ref{bsection},
there exists a level $3n$, $n\in\mathbb{N}$ such that for every
$h_{j}$, its sections on level $3n$ are in the set $\{id,a,b\}$.
It follows that $\left\langle h_{1},\ldots,h_{k}\right\rangle $ is
a subgroup of $\left\langle a,b\right\rangle \wr_{\mathsf{L}_{3n}}{\rm Aut}(\mathsf{T}^{3n})$,
where $\mathsf{T}^{3n}$ is the finite rooted binary of $3n$ levels.
Since $\left\langle a,b\right\rangle $ is finite, we conclude that
$\left\langle h_{1},\ldots,h_{k}\right\rangle $ is finite. 

Recall that for any $n\in\mathbb{N}$, $[a,b]$ is in the rigid stabilizer
of $1^{n}$. When $n\equiv0\mod3$, the element $\iota([b,a]:1^{n})$
has $b$-germs. Consider the collection $\iota([b,a]:1^{n})$ and
$\iota([b,a]^{2}:1^{n})$, $n=0,3,\ldots$. Since $1^{\infty}\cdot baba=101^{\infty}$
and $1^{\infty}\cdot babababa=1101^{\infty}$, it follows that the
set $x=x_{1}x_{2}\ldots\in\partial\mathsf{T}$ where $x$ is cofinal
with $1^{\infty}$ and $x_{3i+1}=1$ for all $i\in\mathbb{N}$ is
contained in the orbit $1^{\infty}\cdot{\rm H}^{b}$.

Consider the projection of $G$ to the abelian group $\left\langle b,c,d\right\rangle =\mathbb{Z}/2\mathbb{Z}\times\mathbb{Z}/2\mathbb{Z}$.
Note that if the the projection of $g$ is not in the subgroup $\left\langle b\right\rangle $,
then $g\notin{\rm H}^{b}$. Indeed under the wreath recursion to level
$3$, $g=(g_{v})_{v\in\mathsf{L}_{3}}\tau$, the projections satisfy
$\bar{g}=\sum_{v\in\mathsf{L}_{3}}\bar{g}_{v}$. Thus $\bar{g}\notin\left\langle b\right\rangle $
implies at least one of the sections $g_{v}$ satisfies that $\bar{g}_{v}\notin\left\langle b\right\rangle $.
Therefore the claim follows from induction on length of $g$. 

Next we show that for any $h\in{\rm H}^{b}$, $x=1^{\infty}\cdot h$
satisfies $x_{3i+1}=1$ for all $i\in\mathbb{N}$. This can be seen
by induction on word length of $h$. For $|h|=1$, $h=a$ or $h=b$,
the claim is true. Suppose the statement is true for $|h|<n$. For
$h$ of length $|h|=n$, take a word $w$ of length $n$ representing
$h$ and perform the wreath recursion to level $3$. We claim that
if $h\in{\rm H}^{b}$ then all sections $w_{v}$, $v\in\mathsf{L}_{3}$,
have even number of $a$'s. As a consequence of the claim, the $4$-th
digit of $1^{\infty}\cdot h$ is $1$. To verify the claim, if $w_{v_{1}v_{2}v_{3}}$
has odd number of $a$'s, then by the definitions of the generators
we have that the projection of the section at its sibling $w_{v_{1}v_{2}\check{v}_{3}}$
to $\left\langle b,c,d\right\rangle $ is not contained in $\left\langle b\right\rangle $.
This contradicts with the assumption that $h\in{\rm H}^{b}$. Apply
the induction hypothesis to $h_{111}$, we conclude that $x=1^{\infty}\cdot h$
satisfies $x_{3i+1}=1$ for all $i\in\mathbb{N}$. 

\end{proof}

Observe that since the subgroup ${\rm H}^{b}$ acts trivially on levels
$3k+1$, at the corresponding levels we can use the substitution $\sigma$
instead of $\zeta$, where $\sigma$ is the Lysionok substitution
on $\{a,b,c,d\}^{\ast}$ given by
\[
\sigma(a)=aca,\ \sigma(b)=d,\ \sigma(c)=b,\ \sigma(d)=c.
\]
When applied to a word in $\{ab,ac,ad\}^{\ast}$, $\sigma$ always
doubles its length while typically the substitution $\zeta$ results
in a multiplicative factor larger than $2$.

Define a sequence of words $(h_{n})_{n=1}^{\infty}$ by
\begin{equation}
h_{2k-1}=\left(\zeta^{2}\circ\sigma\right)^{k-1}\zeta(ac)\ \mbox{and }h_{2k}=(\zeta^{2}\circ\sigma)^{k-1}\circ\zeta^{2}(ad).\label{eq:h2k}
\end{equation}
Note that $\zeta(ac)=abab$ and $\zeta^{2}(ad)=abababab$, thus $h_{2k}=h_{2k-1}^{2}$. 

\begin{lem}

Let the sequence $(h_{n})_{n=1}^{\infty}$ be defined as in (\ref{eq:h2k}).
Then $h_{n}\in{\rm H}^{b}$ for all $n\in\mathbb{N}$ and the sequence
$(h_{n})_{n=1}^{\infty}$ satisfies the cube independence property
on the orbit $1^{\infty}\cdot{\rm H}^{b}$. 

\end{lem}

\begin{proof}

Note that $\sigma(abab)=(id,abab)$. By definition of $\sigma$, we
have that for any word $w$, $\sigma(w)=(w_{1},w)$ where $w_{1}$
is in $\left\langle a,d\right\rangle $. Since any element $\gamma$
in the finite dihedral group $\left\langle a,d|a^{2}=d^{2}=1,(ad)^{4}=1\right\rangle $
satisfies $\gamma^{4}=1$, it follows that if $\sigma(w^{4})=(id,w^{4})$
for any word $w$. Since $\sigma(abab)=(ac)^{4}$, we have $\sigma\circ\left(\zeta^{2}\circ\sigma\right)^{j}\zeta(ac)=(id,\left(\zeta^{2}\circ\sigma\right)^{j}\zeta(ac))$
for all $j\ge1$. Combined with Fact \ref{zeta}, $h_{2k-1}=\left(\zeta^{2}\circ\sigma\right)^{k-1}\zeta(ac)$
is in the level stabilizer of the level $3k-2$. Moreover in the level
$3k-2$, at a vertex $v\in1^{3k-2}\cdot{\rm H}^{b}$, the section
belongs to $\{ac,ca\}$; otherwise the section is trivial. Since $h_{2k}=h_{2k-1}^{2}$,
$h_{2k}$ is in the level stabilizer of the level $3k-1$ and at a
vertex $v\in1^{3k-1}\cdot{\rm H}^{b}$, the section belongs to $\{ad,da\}$;
otherwise the section is trivial. From the description of the sections
we conclude that $h_{n}\in{\rm H}^{b}$ and by Lemma \ref{elem_ind}
they form a cube independent sequence. 

\end{proof}

With this cube independent sequence we perform the same procedure
as before. Take the quasi-cubic sets
\[
H_{n}=\left\{ h_{n}^{\varepsilon_{n}}\ldots h_{1}^{\varepsilon_{1}}:\ \varepsilon_{j}\in\{0,1\},1\le j\le n\right\} ,
\]
and form the convex combination 
\[
\eta_{2}=\sum_{n=1}^{\infty}\frac{c_{n}}{2}\left(\mathbf{u}_{H_{n}}+\mathbf{u}_{H_{n}^{-1}}\right)\mbox{ with }c_{n}=\frac{C_{\epsilon}n^{1+\epsilon}}{2^{n}}.
\]
Let $\mu_{2}=\frac{1}{2}({\bf u}_{S}+\eta_{2})$. By Proposition \ref{indep_tran}
and Proposition \ref{transience}, we have that $(G,\mu_{2})$ has
non-trivial Poisson boundary. 

The length of the elements $(h_{n})$ can be estimated using the substitution
matrices. We obtain the following:

\begin{cor}\label{Hbound}

For any $\epsilon>0$, there exists a non-degenerate symmetric probability
measure $\mu$ on $G=G_{012}$ of the form $\mu=\frac{1}{2}({\bf u}_{S}+\eta)$
of finite entropy and nontrivial Poisson boundary, where $\eta$ is
supported on the subgroup ${\rm H}^{b}$ and the tail decay of $\mu$
satisfies
\[
\mu\left(B(e,r)^{c}\right)\lesssim r^{-\frac{2}{\log_{2}\lambda_{1}}}\log^{1+\epsilon}r.
\]
The number $\lambda_{1}$ is the largest real root of the polynomial
$X^{3}-15X^{2}+44X-32$, $\frac{2}{\log_{2}\mathcal{\lambda}_{1}}\approx0.5700$.
As a consequence,
\[
v_{G}(n)\gtrsim\exp\left(n^{\frac{2}{\log_{2}\lambda_{1}}}/\log^{1+\epsilon}n\right).
\]

\end{cor}

\begin{proof}

For a word $w$ in $\{ab,ac,ad\}^{\ast}$, recall that ${\bf l}(\omega)$
records the number of occurrences of $ab$, $ac$, $ad$ as a column
vector. Then ${\bf l}(\sigma(w))=A{\bf l}(\omega)$ where the matrix
is
\[
A=\left[\begin{array}{ccc}
0 & 1 & 0\\
1 & 1 & 2\\
1 & 0 & 0
\end{array}\right].
\]
Recall the matrix $M$ that ${\bf l}(\zeta(w))=M{\bf l}(\omega)$.
Then 
\[
\left|\left(\zeta^{2}\circ\sigma\right)^{k}(abab)\right|=\left[\begin{array}{ccc}
2 & 2 & 2\end{array}\right]\left(M^{2}A\right)^{k}\left[\begin{array}{c}
2\\
0\\
0
\end{array}\right].
\]
Let $\lambda_{1}$ be the largest real eigenvalue of the matrix 
\[
M^{2}A=\left[\begin{array}{ccc}
6 & 5 & 4\\
2 & 5 & 4\\
2 & 3 & 4
\end{array}\right].
\]
In other words $\lambda_{1}$ is the largest real root of the characteristic
polynomial $X^{3}-15X^{2}+44X-32$. Numerically $\lambda_{1}\approx11.3809$.
It follows that 
\[
|h_{n}|\le C\lambda_{1}^{n/2}.
\]
Therefore from definition of $\eta_{2}$ we have that the tail of
$\mu_{2}=\frac{1}{2}({\bf u}_{S}+\eta_{2})$ satisfies
\[
\mu_{2}\left(B\left(e,\lambda_{1}^{n/2}\right)^{c}\right)\lesssim\frac{n^{1+\epsilon}}{2^{n}}.
\]
The statement follows.

\end{proof}

One cannot have a measure $\mu$ with the following property: $\mu=\frac{1}{2}(\nu+\eta)$,
where $\eta$ is supported on the subgroup ${\rm H}^{b}$ with smaller
germ group $\left\langle b\right\rangle $ and $\nu$ is of finite
support, $(G,\mu)$ has non-trivial Poisson boundary and $\eta$ has
finite $\alpha$-moment for $\alpha$ close to the critical constant
$c_{{\rm rt}}(G,{\rm St}_{G}(1^{\infty}))$. The main drawback is
that the action of any element in ${\rm H}^{b}$ fixes the digits
$x_{3k+1}$, $k\in\mathbb{N}$ as explained in Lemma \ref{Horbit}.
The orbit of $1^{\infty}$ under ${\rm H}^{b}$ is much smaller than
its orbit under $G$. Indeed by norm contracting calculations one
can bound from above the growth of $(\mathbb{Z}/2\mathbb{Z})\wr_{1^{\infty}\cdot{\rm H}^{b}}{\rm H}^{b}$
with respect to the induced word distance and show that a measure
$\eta$ on ${\rm H}^{b}$ with transient induced random walk on the
orbit of $o$ cannot have tail decay close to $r^{-\alpha_{0}}$. 

\section{Main construction of measures of non-trivial Poisson boundary on
Grigorchuk groups\label{sec:Main}}

In this section we prove our main result on existence of measures
on $G_{\omega}$ with non-trivial Poisson boundary and near optimal
tail decay, where $\omega$ satisfies the following assumption. 

\begin{notation}\label{D}

We say the string $\omega$ satisfies Assumption $({\rm Fr}(D))$,
where $D\in\mathbb{N}$, if for every integer $k\ge0$, the string
$\omega_{kD}\ldots\omega_{kD+D-1}$ contains at least one of the following
two substring: $\left\{ {\bf 201},{\bf 211}\right\} $. \textquotedbl Fr\textquotedbl{}
in the notation stands for frequency. Given a string $\omega$ that
satisfies Assumption $({\rm Fr}(D))$, define $(m_{k})_{k=0}^{\infty}$
and $I_{\omega}$ as follows:
\begin{itemize}
\item for each $k$ let $m_{k}$ be the smallest number in $\{0,\ldots,D-3\}$
such that 
\[
\omega_{kD+m_{k}}\omega_{kD+m_{k}+1}\omega_{kD+m_{k}+2}\in\left\{ {\bf 201},{\bf 211}\right\} ;
\]
 
\item let $I_{\omega}$ be the set $I_{\omega}:=\left\{ kD+m_{k}+3:\ k\ge0\right\} $. 
\end{itemize}
\end{notation}

Note that by definition $D\ge3$. The value of $D$ is not important,
as long as it is finite. Let $\pi$ be a permutation of letters $\{{\bf 0},{\bf 1},{\bf 2}\}$
and $\pi(\omega)$ be the string $(\pi(\omega))_{i}=\pi(\omega_{i})$.
Then by definition of Grigorchuk groups it is clear that $G_{\omega}$
is isomorphic to $G_{\pi(\omega)}$. We call $\omega\mapsto\pi(\omega)$
a renaming of letters. We may also take a finite shift of $\omega$
if the resulting string is more convenient. 

\begin{exa}

The following examples of $\omega$ satisfies Assumption $({\rm Fr}(D))$
with some finite $D$ depending on the sequence:
\begin{itemize}
\item Up to a renaming of letters $\{{\bf 0},{\bf 1},{\bf 2}\}$, every
periodic sequence $\omega$ which contains all three letters.
\item A sequence of the form $({\bf 012})^{i_{1}}{\bf 1}^{j_{1}}\ldots({\bf 012})^{i_{k}}{\bf 1}^{j_{k}}\ldots$
with all $i_{k},j_{k}$ uniformly bounded.
\item A word obtained by concatenating powers of ${\bf 201}$ and ${\bf 211}$. 
\end{itemize}
\end{exa}

\subsection{Distance bounds on the Schreier graph}

In this subsection we review some elementary facts about distances
on the orbital Schreier graph of $1^{\infty}$ under action of $G_{\omega}$. 

Let $o=1^{\infty}\in\partial\mathsf{T}$ and denote by $\mathcal{S}_{\omega}$
its \emph{Schreier graph} under the action of $G_{\omega}$: the vertex
set of $\mathcal{S}_{\omega}$ is $o\cdot G_{\omega}$ and two vertices
$x,y\in o\cdot G_{\omega}$ are connected by an edge labelled with
$s\in\{a,b_{\omega},c_{\omega},d_{\omega}\}$ if $y=x\cdot s$. The
Schreier graph $\mathcal{S}_{\omega}$ can be constructed by applying
global substitution rules, see for example \cite[Section 7]{Grigorchuk11}. 

Let $d_{\mathcal{S}_{\omega}}$ denote the graph distance on the Schreier
graph $\mathcal{S}_{\omega}$. Because the unlabelled Schreier graph
of $o$ does not depend on the sequence $\omega$, it follows that
the graph distance also doesn't depend on $\omega$. For this reason
we can omit reference to $\omega$ and write $d_{\mathcal{S}}$ for
the graph distance. It is convenient to read the distance $d_{\mathcal{S}}$
from the Gray code enumeration of the orbit $o\cdot G_{\omega}$.
Explicitly, for $x=x_{1}x_{2}\ldots\in\partial\mathsf{T}$, flip all
digits of $x$ to the ray $\check{x}_{1}\check{x}_{2}\ldots$ where
$\check{x}_{i}=1-x_{i}$. The Grey code of $x$ is $\bar{x}=\bar{x}_{1}\bar{x}_{2}\ldots$
where 
\[
\bar{x}_{i}=\check{x}_{1}+\ldots+\check{x}_{i}\mod2.
\]
Note that for $x$ cofinal with $1^{\infty}$, its Gray code $\bar{x}$
has only finitely many $1$'s . We regard such an $\bar{x}$ as an
element in $\{0\}\cup\mathbb{N}$ represented by a binary string.
Then on the Schreier graph of $o$, 
\[
d_{\mathcal{S}}(x,y)=\left|\bar{x}-\bar{y}\right|.
\]
The distance to $1^{\infty}$ is particularly easy to read because
the Gray code of the point $1^{\infty}$ is constant $0$. 

\begin{fact}

Let $x\in\partial\mathsf{T}$ be a point that is cofinal with $1^{\infty}$.
Let $n(x)=\max\{k:\ x_{k}=0\}$. Then
\[
2^{n(x)-1}\le d_{\mathcal{S}}(x,1^{\infty})\le2^{n(x)}-1.
\]

\end{fact}

\begin{proof}

We have $d_{\mathcal{S}}(x,1^{\infty})=|\bar{x}|$. Since the maximum
index of digit $0$ in $x$ is $n(x)$, the Grey code of $x$ is of
the form $u1000\ldots$ where $u$ is a prefix of length $n(x)-1$. 

\end{proof}

For general points a similar upper bound holds.

\begin{fact}

Let $x,y$ be two points cofinal with $1^{\infty}$. Denote by $\mathfrak{s}$
the shift on strings. Then 
\[
d_{\mathcal{S}}(x,y)\le2^{n}d_{\mathcal{S}}(\mathfrak{s}^{n}x,\mathfrak{s}^{n}y)+2^{n}-1.
\]

\end{fact}

\begin{proof}

Write $x$ and $y$ in Gray codes. The Grey code of $x$ is $\overline{x}=u\overline{s^{n}x}$,
where $u$ is some prefix of length $n$, similarly $\bar{y}=v\overline{\mathfrak{s}^{n}y}$,
where $v$ is some prefix of length $n$. It follows that 
\begin{align*}
d_{\mathcal{S}}(x,y) & =|\bar{x}-\bar{y}|\\
 & \le\left|u\overline{\mathfrak{s}^{n}x}-u\overline{\mathfrak{s}^{n}y}\right|+\left|u\overline{\mathfrak{s}^{n}y}-v\overline{\mathfrak{s}^{n}y}\right|\\
 & \le2^{n}d_{\mathcal{S}}(\mathfrak{s}^{n}x,\mathfrak{s}^{n}y)+2^{n}-1.
\end{align*}

\end{proof}

As a consequence we have the following estimate of displacement. Recall
that $\left|\cdot\right|_{G_{\omega}}$ denotes the word length in
the group $G_{\omega}$ equipped with generating set $\left\{ a,b_{\omega},c_{\omega},d_{\omega}\right\} $. 

\begin{lem}\label{translate}

Let $g\in G_{\omega}$ and $g=(g_{v})_{v\in\mathsf{L}_{n}}\tau$ be
its wreath recursion to level $n$. Then for $x=x_{1}x_{2}\ldots\in\partial\mathsf{T}$,
\[
d_{\mathcal{S}}\left(x,x\cdot g\right)\le2^{n}\left(\left|g_{x_{1}\ldots x_{n}}\right|_{G_{s^{n}\omega}}+1\right).
\]

\end{lem}

\begin{proof}

Write $x=vx'$, where $v$ is the prefix $v=x_{1}\ldots x_{n}$ and
$x'=\mathfrak{s}^{n}x$. Then $x\cdot g=(v\cdot\tau)x'\cdot g_{v}$.
Therefore 
\[
d_{\mathcal{S}}(x,x\cdot g)\le2^{n}d_{\mathcal{S}}(x',x'\cdot g_{v})+2^{n}-1\le2^{n}(|g_{v}|+1).
\]

\end{proof}

\subsection{Construction of the measures \label{subsec:elements}}

Throughout the rest of this section we assume that $\omega$ satisfies
Assumption $({\rm Fr}(D))$. The sketch for the first Grigorchuk group
in the Introduction and explicit descriptions in Example \ref{firstmain}
may help to understand the definitions. To avoid possible confusion
about indexing, keep in mind that the string $\omega\in\{{\bf 0},{\bf 1},{\bf 2}\}^{\infty}$
starts with $\omega_{0}$, $\omega=\omega_{0}\omega_{1}\ldots$ while
tree vertices are recorded as $v=v_{1}v_{2}\ldots$. 

First take the sequence of words obtained by substitutions:
\begin{equation}
g_{n}=\begin{cases}
\zeta_{\omega_{0}}\circ\ldots\zeta_{\omega_{n-1}}(ab_{\mathfrak{s}^{n}\omega}) & \mbox{ if }\omega_{n-1}\neq{\bf 2},\\
\zeta_{\omega_{0}}\circ\ldots\zeta_{\omega_{n-1}}(ac_{\mathfrak{s}^{n}\omega}) & \mbox{ if }\omega_{n-1}={\bf 2}.
\end{cases}\label{eq:gngeneral}
\end{equation}
The sequence $(g_{n})$ has the cube independence property by Lemma
\ref{elem_ind}. From its definition in (\ref{eq:gngeneral}), $g_{n}$
has $b$-germs if $\omega_{n-1}\neq{\bf 2}$ and $c$-germs if $\omega_{n-1}={\bf 2}$. 

Our goal is to construct a symmetric measure on $G_{\omega}$ such
that along the corresponding random walk trajectory $(W_{n})$, the
$\left\langle b\right\rangle $-coset of the germ $\Phi_{W_{n}}(1^{\infty})$
stabilizes. To this end we introduce modifications to these $g_{n}$
with $c$-germs.

The modifications are performed by taking conjugations of generators
$c_{\mathfrak{s}^{j}\omega}$. Recall the notation that $\iota\left(h,v\right)$
denotes the group element in the rigid stabilizer of $v$ and acts
as $h$ in the subtree rooted at $v$. Note that the element $\gamma_{\mathfrak{s}^{n}\omega}$,
where $\gamma$ is a letter in $\{b,c,d\}$, in general is not in
the rigid stabilizer of a level $n$ vertex. However, we can find
commutators of the form $\left[\gamma_{\mathfrak{s}^{n}\omega},a\right]$
in the rigid stabilizers. The following fact estimates length of such
elements and will be used repeatedly.

\begin{figure} \includegraphics[scale=1.2]{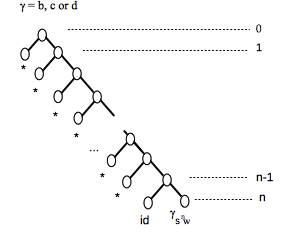} \caption{The letter $\gamma$ is  in $\{b,c,d\}$ such that $\omega_{n-1}(\gamma)=id$. The symbol $\ast$ means the section is either $a$ or $id$.} 
\end{figure}

\begin{figure} \includegraphics[scale=1]{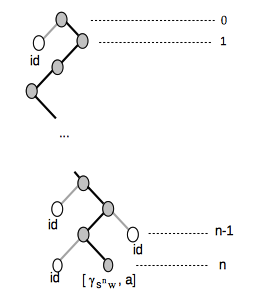} \caption{ The element $\left[\gamma_{\mathfrak{s}^{n}\omega},a\right]$ is in the level $n$ rigid stabilizer. By definition, the element $\iota\left(\left[\gamma_{\mathfrak{s}^{n}\omega},a\right],v\right)$ only has nontrivial section at $v$.}\label{fig:rigidportrait} 
\end{figure}

\begin{fact}\label{rigid}

For $n\in\mathbb{N}$, let $\gamma$ be the letter in $\{b,c,d\}$
be such that $\omega_{n-1}(\gamma)=id$. Then for any vertex $v\in\mathsf{L}_{n}$,
$\left[\gamma{}_{\mathfrak{s}^{n}\omega},a\right]$ is in the rigid
stabilizer ${\rm Rist}_{G_{\omega}}(v)$ and 
\[
\left|\iota\left(\left[\gamma_{\mathfrak{s}^{n}\omega},a\right],v\right)\right|\le2^{n+2}.
\]

\end{fact}

\begin{proof}

For each $n\in\mathbb{N}$, let $\gamma_{n}\in\{b,c,d\}$ be the letter
that $\omega_{n-1}(\gamma_{n})=id$. 

For each $n\ge2$, fix a choice of letter $y_{n-1}\in\{b,c,d\}$ satisfying
such that $\omega_{n-2}(y_{n-1})=id$ (that is, $y_{n-1}=\gamma_{n-1}$)
if $\omega_{n-2}\neq\omega_{n-1}$; and $\omega_{n-2}(y_{n-1})=a$
if $\omega_{n-2}=\omega_{n-1}$. For $n=1$, fix a choice of letter
$y_{0}$ such that $\omega_{0}(y_{0})=a$. Consider the Lysionok type
substitution $\sigma_{n}$ on $\{a,b,c,d\}^{\ast}$ given by 
\[
a\mapsto ay_{n-1}a,\ b\mapsto b,\ c\mapsto c,\ d\mapsto d.
\]
By definition of the substitution $\sigma_{n}$, we have that in one
step wreath recursion, $\sigma_{n}\left(\left[a,\gamma_{n}\right]\right)=\left(id,\left[a,\gamma_{n}\right]\right)$.

Recall that $\mathbf{F}$ is the free product $\left(\mathbb{Z}/2\mathbb{Z}\right)\ast\left(\mathbb{Z}/2\mathbb{Z\times\mathbb{Z}}/2\mathbb{Z}\right)=\left\langle a\right\rangle \ast\left(\left\langle b\right\rangle \times\left\langle c\right\rangle \right)$.
Denote by $K_{n}$ the normal closure of $[a,\gamma_{n}]$ in $\mathbf{F}$
and write $g^{h}$ for the conjugation $h^{-1}gh$. By direct calculation,
in the first case where $\omega_{n-2}\neq\omega_{n-1}$, we have $y_{n-1}=\gamma_{n-1}$
and $\sigma_{n}\left([a,\gamma_{n}]\right)=a\gamma_{n-1}a\gamma_{n}a\gamma_{n-1}a\gamma_{n}=[a,\gamma_{n-1}][\gamma_{n-1},a]^{\gamma_{n}}$;
while in the second case where $\omega_{n-2}=\omega_{n-1}$, we have
$\gamma_{n-1}=\gamma_{n}$ and $\sigma_{n}\left([a,\gamma_{n}]\right)=ay_{n-1}a\gamma_{n}ay_{n-1}a\gamma_{n}=\left[a,\gamma_{n}\right]^{y_{n-1}a}[a,\gamma_{n}]$.
It follows that for a word ${\bf w}\in K_{n}$, the image $\sigma_{n}({\bf w})$
is in $K_{n-1}$ and moreover, in one step wreath recursion, we have
$\sigma_{n}\left({\bf w}\right)=\left(id,{\bf w}\right)$.

Let $v=v_{1}v_{2}\ldots v_{n}$ be the address of $v$. Apply the
substitutions recursively: start with ${\bf w}_{0}=\gamma_{n}a\gamma_{n}a$,
where $\gamma_{n}$ is the letter in $\{b,c,d\}$ be such that $\omega_{n-1}(\gamma_{n})=id$
as in the statement. For $0\le j\le n-1$, set
\[
{\bf w}_{j+1}=\begin{cases}
\sigma_{n-j}({\bf w}_{j}) & \mbox{if }v_{n-j}=1,\\
a\sigma_{n-j}({\bf w}_{j})a & \mbox{if }v_{n-j}=0.
\end{cases}
\]
It is understood that reduction in $\mathbf{F}$, for instance $a^{2}=1$,
is always performed to words. For example $a(abac)^{2}a=\left(baca\right)^{2}$.
Then it is clear that $|w_{j}|=2^{2+j}$. 

We now verify that the word ${\bf w}_{n}$ obtained by substitutions
as above evaluates to the desired element $\iota\left(\left[\gamma_{\mathfrak{s}^{n}\omega},a\right],v\right)$
in the group $G_{\omega}$. We start with ${\bf w}_{0}=[\gamma_{n},a]\in K_{n}$.
Then by the property of the substitutions above, for $0\le j\le n-1$,
under one step wreath recursion we have 
\[
{\bf w}_{j+1}=\begin{cases}
(id,{\bf w}_{j}) & \mbox{if }v_{n-j}=1,\\
({\bf w}_{j},id) & \mbox{if }v_{n-j}=0.
\end{cases}
\]
Therefore perform the recursion to ${\bf w}_{n}$ down $n$ levels,
we have that under the projection $\pi_{\omega}:\mathbf{F}\to G_{\omega}$
given by $a\mapsto a$, $b\mapsto b_{\omega}$, $c\mapsto c_{\omega}$
and $d\mapsto d_{\omega}$, it evaluates to $\iota\left(\left[\left(\gamma_{n}\right)_{\mathfrak{s}^{n}\omega},a\right],v\right)$. 

\end{proof}

For $k,n\in\mathbb{N}$ both divisible by $D$, define the set $\mathsf{W}_{k}^{n}$
of vertices of depth $k$ such that $u_{i}=1$ except for those $i$
with $n+i\in I_{\omega}$,
\begin{equation}
\mathsf{W}_{k}^{n}:=\left\{ u\in\mathsf{L}_{k}:\ u=u_{1}\ldots u_{k},\ u_{i}=1\mbox{ if }n+i\notin I_{\omega}\right\} .\label{eq:Wk}
\end{equation}
Recall that the set $I_{\omega}$ is defined in Notation \ref{D}.
The cardinality of the set $\mathsf{W}_{k}^{n}$ is $2^{k/D}$. 

Given an integer $j$, denote by $\bar{j}$ the residue of $j$ mod
$D$, $\bar{j}\in\{0,\ldots,D-1\}$. We now define the set $\mathsf{V}_{k}^{j}$
which will be the index set of a collection of conjugates of the element
$c_{s^{j}\omega}$ in $G_{j}=G_{\mathfrak{s}^{j}\omega}$. Let $k$
be an integer divisible by $D$, write $\ell(k,j)=\frac{1}{D}\left(j+D-\bar{j}+k\right)$
and define $\mathsf{V}_{k}^{j}$ the collection of vertices 
\begin{equation}
\mathsf{V}_{k}^{j}:=\left\{ 1^{D-\bar{j}}u1^{m_{\ell(j,k)}+2}0:\ u\in\mathsf{W}_{k}^{j+D-\bar{j}}\right\} ,\label{eq:Vk}
\end{equation}
where $m_{\ell}$ is defined in Notation \ref{D}, $m_{\ell}\in\{0,\ldots,D-3\}$.
In words, $\mathsf{V}_{k}^{j}$ is obtained from the set $\mathsf{W}_{k}^{j+D-\bar{j}}$
by appending some $1$'s as prefix to match with $D$-blocks and adding
$1^{m_{\ell(j,k)}+2}0$ at the end. It is important that any $v\in\mathsf{V}_{k}^{j}$
ends with digit $0$. The cardinality of the set $V_{k}^{j}$ is $2^{k/D}$,
the same as the cardinality of $\mathsf{W}_{k}^{j+D-\bar{j}}$. 

Given a vertex $v=v_{1}\ldots v_{k'}\in\mathsf{V}_{k}^{j}$, where
$k'=D-\bar{j}+k+m_{\ell(j,k)}+3$ denotes the length of $v$, take
the following sequence of elements in $G_{j}=G_{\mathfrak{s}^{j}\omega}$:
\begin{equation}
h_{i}^{v}:=\begin{cases}
id & \mbox{if }v_{i}=1\\
\iota\left(\left[b_{\mathfrak{s}^{j+i-2}\omega},a\right],v_{1}\ldots v_{i-2}\right) & \mbox{if }v_{i}=0.
\end{cases}\label{eq:hv}
\end{equation}
By definition of $\mathsf{V}_{k}^{j}$ in (\ref{eq:Vk}), if $v_{i}=0$
then the string $\omega_{j+i-3}\omega_{j+i-2}\omega_{j+i-1}$ is either
${\bf 201}$ or ${\bf 211}$. Since the digit $\omega_{j+i-3}$ in
this case must be ${\bf 2}$, by Fact \ref{rigid}, $\left[b_{\mathfrak{s}^{j+i-2}\omega},a\right]$
is in the rigid stabilizer of the vertex $v_{1}\ldots v_{i-2}$ in
$G_{j}$. The letter following $\omega_{j+i-3}$ is different from
${\bf 2}$, therefore $\left[b_{\mathfrak{s}^{j+i-2}\omega},a\right]$
acts as $b_{\mathfrak{s}^{j+i-1}\omega}a$ on the right subtree. Since
$h_{i}^{v}$ is in the rigid stabilizer of $v_{1}\ldots v_{i-2}$,
it is easy to verify recursively that by definition of the elements
$h_{i}^{v}$ in (\ref{eq:hv}), we have
\[
1^{\infty}\cdot h_{1}^{v}h_{2}^{v}\ldots h_{k'}^{v}=v1^{\infty}.
\]

\begin{figure} \includegraphics[scale=0.5]{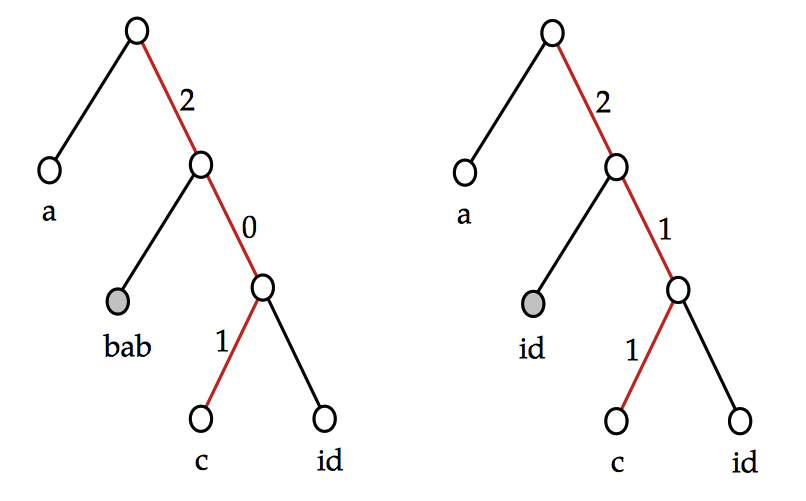} \caption{  A conjugated element for $v=110$, the only label $c$ is at $v$ (the corresponding ray is red). The picture shows a turn over $\bf{201}$ on the left and a turn over $\bf{211}$ on the right, in general notation.}\label{fig:201} \end{figure}

For each $v\in\mathsf{V}_{k}^{j}$, take the conjugation of $c_{\mathfrak{s}^{j}\omega}$
in $G_{j}$ 
\begin{equation}
\mathfrak{c}_{j}^{v}:=\left(h_{1}^{v}h_{2}^{v}\ldots h_{k'}^{v}\right)^{-1}c_{\mathfrak{s}^{j}\omega}\left(h_{1}^{v}h_{2}^{v}\ldots h_{k'}^{v}\right).\label{eq:ckv}
\end{equation}
We have the following description of the element $\mathfrak{c}_{j}^{v}$.
The portrait of an element along a ray segment is explained in Subsection
\ref{subsec:portrait}.

\begin{lem}\label{ckgerm}

Let $v\in\mathsf{V}_{k}^{j}$, $D|k$ and write $k'=|v|$ the length
of $v$. The vertex $v$ is fixed by $\mathfrak{c}_{j}^{v}$ and along
the ray segment $v$ the only nontrivial sections of $\mathfrak{c}_{j}^{v}$
are at 
\[
\{v_{1}\ldots v_{i}0:\ \omega_{j+i}\neq{\bf 1},\ 0\le i\le k'-1\}.
\]
Among these, for $i\le k'-2$ and $\omega_{j+i}\neq{\bf 1}$, if $v_{i+2}=0$
then the section at $v_{1}\ldots v_{i}0$ is $b_{\mathfrak{s}^{j+i+1}\omega}ab_{\mathfrak{s}^{j+i+1}\omega}$,
otherwise the section is $a$. For $i=k'-1$, at level $k'$, the
only non-trivial section is $c_{\mathfrak{s}^{j+k'}\omega}$ at $v$. 

\end{lem}

\begin{proof}

For those indices $i$ such that $v_{i}=0$, conjugation by $h_{i}^{v}$
is nontrivial only in the subtree rooted at $v_{1}\ldots v_{i-2}$,
where the section is conjugated by $\left[b_{\mathfrak{s}^{j+i-2}\omega},a\right]$.
Effect of conjugation by $\left[b_{\mathfrak{s}^{j+i-2}\omega},a\right]$
is drawn explicitly in Figure \ref{fig:201} with subscripts omitted.
The portrait of $\mathfrak{c}_{j}^{v}$ is obtained by applying these
conjugations one by one (where every non-trivial conjugation corresponding
to a $0$ in $v$ results in a turn illustrated in the pictures). 

\end{proof}

\begin{rem}

The reason we need ${\bf 201}$ or ${\bf 211}$ in Assumption $({\rm Fr}(D))$
to perform these conjugations is the following: the digit ${\bf 2}$
is needed for $[b,a]$ to be in the rigid stabilizer at the corresponding
level, the second digit different from ${\bf 2}$ implies $[b,a]$
acts as $ab,ba$ on the next level, and the last digit ${\bf 1}$
implies that the sections of $c$ that gets swapped are $c$ and $id$.
As a consequence, $\mathfrak{c}_{j}^{v}$ fixes the ray $1^{\infty}$
and vertices of the form $1^{m}0$. It is possible that the arguments
can work through under the weaker assumption that one can find ${\bf 20}$
or ${\bf 21}$ in every $D$-block, but it simplifies calculations
to assume the third digit is ${\bf 1}$. 

\end{rem}

\begin{figure} 
\includegraphics[scale=0.8]{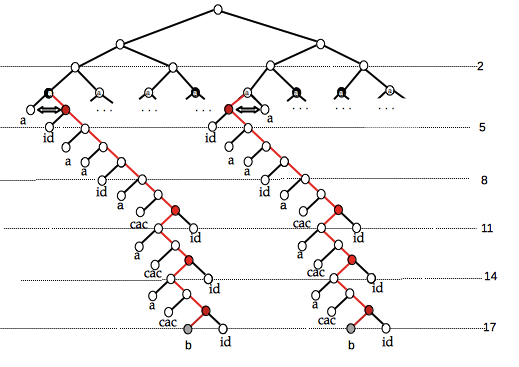} 
\caption{Portrait of the element $\tilde{g}_j^v$ with $j=3$ and $v=11(111)(110)(110)(110)$ in the first Grigorchuk group, with usual notation. Because of space limitations the picture does not show the full portrait: in places marked as ... there are two types of portraits: the portrait in the subtree at a black  vertex is the same as at $000$; and in the subtree at a white  vertex, the portrait is the same as at $100$.}
\label{largeportrait} 
\end{figure}

Next we apply the substitutions $\zeta_{x}$ to $a\mathfrak{c}_{j}^{v}$,
which is an element in $G_{j}$, to obtain an element in $G_{\omega}$.
Given $j$ and a vertex $v\in\mathsf{V}_{k}^{j}$ where $D|k$, define
\begin{equation}
\tilde{g}_{j}^{v}:=\zeta_{\omega_{0}}\circ\ldots\circ\zeta_{\omega_{j-1}}\left(a\mathfrak{c}_{j}^{v}\right).\label{eq:gjv}
\end{equation}
Figure \ref{largeportrait} illustrates the portrait of an element
of this type in the first Grigorchuk group. element $\tilde{g}_{j}^{v}$,
with $j=3$ and $v=11111110110110$, in the first Grigorchuk group
$G_{012}$. 

\begin{lem}\label{lemgj}

The element $\tilde{g}_{j}^{v}$ is well defined (the substitutions
$\zeta_{\omega_{i}}$'s can be applied). It is in the $j$-th level
stabilizer and when $\omega_{j-1}\neq{\bf 1}$, the sections of $\tilde{g}_{j}^{v}$
at vertices of $\mathsf{L}_{j}$ are in $\left\{ a\mathfrak{c}_{j}^{v},\mathfrak{c}_{j}^{v}a\right\} $
. 

\end{lem}

\begin{proof}

In $G_{j}$, each non-trivial $h_{i}^{v}$ is defined to be $\iota\left(\left[b_{\mathfrak{s}^{j+i-2}\omega},a\right],v_{1}\ldots v_{i-2}\right)$.
As explained in the proof of Fact \ref{rigid}, the element $h_{i}^{v}$
can be represented as a reduced word in $\left\{ ab,ac,ad\right\} ^{\ast}$,
under the projection $\pi_{\mathfrak{s}^{j}\omega}:\mathbf{F}\to G_{\mathfrak{s}^{j}\omega}$.
It follows that after applying the conjugations to $c$, the element
$\mathfrak{c}_{j}^{v}$ can be represented either by a reduced word
of the form $wa$, where $w\in\{ab,ac,ad\}^{\ast}$, or by $\gamma w$
where $\gamma\in\{b,c,d\}$ and $w\in\{ab,ac,ad\}^{\ast}$. In first
case $a\mathfrak{c}_{j}^{v}$ can be represented by a word in $\{ba,ca,da\}^{\ast}$,
in the second case it can be represented by a word in $\{ab,ac,ad\}^{\ast}$.
In either case the substitutions $\zeta_{\omega_{i}}$ can be applied
to $a\mathfrak{c}_{j}^{v}$. The second claim follows from (\ref{eq:subformula}). 

\end{proof}

We now summarize what we have defined so far. For any integer $j$
(indicating the level) and integer $k$ divisible by $D$ (parametrizing
the length of the segment where conjugations are performed), we have
defined an index set $V_{k}^{j}$ as in (\ref{eq:Vk}). For each string
$v\in V_{k}^{j}$, we take the conjugated element $\mathfrak{c}_{j}^{v}$
in $G_{\mathfrak{s}^{j}\omega}$ as defined in (\ref{eq:ckv}) and
use the substitutions $\zeta_{x}$ to obtain an element $\tilde{g}_{j}^{v}$
in $G_{\omega}$ as defined in (\ref{eq:gjv}). Keep in mind the picture
that the portrait of $\mathfrak{c}_{j}^{v}$ is directed along $v$,
see Figure \ref{fig:portaitck} and at each vertex in $\mathsf{L}_{j}$,
the section of $\tilde{g}_{j}^{v}$ is either $a\mathfrak{c}_{j}^{v}$
or $\mathfrak{c}_{j}^{v}a$, see Figure \ref{largeportrait}.

Let $(k_{n})$ be a sequence of increasing positive integers divisible
by $D$ to be determined later, $k_{n}\ll n$. In what follows $n$
is always an integer divisible by $D$. For each $j\in\{1,\ldots,n\}$,
take the following sets $\mathfrak{F}_{j,n}$ of group elements in
$G_{\omega}$, 
\begin{itemize}
\item For an index $j$ such that $\omega_{j-1}={\bf 2}$ and $n-k_{n}<j\le n$,
define $\mathfrak{F}_{j,n}$ to be the following set of elements $\tilde{g}_{j}^{v}=\zeta_{\omega_{0}}\circ\ldots\circ\zeta_{\omega_{j-1}}\left(a\mathfrak{c}_{j}^{v}\right)$,
\begin{equation}
\mathfrak{F}_{j,n}:=\left\{ \tilde{g}_{j}^{v}|\ v\in\mathsf{V}_{2k_{n}}^{j}\mbox{ and }\left|1^{\infty}\wedge v\right|\ge n-j+D\right\} ,\label{eq:Fjn}
\end{equation}
where $u\wedge v$ denotes the longest common prefix of two rays $u$
and $v$. The set $\mathfrak{F}_{j,n}$ is indexed by the subset of
$\mathsf{V}_{2k_{n}}^{j}$ which consists of vertices with prefix
$1^{n-j+D}$. The cardinality of $\mathfrak{F}_{j,n}$ is $2^{\frac{2k_{n}-(n-j+\bar{j})}{D}}$
in this case. 
\item Otherwise, that is, for indices $j\in\{1,\ldots,n\}$ other than those
with $\omega_{j-1}={\bf 2}\mbox{ and }n-k_{n}<j\le n$, keep the single
element $g_{j}$ and set $\mathfrak{F}_{j,n}=\{g_{j}\}$, where $g_{j}$
is defined in (\ref{eq:gngeneral}). 
\end{itemize}
Informally, for those indices $j$ within $k_{n}$ distance to $n$
with \textquotedbl bad\textquotedbl{} germs, we replace $g_{j}$ by
a set $\mathfrak{F}_{j,n}$ of conjugated elements parametrized by
strings; while for those $g_{j}$'s with \textquotedbl good\textquotedbl{}
germs, we keep them as they are.

We proceed to construct a measure $\upsilon_{n}$ on $G_{\omega}$
which will replace the $n$-quasi-cubic measure. A random element
with distribution $\upsilon_{n}$ can be obtained as follows. Take
independent random variables $\{\epsilon_{i},\gamma_{i}|\ 1\le i\le n\}$
where for each $i$, $\epsilon_{i}$ is uniform on $\{0,1\}$ and
$\gamma_{i}$ is uniform on the set $\mathfrak{F}_{i,n}$. Then the
group element $\gamma_{n}^{\epsilon_{n}}\ldots\gamma_{1}^{\epsilon_{1}}$
has distribution $\upsilon_{n}$. We refer to $\upsilon_{n}$ as a
\emph{uniformised $n$-quasi-cubic measure}, it depends on the parameters
$(k_{n})$. The measure $\upsilon_{n}$ is similar to an $n$-quasi-cubic
measure considered in earlier sections, but with an extra layer of
randomness in the choice of the cube independent sequence $\gamma_{1},\ldots\gamma_{n}$. 

More formally, the distribution $\upsilon_{n}$ is defined as follows.
Take the direct product
\begin{equation}
\mathfrak{F}_{n}=\prod_{i=1}^{n}\mathfrak{F}_{i,n}\label{eq:ucube}
\end{equation}
and the product of the hypercube $\{0,1\}^{n}$ with $\mathfrak{F}_{n}$,
\begin{equation}
\varLambda_{n}=\{0,1\}^{n}\times\mathfrak{F}_{n}.\label{eq:lambdan}
\end{equation}
Define $\theta_{n}:\varLambda_{n}\to G_{\omega}$ to be the map 
\[
\theta_{n}\left(\left(\boldsymbol{\epsilon},\boldsymbol{\gamma}\right)\right)=\gamma_{n}^{\epsilon_{n}}\ldots\gamma_{1}^{\epsilon_{1}},
\]
where $\boldsymbol{\epsilon}=\left(\epsilon_{1},\ldots,\epsilon_{n}\right)\in\{0,1\}^{n}$
and $\boldsymbol{\gamma}=\left(\gamma_{1},\ldots,\gamma_{n}\right)$,
$\gamma_{i}\in\mathfrak{F}_{i,n}$. Take the measure $\upsilon_{n}$
on $G_{\omega}$ to be the push-forward of the uniform measure $\mathbf{u}_{\varLambda_{n}}$
under $\theta_{n}$, that is 
\begin{equation}
\upsilon_{n}(g)=\frac{\left|\{\left(\boldsymbol{\epsilon},\boldsymbol{\gamma}\right)\in\varLambda_{n}:\theta_{n}\left(\left(\boldsymbol{\epsilon},\boldsymbol{\gamma}\right)\right)=g\}\right|}{|\varLambda_{n}|}.\label{eq:zetan}
\end{equation}

For the purpose of symmetrization, set the measure $\check{\upsilon}_{n}$
to be
\[
\check{\upsilon}_{n}(g)=\upsilon_{n}\left(g^{-1}\right).
\]

Finally, for $\beta\in(0,1)$, take the convex combination of the
measures 
\begin{equation}
\mu_{\beta}=\frac{1}{2}\mathbf{u}_{S}+\frac{1}{2}\sum_{n\in\mathbb{N},D|n}C_{\beta}2^{-n\beta}\left(\upsilon_{n}+\check{\upsilon}_{n}\right),\label{eq:eta3}
\end{equation}
where ${\bf u}_{S}$ is the uniform measure on the generating set
$\{a,b_{\omega},c_{\omega},d_{\omega}\}$, $C_{\beta}>0$ is the normalization
constant such that $\mu_{\beta}$ is a probability measure. Note that
although suppressed in the notations, the measure $\mu_{\beta}$ depends
on the sequence $(k_{n})$. 

\begin{exa}\label{firstmain}

We explain the definitions on the first Grigorchuk group $G=G_{{\bf 012}}$,
in the usual notation as in (\ref{eq:un}). The correspondence between
two systems of notations (usual and general) on $G$ is explained
in (\ref{eq:identification}). 

The defining string of $G$ is $({\bf 012})^{\infty}={\bf 01}({\bf 201})^{\infty}$,
with a shift of two digits it satisfies Assumption ${\rm Fr}(D)$,
$D=3$. In Notation \ref{D}, for the shifted string $\omega=({\bf 201})^{\infty}$,
we have $m_{k}=0$ and $I_{\omega}=3\mathbb{N}$. 

On $G$ the sequence $(g_{n})$ defined in formula (\ref{eq:gngeneral})
is

\[
g_{j}=\begin{cases}
\zeta^{j}(ac) & \mbox{if }j\equiv0,1\mod3,\\
\zeta^{j}(ad) & \mbox{if }j\equiv2\mod3,
\end{cases}
\]
in the usual notation for the first Grigorchuk group. This sequence
has appeared in Subsection \ref{subsec:critical} and Section \ref{sec: growth-first}.
Among them those $g_{j}$ with $j\equiv0\mod3$ have $c$-germs, for
which we will perform the conjugations.

In what follows $k$ is divisible by $3$ and $n\equiv2\mod3$ (where
$2$ comes from the shift of two digits from $({\bf 012})^{\infty}$
to $({\bf 201})^{\infty}$). The set $\mathsf{W}_{k}^{n}$ defined
in (\ref{eq:Wk}) is 
\[
\mathsf{W}_{k}^{n}=\left\{ u:\ |u|=k\mbox{ and }u_{3j+1}=u_{3j+2}=1\mbox{ for all }j\ge0\right\} .
\]
In other words, $\mathsf{W}_{k}^{n}$ consists of vertices of length
$k$ that are concatenations of segments $111$ and $110$. For $j\equiv0\mod3$,
the set $\mathsf{V}_{k}^{j}$ defined in (\ref{eq:Vk}) is
\[
\mathsf{V}_{k}^{j}=\left\{ 11u110:\ u\in\mathsf{W}_{k}^{j+2}\right\} ,
\]
that is vertices of the form $11\underbrace{(11\ast)\ldots(11\ast)}_{k\ {\rm digits}}110$,
where at $\ast$ the digit can be $0$ or $1$. The cardinality of
the set $\mathsf{V}_{k}^{j}$ is $2^{k/3}$. 

We remark that the set $\mathsf{V}_{k}^{j}$ is contained in the ${\rm H}^{b}$-orbit
of the vertex $1^{2+k+3}$, where ${\rm H}^{b}$ is the subgroup of
$G$ which consists of elements with only $\left\langle b\right\rangle $-germs,
see Figure \ref{orbitHb}.

\begin{figure} \includegraphics[scale=0.6]{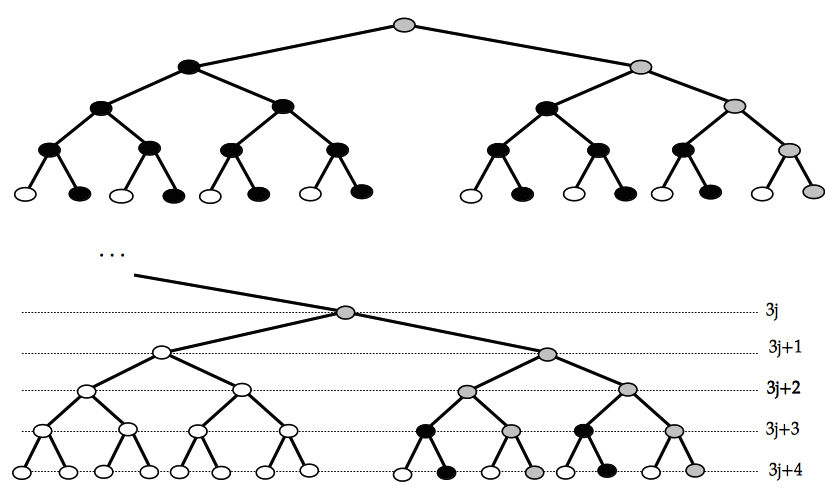} \caption{  Colors grey and black mark vertices that are in the orbit of the rightmost ray under $\rm{H}^b$. These are vertices of the form $(\ast)(\ast\ast1)(\ast\ast1)\dots$. The subsets of grey vertices on levels $\equiv 2\mod 3$  show vertices that are in the indexing set for conjugaction in the definition of $\mathfrak{c}^v$. They are of the form $(11)(11\ast)(11\ast)\dots$.}\label{orbitHb} \end{figure}

Figure \ref{fig:big} draws the portrait of a conjugated element $\mathfrak{c}^{v}=\mathfrak{c}_{j}^{v}$
indexed by $v\in\mathsf{V}_{k}^{j}$ as defined in (\ref{eq:ckv}).
Note that since $({\bf 012})^{\infty}$ is of period $3$, for all
$j\equiv0\mod3$, $\mathfrak{c}_{j}^{v}$ is the same element, thus
we can omit reference to $j$. 

For $n\equiv2\mod3$, take a parameter $k_{n}$ divisible by $3$.
In the sequence $(g_{j})_{j=1}^{n}$, for these $j\le n-k_{n}$ or
$j\equiv1,2\mod3$, let $\mathfrak{F}_{j,n}=\{g_{j}\}$; for these
$j$ such that $j\equiv0\mod3$ and $n-k_{n}<j\le n$, replace $g_{j}$
by the set $\mathfrak{F}_{j,n}$ defined in (\ref{eq:Fjn}), which
is
\[
\mathfrak{F}_{j,n}=\left\{ \zeta^{j}(a\mathfrak{c}_{j}^{v}):\ v=11\underbrace{(11\ast)\ldots(11\ast)}_{2k_{n}\ {\rm digits}}110\mbox{ and }v\mbox{ has prefix }1^{n-j+3}\right\} .
\]
The uniformised $n$-quasi-cubic measure $\upsilon_{n}$ can be described
as follows. For $1\le j\le n$, independently choose each $\gamma_{j}$
uniformly from the set $\mathfrak{F}_{j,n}$, and choose an independent
random variable $(\epsilon_{1},\ldots,\epsilon_{n})$ uniformly from
$\{0,1\}^{n}$, then the random group element $\gamma_{n}^{\epsilon_{n}}\ldots\gamma_{1}^{\epsilon_{n}}$
has distribution $\upsilon_{n}$.

Finally the measure $\mu_{\beta}$ defined with parameter sequence
$(k_{n})$ in (\ref{eq:eta3}) is a convex combination of ${\bf u}_{S}$,
$S=\{a,b,c,d\}$ and the uniformised $n$-quasi-cubic measures $\upsilon_{n}$
and their inverses $\check{\upsilon}_{n}$, over all $n\equiv2\mod3$. 

\end{exa}

\subsection{Non-triviality of the Poisson boundary }

The goal of this subsection is to show that with some appropriate
choices of $\beta$ and $A$, the measure $\mu_{\beta}$ defined in
(\ref{eq:eta3}) with $k_{n}=A\left\lfloor \log_{2}n\right\rfloor $
has non-trivial Poisson boundary. 

Recall that as explained in Example \ref{ggerms}, the isotropy group
of the groupoid of germs of $G_{\omega}\curvearrowright\partial\mathsf{T}$
at $1^{\infty}$ is isomorphic to $\mathbb{Z}/2\mathbb{Z}\times\mathbb{Z}/2\mathbb{Z}=\left\langle b\right\rangle \times\left\langle c\right\rangle $.
An element $g\in G_{\omega}$ has only $\left\langle b\right\rangle $-germs
if and only if there is a level $n$ such that all sections $g_{v}$
are in $\left\{ id,a,b_{\mathfrak{s}^{n}\omega}\right\} $, $v\in\mathsf{L}_{n}$.
Let $H=\left\langle b\right\rangle <\left(\mathcal{G}_{\omega}\right)_{o}$
and $\mathcal{H}^{b}$ be the set of germs that are either trivial
or $b$ as in (\ref{eq:H_v}). To apply Proposition \ref{transience},
we need to verify that the summation
\[
\sum_{x\in o\cdot G}{\bf G}_{P_{\mu_{\beta}}}(o,x)\mu_{\beta}\left(\left\{ g\in G:\ (g,x)\notin\mathcal{H}^{b}\right\} \right)
\]
is finite, where $P_{\mu_{\beta}}$ is the transition kernel induced
by $\mu_{\beta}$ on the orbit $o\cdot G$. To this end we show separately
an upper estimate on the Green function and a bound on the $\mu_{\beta}$-weights
in the summation.

\begin{prop}\label{green}

Let $\beta\in(1-\frac{1}{D},1)$ and $k_{n}=A\left\lfloor \log_{2}n\right\rfloor $.
For any $\epsilon>0$, there exists a constant $C=C(\beta,D,\epsilon)>0$
such that for any $x,y\in1^{\infty}\cdot G$, the Green function satisfies
\[
\mathbf{G}_{P_{\mu_{\beta}}}(x,y)\le C\left(\frac{d_{\mathcal{S}}(x,y)}{\log_{2}^{2A}d_{\mathcal{S}}(x,y)}\right)^{\beta-1}(\log_{2}d_{\mathcal{S}}(x,y))^{-\frac{1-\beta}{1+\beta}\left(\frac{2A}{D}-\epsilon\right)}.
\]

\end{prop} 

Proposition \ref{green} is proved in Subsection \ref{subsec:greenupper},
it is a consequence of the off-diagonal upper bound for transition
probabilities in \cite{BGK}. For the $\mu_{\beta}$-weights in the
summation, we have the following.

\begin{prop}\label{zetasum}

Let $\upsilon_{n}$ be defined as in (\ref{eq:zetan}) and $o=1^{\infty}$.
Let $f:\mathbb{R}_{+}\to\mathbb{R}_{+}$ be a non-increasing function
such that there exists constants $D'\in(D,\infty)$, 
\[
\frac{f\left(2x\right)}{f(x)}\ge2^{-\frac{1}{D'}}\mbox{ for all }x\ge1.
\]
Then there exists a constant $C=C(D')<\infty$ such that 

\[
\sum_{x\in o\cdot G}f\left(d_{\mathcal{S}}(o,x)\right)\upsilon_{n}\left(\left\{ g\in G:\ (g,x)\notin\mathcal{H}^{b}\right\} \right)\le C2^{n}f\left(2^{n+2k_{n}}\right).
\]
The same inequality holds with $\upsilon_{n}$ replaced by $\check{\upsilon}_{n}$. 

\end{prop}

Proposition \ref{zetasum} is proved in Subsection \ref{subsec:weights}.
With these two estimates we prove that for some appropriate choice
of parameters, along the random walk trajectory the $\left\langle b\right\rangle $-coset
of the germ at $1^{\infty}$ stabilizes. 

\begin{thm}\label{stabilization}

Let $\omega$ be a string satisfying Assumption $({\rm Fr}(D))$.
Then for $\beta\in\left(1-\frac{1}{D},1\right)$ and integer $A>\frac{D(1+\beta)}{2(1-\beta)}$
divisible by $D$, the measure $\mu_{\beta}$ on $G_{\omega}$ defined
as in (\ref{eq:eta3}) with $k_{n}=A\left\lfloor \log_{2}n\right\rfloor $
has non-trivial Poisson boundary.

\end{thm}

\begin{proof}[Proof of Theorem \ref{stabilization} given Prop. \ref{green} and \ref{zetasum}]

We show that the summation
\[
\sum_{x\in o\cdot G}{\bf G}_{P_{\mu_{\beta}}}(o,x)\mu_{\beta}\left(\left\{ g\in G:\ (g,x)\notin\mathcal{H}^{b}\right\} \right)
\]
is finite, then non-triviality of the Poisson boundary $\left(G_{\omega},\mu_{\beta}\right)$
follows from Proposition \ref{transience}. 

Let $\epsilon>0$ be a small constant such that $\frac{1-\beta}{1+\beta}\left(\frac{2A}{D}-\epsilon\right)>1$.
By Proposition \ref{green}, there is a constant $C<\infty$ only
depending on $\beta,D,\epsilon$ such that 
\[
\mathbf{G}_{P_{\mu_{\beta}}}(x,y)\le C\left(\frac{d_{\mathcal{S}}(x,y)}{\log_{2}^{2A}d_{\mathcal{S}}(x,y)}\right)^{\beta-1}(\log_{2}d_{\mathcal{S}}(x,y))^{-\frac{1-\beta}{1+\beta}\left(\frac{2A}{D}-\epsilon\right)}.
\]
Let $f(s)=C\left(\frac{s}{\log_{2}^{2A}s}\right)^{\beta-1}\left(\log_{2}s\right)^{-\frac{1-\beta}{1+\beta}\left(\frac{2A}{D}-\epsilon\right)}.$
Then $f$ satisfies the assumption of Proposition \ref{zetasum} with
$D'=\frac{1}{1-\beta}>D$. By Proposition \ref{zetasum}, 
\begin{align*}
 & \sum_{x\in o\cdot G}{\bf G}_{P_{\mu_{\beta}}}(o,x)\frac{\upsilon_{n}+\check{\upsilon}_{n}}{2}\left(\left\{ g\in G:\ (g,x)\notin\mathcal{H}^{b}\right\} \right)\\
\le & \sum_{x\in o\cdot G}f(d_{\mathcal{S}}(o,x))\frac{\upsilon_{n}+\check{\upsilon}_{n}}{2}\left(\left\{ g\in G:\ (g,x)\notin\mathcal{H}^{b}\right\} \right)\\
\le & C'2^{n}f\left(2^{n+2k_{n}}\right),
\end{align*}
where $C'$ is a constant depending only on $\beta$. Summing up this
estimate to $\mu_{\beta}$, 
\begin{align*}
 & \sum_{x\in o\cdot G}{\bf G}_{P_{\mu_{\beta}}}(o,x)\mu_{\beta}\left(\left\{ g\in G:\ (g,x)\notin\mathcal{H}^{b}\right\} \right)\\
\le & C_{\beta}\sum_{D|n}\sum_{x\in o\cdot G}f(d_{\mathcal{S}}(o,x))2^{-n\beta}\frac{\upsilon_{n}+\check{\upsilon}_{n}}{2}\left(\left\{ g\in G:\ (g,x)\notin\mathcal{H}^{b}\right\} \right)\\
\le & C_{\beta}C'\sum_{D|n}2^{-n\beta}2^{n}f\left(2^{n+2k_{n}}\right)\\
\le & C_{\beta}C'C\sum_{D|n}2^{n(1-\beta)}\left(\frac{2^{n+2k_{n}}}{(n+2k_{n})^{2A}}\right)^{\beta-1}\left(n+2k_{n}\right)^{-\frac{1-\beta}{1+\beta}\left(\frac{2A}{D}-\epsilon\right)}\\
\le & C_{\beta}C'C\sum_{n=1}^{\infty}n^{-\frac{1-\beta}{1+\beta}\left(\frac{2A}{D}-\epsilon\right)}.
\end{align*}
In the last step we plugged in $k_{n}=A\left\lfloor \log_{2}n\right\rfloor $.
Since the constant $A$ is assumed to satisfy $\frac{1-\beta}{1+\beta}\left(\frac{2A}{D}-\epsilon\right)>1$,
it follows that the summation is finite. 

\end{proof}

The rest of the section is devoted to the proofs of Proposition \ref{green}
(in Subsection \ref{subsec:greenupper}) and Proposition \ref{zetasum}
(in Subsection \ref{subsec:weights}), with some preparation given
in Subsection \ref{subsec:prep}. 

\subsection{Estimates on the induced transition kernel\label{subsec:prep}}

In this subsection we give some bounds on the transition kernel $P_{\mu_{\beta}}$,
which will be used to derive Green function upper bounds. We use definitions
and notations from Subsection \ref{subsec:elements}.

By the definition of the set $\mathfrak{F}_{j,n}$ in (\ref{eq:Fjn})
we have the following:

\begin{fact}\label{same}

Let $j$ be an index such that $\omega_{j-1}={\bf 2}$ and $n-k_{n}<j\le n$.
Then on the finite level $n+D+1$, the action of any element $\tilde{g}_{j}^{v}\in\mathfrak{F}_{j,n}$
is the same as the element $g_{j}$. 

\end{fact}

\begin{proof}

This is because in the definition of $\mathfrak{F}_{j,n}$, the ray
segment $v$ is required to satisfy $|v\wedge1^{\infty}|\ge n-(j-\bar{j})+D$.
The first level that sees the difference between the portraits of
$\tilde{g}_{j}^{v}$ and $g_{j}$ is $\ge n+D+2$. 

\end{proof}

Because of the previous fact, the following uniqueness property is
inherited from $(g_{n})$. 

\begin{lem}\label{unique}

Let $x\in\mathsf{L}_{n+D+1}$. Suppose $\left(\epsilon_{i},\gamma_{i}\right)_{i=1}^{n},\left(\tilde{\epsilon}_{i},\tilde{\gamma}_{i}\right)_{i=1}^{n}\in\varLambda_{n}$
are such that 
\[
x\cdot\gamma_{n}^{\epsilon_{n}}\ldots\gamma_{1}^{\epsilon_{1}}=x\cdot\tilde{\gamma}_{n}^{\tilde{\epsilon}_{n}}\ldots\tilde{\gamma}_{1}^{\tilde{\epsilon}_{1}}.
\]
Then $\epsilon_{i}=\tilde{\epsilon}_{i}$ for all $1\le i\le n$.

\end{lem}

We derive a displacement estimate from Lemma \ref{translate} and
description of elements in $\mathfrak{F}_{j,n}$ from Lemma \ref{ckgerm}
and Lemma \ref{lemgj}.

\begin{lem}\label{modifi}

Let $j$ be an index such that $\omega_{j-1}={\bf 2}$ and $n-k_{n}<j\le n$.
Let $\tilde{g}_{j}^{v}\in\mathfrak{F}_{j,n}$, then for any $x\in1^{\infty}\cdot G$,
\[
d_{\mathcal{S}}\left(x,x\cdot\tilde{g}_{j}^{v}\right)\le\begin{cases}
2^{j+|\mathfrak{s}^{j+1}x\wedge\mathfrak{s}v|+4} & \mbox{if }|\mathfrak{s}^{j+1}x\wedge\mathfrak{s}v|\ge n-j+D,\\
2^{j+2} & \mbox{if }|\mathfrak{s}^{j+1}x\wedge\mathfrak{s}v|<n-j+D.
\end{cases}
\]

\end{lem}

\begin{proof}

We show the second bound first. Since by definitions, $v$ has prefix
$1^{n-j+\bar{j}+D}$, $|\mathfrak{s}^{j+1}x\wedge\mathfrak{s}v|<n-j+D$
implies that $|\mathfrak{s}^{j+1}x\wedge1^{\infty}|<n-j+D$. In this
case the action of $\tilde{g}_{j}^{v}$ on $x$ is the same as $g_{j}$,
that is $x\cdot\tilde{g}_{j}^{v}=x\cdot g_{j}$. Therefore 
\[
d_{\mathcal{S}}\left(x,x\cdot\tilde{g}_{j}^{v}\right)=d_{\mathcal{S}}(x,x\cdot g_{j})\le2^{j+2}.
\]
Now suppose $|\mathfrak{s}^{j+1}x\wedge\mathfrak{s}v|\ge n-j+D$.
By definition of $\tilde{g}_{j}^{v}$ and Lemma \ref{lemgj}, depending
parity of $\sum_{i=1}^{j}x_{i}$, we need to consider $(\mathfrak{s}^{j}x)\cdot a\mathfrak{c}_{j}^{v}$
or $(\mathfrak{s}^{j}x)\cdot\mathfrak{c}_{j}^{v}a$. By description
of $\mathfrak{c}_{j}^{v}$ in Lemma \ref{ckgerm} and the distance
bound in Lemma \ref{translate}, we have that for any ray $y$, 
\[
d_{\mathcal{S}}(y,y\cdot\mathfrak{c}_{j}^{v})\le2^{|y\wedge v|+4}.
\]
Since
\[
d_{\mathcal{S}}\left(x,x\cdot\tilde{g}_{j}^{v}\right)\le2^{j}\max\left\{ d_{\mathcal{S}}(\mathfrak{s}^{j}x,\left(\mathfrak{s}^{j}x\right)\cdot a\mathfrak{c}_{j}^{v}),d_{\mathcal{S}}(\mathfrak{s}^{j}x,\left(\mathfrak{s}^{j}x\right)\cdot\mathfrak{c}_{j}^{v}a)\right\} ,
\]
the claim follows. 

\end{proof}

Lemma \ref{modifi} implies that

\[
\max_{\gamma\in\mathfrak{F}_{j,n}}d_{S}(x,x\cdot\gamma)\le2^{j+2k_{n}+2D+4}.
\]
By the triangle inequality, for any $x$ and $\boldsymbol{\epsilon}\in\{0,1\}^{n}$,
\[
\max_{\boldsymbol{\gamma}\in\mathfrak{F}_{n}}d_{S}(x,x\cdot\gamma_{n}^{\epsilon_{n}}\ldots\gamma_{1}^{\epsilon_{1}})\le2^{n+2k_{n}+2D+5}.
\]
The purpose of introducing the conjugations $\mathfrak{c}_{j}^{v}$
and randomizing over $v$ in the construction of the measure $\upsilon_{n}$
is to average so that the tail decays of measures is lighter than
the maximum indicated above. 

\begin{lem}\label{modifiedtail}

We have the following upper bounds:
\begin{description}
\item [{(i)}] Let $j$ be a level such that $n-k_{n}\le j\le n$ and $\omega_{j-1}={\bf 2}$.
For any $\ell$ such that $n\le\ell\le j+2k_{n}+2D+4$ and $x\in1^{\infty}\cdot G$,
\[
{\bf u}_{\mathfrak{F}_{j,n}}\left(\left\{ \gamma:\ d_{\mathcal{S}}(x,x\cdot\gamma)\ge2^{\ell}\right\} \right)\le2^{-\frac{1}{D}\left(\ell-n\right)+2}.
\]
\item [{(ii)}] For any $\left(\epsilon_{1},\ldots,\epsilon_{n}\right)\in\{0,1\}^{n}$,
$n\le\ell\le n+2k_{n}+2D+5$ and $x\in1^{\infty}\cdot G$,
\[
{\bf u}_{\mathfrak{F}_{n}}\left(\left\{ \boldsymbol{\gamma}:\ d_{\mathcal{S}}\left(x,x\cdot\gamma_{n}^{\epsilon_{n}}\ldots\gamma_{1}^{\epsilon_{1}}\right)\ge k_{n}2^{\ell}\right\} \right)\le8k_{n}2^{-\frac{1}{D}(\ell-n)}.
\]
\end{description}
\end{lem}

\begin{proof}

(i). By Lemma \ref{modifi}, for such $\ell$,
\[
{\bf u}_{\mathfrak{F}_{j,n}}\left(\left\{ \gamma:\ d_{S}(x,x\cdot\gamma)\ge2^{\ell}\right\} \right)\le{\bf u}_{\mathfrak{F}_{j,n}}\left(\left\{ \tilde{g}_{3j}^{v}:\ \left|\mathfrak{s}^{j+1}x\wedge\mathfrak{s}v\right|\ge\ell-j-4\right\} \right).
\]
Recall that $\mathfrak{F}_{j,n}$ is indexed by strings $v\in\mathsf{V}_{2k_{n}}^{j}$
and $\left|1^{\infty}\wedge v\right|\ge n-j+D$. The number of such
strings $v$ such that $\left|\mathfrak{s}^{j+1}x\wedge\mathfrak{s}v\right|\ge\ell-j-4$
is bounded from above by $2^{\frac{1}{D}(D-\bar{j}+2k_{n}-(\ell-j-3))}$,
because the first $\ell-j-3$ digits of $v$ is prescribed to agree
with $1\mathfrak{s}^{j+1}x$. By definition, the size of $\mathfrak{F}_{j,n}$
is $2^{\frac{1}{D}(2k_{n}-(n-j+\bar{j}))}$. Therefore, for the uniform
measure on $\mathfrak{F}_{j,n}$, we have that 
\[
{\bf u}_{\mathfrak{F}_{j,n}}\left(\left\{ \tilde{g}_{3j}^{v}:\ \left|\mathfrak{s}^{j+1}x\wedge\mathfrak{s}v\right|\ge\ell-j-4\right\} \right)\le\frac{2^{\frac{1}{D}(2k_{n}-(\ell-j+\bar{j}-3-D))}}{2^{\frac{1}{D}(2k_{n}-(n-j+\bar{j}))}}\le4\cdot2^{-\frac{1}{D}(\ell-n)}.
\]

(ii). Since for $i\le n-k_{n}$, $\mathfrak{F}_{i,n}=\{g_{i}\}$,
by triangle inequality we have for all $x$,
\[
d_{\mathcal{S}}\left(x,x\cdot\gamma_{n-k_{n}}^{\epsilon_{n-k_{n}}}\ldots\gamma_{1}^{\epsilon_{1}}\right)\le2^{n-k_{n}+3}.
\]
 For indices $i>n-k_{n}$,
\begin{align*}
 & {\bf u}_{\mathfrak{F}_{n}}\left(\left\{ \boldsymbol{\gamma}:\ d_{\mathcal{S}}\left(x,x\cdot\gamma_{n}^{\epsilon_{n}}\ldots\gamma_{n-k_{n}+1}^{\epsilon_{n-k_{n}+1}}\right)\ge k_{n}2^{\ell}\right\} \right)\\
\le & {\bf u}_{\mathfrak{F}_{n}}\left(\left\{ \boldsymbol{\gamma}:\ \exists j,\ n-k_{n}<j\le n\mbox{ and }d_{\mathcal{S}}\left(x\cdot\gamma_{n}^{\epsilon_{n}}\ldots\gamma_{j+1}^{\epsilon_{j+1}},x\cdot\gamma_{n}^{\epsilon_{n}}\ldots\gamma_{j}^{\epsilon_{j}}\right)\ge2^{\ell}\right\} \right)\\
\le & \sum_{j=n-k_{n}+1}^{n}\sup_{x\in1^{\infty}\cdot G}{\bf u}_{\mathfrak{F}_{j,n}}\left(\left\{ \gamma_{j}:\ d_{\mathcal{S}}(x,x\cdot\gamma_{j})\ge2^{\ell}\right\} \right)\\
\le & 4k_{n}2^{-\frac{1}{D}(\ell-n)}.
\end{align*}
Combining the two parts, we obtain 
\begin{align*}
 & {\bf u}_{\mathfrak{F}_{n}}\left(\left\{ \boldsymbol{\gamma}:\ d_{\mathcal{S}}\left(x,x\cdot\gamma_{n}^{\epsilon_{n}}\ldots\gamma_{1}^{\epsilon_{1}}\right)\ge k_{n}2^{\ell}\right\} \right)\\
\le & {\bf u}_{\mathfrak{F}_{n}}\left(\left\{ \boldsymbol{\gamma}:\ d_{\mathcal{S}}\left(x,x\cdot\gamma_{n}^{\epsilon_{n}}\ldots\gamma_{n-k_{n}}^{\epsilon_{n-k_{n}}}\right)\ge k_{n}2^{\ell}-2^{n-k_{n}+3}\right\} \right)\\
\le & 8k_{n}2^{-\frac{1}{D}(\ell-n)}.
\end{align*}

\end{proof}

Because of the modifications, the induced transition kernel $P_{\mu_{\beta}}$
on the Schreier graph $\mathcal{S}_{\omega}$ is not comparable to
a transition kernel expressed in terms of the distance function $d_{\mathcal{S}}(x,y)$.
Indeed the jump kernel $P_{\mu_{\beta}}(x,\cdot)$ depends on the
location $x$, rather than the distance $d_{\mathcal{S}}(x,\cdot)$.
The following upper bounds somewhat consider the worst scenario. 

\begin{prop}\label{eta3kernel}

Consider $\mu_{\beta}$ defined in (\ref{eq:eta3}) with $k_{n}=A\left\lfloor \log_{2}n\right\rfloor $,
where $1-\frac{1}{D}<\beta<1$ and $A$ is an integer divisible by
$D$. There exists constant $C=C(D,\beta)<\infty$ such that:
\begin{description}
\item [{(i)}] For any $x,y\in1^{\infty}\cdot G$ ,
\[
P_{\mu_{\beta}}(x,y)\le Cd_{\mathcal{S}}(x,y)^{-1-\beta}\left(\log_{2}d_{\mathcal{S}}(x,y)\right)^{2A(1-\frac{1}{D}+\beta)}\left(\log_{2}\log_{2}d_{\mathcal{S}}(x,y)\right)^{1+\frac{1}{D}}.
\]
\item [{(ii)}] For any $x\in1^{\infty}\cdot G$,
\begin{align*}
\sum_{y:d_{S}(x,y)\ge r}P_{\mu_{\beta}}(x,y) & \le Cr^{-\beta}\left(\log_{2}r\right)^{2A(\beta-\frac{1}{D})}\left(\log_{2}\log_{2}r\right)^{1+\frac{1}{D}}.\\
\sum_{y:d_{S}(x,y)\le r}d_{S}(x,y)^{2}P_{\mu_{\beta}}(x,y) & \le Cr^{2-\beta}\left(\log_{2}r\right)^{2A(\beta-\frac{1}{D})}\left(\log_{2}\log_{2}r\right)^{1+\frac{1}{D}}.
\end{align*}
\end{description}
\end{prop}

\begin{proof}

(i). Given $x,y$, write $\ell(x,y)=\log_{2}d_{\mathcal{S}}(x,y)$
and let 
\[
n_{0}=\min\left\{ n:\ \ell(x,y)\le n+2k_{n}+2D+5\right\} .
\]
Then for $n<n_{0}$, $y\notin x\cdot\theta_{n}(\varLambda_{n})$.
For $n\ge n_{0}$, either $y\notin x\cdot\theta_{n}(\varLambda_{n})$,
in this case $\upsilon_{n}$ does not contribute to $P_{\mu_{\beta}}(x,y)$;
or by Lemma \ref{unique} there is a unique $\boldsymbol{\epsilon}\in\{0,1\}^{n}$
such that there exists $\boldsymbol{\gamma}\in\mathfrak{F}_{n}$ with
$x\cdot\gamma_{n}^{\epsilon_{n}}\ldots\gamma_{1}^{\epsilon_{1}}=y$.
In the latter case, we have that for $n\ge n_{0}$,
\begin{equation}
P_{\upsilon_{n}}(x,y)=\frac{1}{2^{n}}{\bf u}_{\mathfrak{F}_{n}}\left(\left\{ \boldsymbol{\gamma}:\ x\cdot\gamma_{n}^{\epsilon_{n}}\ldots\gamma_{1}^{\epsilon_{1}}=y\right\} \right)\le2^{-n}.\label{eq:2n}
\end{equation}
We now show another upper bound for these index $n$ such that $n_{0}\le n<\ell(x,y)-\log_{2}k_{n}$.
Note that 
\begin{align*}
P_{\upsilon_{n}}(x,y) & =\frac{1}{2^{n}}{\bf u}_{\mathfrak{F}_{n}}\left(\left\{ \boldsymbol{\gamma}:\ x\cdot\gamma_{n}^{\epsilon_{n}}\ldots\gamma_{1}^{\epsilon_{1}}=y\right\} \right)\\
 & \le\frac{1}{2^{n}}{\bf u}_{\mathfrak{F}_{n}}\left(\left\{ \boldsymbol{\gamma}:\ d_{\mathcal{S}}\left(x,x\cdot\gamma_{n}^{\epsilon_{n}}\ldots\gamma_{1}^{\epsilon_{1}}\right)\ge d_{\mathcal{S}}(x,y)\right\} \right).
\end{align*}
By part (ii) in Lemma \ref{modifiedtail}, with $\ell=\ell(x,y)-\log_{2}k_{n}$
so that $k_{n}2^{\ell}=d_{\mathcal{S}}(x,y)$, we have 
\[
{\bf u}_{\mathfrak{F}_{n}}\left(\left\{ \boldsymbol{\gamma}:\ d_{\mathcal{S}}\left(x,x\cdot\gamma_{n}^{\epsilon_{n}}\ldots\gamma_{1}^{\epsilon_{1}}\right)\ge d_{\mathcal{S}}(x,y)\right\} \right)\le8k_{n}2^{-\frac{1}{D}(\ell(x,y)-\log_{2}k_{n}-n)}.
\]
Therefore, for $n\in\left[n_{0},\ell(x,y)-\log_{2}k_{n}\right],$
we have
\begin{equation}
P_{\upsilon_{n}}(x,y)\le\frac{8}{2^{n}}k_{n}2^{-\frac{1}{D}\left(\ell(x,y)-n-\log_{2}k_{n}\right)}.\label{eq:2ncorrected}
\end{equation}
Note that for $n>\ell(x,y)-\log_{2}k_{n}$, the bound (\ref{eq:2ncorrected})
is still valid because of (\ref{eq:2n}). Since the induced transition
probabilities satisfy $P_{\check{\upsilon}_{n}}(x,y)=P_{\upsilon_{n}}(y,x),$
we have that $P_{\check{\upsilon}_{n}}(x,y)$ admits the same upper
bound in terms of $\ell(x,y)$. 

Split $P_{\mu_{\beta}}(x,y)$ into two parts: 
\begin{align*}
I & =\sum_{n\ge n_{0}}^{\ell(x,y)}\frac{a_{n}}{2}\left(P_{\upsilon_{n}}(x,y)+P_{\check{\upsilon}_{n}}(x,y)\right)\\
II & =\sum_{n>\ell(x,y)}\frac{a_{n}}{2}\left(P_{\upsilon_{n}}(x,y)+P_{\check{\upsilon}_{n}}(x,y)\right),
\end{align*}
where the coefficients are $a_{n}=C_{\beta}2^{-\beta n}$. For the
first part, the upper bound (\ref{eq:2ncorrected}) on $P_{\upsilon_{n}}(x,y)$
implies
\begin{align*}
I & \le\sum_{n\ge n_{0}}^{\ell(x,y)}\frac{8a_{n}}{2^{n}}k_{n}^{1+\frac{1}{D}}2^{-\frac{1}{D}\left(\ell(x,y)-n\right)}\\
 & =2^{-\frac{1}{D}\ell(x,y)+3}\sum_{n\ge n_{0}}^{\ell(x,y)}\frac{C_{\beta}}{2^{n(1-\frac{1}{D}+\beta)}}k_{n}^{1+\frac{1}{D}}\\
 & \le AC_{\beta}2^{-\frac{1}{D}\ell(x,y)+3}\sum_{n\ge n_{0}}^{\ell(x,y)}2^{-n\left(1-\frac{1}{D}+\beta\right)}(\log_{2}\ell(x,y))^{1+\frac{1}{D}}\\
 & \le C'2^{-\frac{1}{D}\ell(x,y)}2^{-n_{0}\left(1-\frac{1}{D}+\beta\right)}(\log_{2}\ell(x,y))^{1+\frac{1}{D}}.
\end{align*}
Since $n_{0}+2k_{n_{0}}\le\ell(x,y)\le n_{0}+2k_{n_{0}}+2D+5$, we
have 
\begin{align*}
I & \le C'2^{-\frac{1}{D}\ell(x,y)}2^{-(\ell(x,y)-2k_{n_{0}})\left(1-\frac{1}{D}+\beta\right)}(\log_{2}\ell(x,y))^{1+\frac{1}{D}}\\
 & \le C'2^{-\ell(x,y)(1+\beta)}\ell(x,y)^{2A(1-\frac{1}{D}+\beta)}(\log_{2}\ell(x,y))^{1+\frac{1}{D}}\\
 & =C'd_{\mathcal{S}}(x,y)^{-1-\beta}\left(\log_{2}d_{\mathcal{S}}(x,y)\right)^{2A(1-\frac{1}{D}+\beta)}\left(\log_{2}\log_{2}d_{\mathcal{S}}(x,y)\right)^{1+\frac{1}{D}}.
\end{align*}

For the second part, from the bound $P_{\upsilon_{n}}(x,y)\le2^{-n}$
we deduce 
\[
II\le C_{\beta}2^{-(1+\beta)\ell(x,y)}=C_{\beta}d_{\mathcal{S}}(x,y)^{-1-\beta}.
\]
Combine these two parts we obtain statement (i).

(ii). To bound the tail, note that 
\begin{align*}
\sum_{y:d_{S}(x,y)\ge r}P_{\upsilon_{n}}(x,y) & =\frac{1}{2^{n}}\sum_{\boldsymbol{\epsilon}\in\{0,1\}^{n}}{\bf u}_{\mathfrak{F}_{n}}\left(\left\{ \boldsymbol{\gamma}:\ d\left(x\cdot\gamma_{n}^{\epsilon_{n}}\ldots\gamma_{1}^{\epsilon_{1}},x\right)\ge r\right\} \right),\\
\sum_{y:d_{S}(x,y)\ge r}P_{\check{\upsilon}_{n}}(x,y) & =\frac{1}{2^{n}}\sum_{\boldsymbol{\epsilon}\in\{0,1\}^{n}}{\bf u}_{\mathfrak{F}_{n}}\left(\left\{ \boldsymbol{\gamma}:\ d\left(x\cdot\left(\gamma_{n}^{\epsilon_{n}}\ldots\gamma_{1}^{\epsilon_{1}}\right)^{-1},x\right)\ge r\right\} \right).
\end{align*}
Then it follows from Lemma \ref{modifi} that for $n\le\log_{2}r\le n+2k_{n}+2D+5$,
\begin{align*}
\frac{1}{2}\sum_{y:d_{S}(x,y)\ge r}\left(P_{\upsilon_{n}}+P_{\check{\upsilon}_{n}}\right)(x,y) & \le\sup_{x,\boldsymbol{\epsilon}}{\bf u}_{\mathfrak{F}_{n}}\left(\left\{ \boldsymbol{\gamma}:\ d\left(x\cdot\left(\gamma_{n}^{\epsilon_{n}}\ldots\gamma_{1}^{\epsilon_{1}}\right)^{-1},x\right)\ge r\right\} \right)\\
 & \le8k_{n}^{1+\frac{1}{D}}2^{-\frac{1}{D}(\log_{2}r-n)}.
\end{align*}
Let $n_{0}=\min\left\{ n:\ 2^{n+2k_{n}+2D+5}\ge r\right\} $. Then
\begin{align*}
\sum_{y:d_{S}(x,y)\ge r}P_{\mu_{\beta}}(x,y) & \le\sum_{n\ge n_{0}}^{\log_{2}r}a_{n}\sum_{y:d_{S}(x,y)\ge r}\frac{1}{2}\left(P_{\upsilon_{n}}+P_{\check{\upsilon}_{n}}\right)(x,y)+\sum_{n>\log_{2}r}a_{n}\\
 & \le\sum_{n\ge n_{0}}^{\log_{2}r}8a_{n}k_{n}^{1+\frac{1}{D}}2^{-\frac{1}{D}(\log_{2}r-n)}+\sum_{n>\log_{2}r}a_{n}\\
 & \le C'2^{-\beta n_{0}}k_{n_{0}}^{1+\frac{1}{D}}2^{-\frac{1}{D}(\log_{2}r-n_{0})}+C_{\beta}r^{-\beta},
\end{align*}
where in the last step the assumption that $\beta>1/D$ is used. By
definition of $n_{0}$, we have that 
\[
\sum_{y:d_{S}(x,y)\ge r}P_{\mu_{\beta}}(x,y)\le C'r^{-\beta}\left(\log_{2}r\right)^{2A\left(\beta-\frac{1}{D}\right)}\left(\log_{2}\log_{2}r\right)^{1+\frac{1}{D}}.
\]
Finally, for the truncated second moment, by the tail bound obtained
above, we obtain
\begin{align*}
\sum_{y:d_{S}(x,y)\le r}d_{S}(x,y)^{2}P_{\mu_{\beta}}(x,y) & \le\sum_{j=1}^{\log_{2}r}(2^{j})^{2}\sum_{y:d_{S}(x,y)\ge2^{j-1}}P_{\mu_{\beta}}(x,y)\\
 & \le Cr^{2-\beta}\left(\log_{2}r\right)^{2A\left(\beta-\frac{1}{D}\right)}\left(\log_{2}\log_{2}r\right)^{1+\frac{1}{D}}.
\end{align*}

\end{proof}

\subsection{Upper bounds on the Green function\label{subsec:greenupper}}

Throughout this subsection $\beta\in(1-\frac{1}{D},1)$ and $k_{n}=A\left\lfloor \log_{2}n\right\rfloor $,
and the measure $\mu_{\beta}$ is defined in (\ref{eq:eta3}).

The following on-diagonal upper bound does not depend on the choice
of $(k_{n})$ because by construction, on the finite level $\mathsf{L}_{n}$
the transition kernel induced by $\upsilon_{n}$ coincide with the
one induced by ${\bf u}_{F_{n}}$. The same argument as in Lemma \ref{CSC}
implies that the $\mu_{\beta}$-induced random walk on the orbit $1^{\infty}\cdot G_{\omega}$
admits the following on-diagonal upper bound.

\begin{prop}\label{eta3diag}

Let $P_{\mu_{\beta}}$ be the transition kernel on $1^{\infty}\cdot G_{\omega}$
induced by $\mu_{\beta}$. Then there exists a constant $C=C(\beta)<\infty$
such that for $t\in\mathbb{N}$,
\[
\sup_{x,y\in1^{\infty}\cdot G}P_{\mu_{\beta}}^{t}(x,y)\le\frac{C}{t^{1/\beta}}.
\]

\end{prop}

\begin{proof}

Let $P_{n}$ be the transition kernel on the orbit of $o$ induced
by $\frac{1}{2}\left(\upsilon_{n}+\check{\upsilon}_{n}\right)$. As
in the proof of Lemma \ref{CSC}, from Lemma \ref{unique} we derive
that for a set $U\subset o\cdot G_{\omega}$ with volume $|U|\le2^{n-1}$,
\[
\frac{\left|\partial_{P_{n}}U\right|}{|U|}\ge\frac{1}{2}.
\]
Then by the Cheeger inequality (\ref{eq:cheeger}) and the fact that
$P_{\mu_{\beta}}$ is a convex combination of $P_{n}$'s, we have
that the $\ell^{2}$-isoperimetric profile of $P_{\mu_{\beta}}$ satisfies
\[
\Lambda_{P_{\mu_{\beta}}}(2^{n-1})\ge\frac{C_{\beta}}{2}2^{-n\beta},\ n\ge1.
\]
That is, $\Lambda_{P_{\mu_{\beta}}}(v)\ge c_{\beta}v^{-\beta}$ for
$v\ge1$. Such a lower bound for the $\ell^{2}$-isoperimetric profile
implies the stated upper bound, see \cite[Proposition II.1]{coulhon}
or \cite[Theorem 1.1]{grigoryan}.

\end{proof}

We will need an off-diagonal upper bound on $P_{\mu_{\beta}}^{t}$,
which can be deduced from the Meyer's construction and the Davies
method as in Barlow-Grigor'yan-Kumagai \cite{BGK}. Originally, the
Davies method \cite{davies1,davies2} was developed to derive Gaussian
off-diagonal upper bounds. It was extended to more general Markov
semigroups in \cite{CKS}. The Davies method has been successfully
applied to jump processes, see for example \cite{BGK,ChenKumagai,BBCK}
and references therein. 

Let $J(x,y)$ be a symmetric transition kernel on a countable set
$X$. For technical reasons, it is more convenient to consider continuous
time random walk. Let $P_{t}$ be the associated heat semigroup and
$p(t,x,y)$ its transition density. The following proposition follows
from the proof of heat kernel upper bound in \cite[Section 3]{BGK},
see also \cite[Section 4.4]{ChenKumagai}. It states that if we have
an on-diagonal upper bound and upper bounds on the tail of $J(x,\cdot)$
and growth of truncated second moment of $J(x,\cdot)$, then the Davies
method provides an off-diagonal upper bound on $p(t,x,y)$. Note that
a uniform volume condition $\mu(B_{\rho}(x,r))\asymp V(r)$ is not
required for such an upper bound. 

\begin{prop}[\cite{BGK}]\label{davies}

Let $J(x,y)$ be a symmetric transition kernel on a countable set
$X$ and let $\rho$ be a metric on $X$. Suppose 
\begin{description}
\item [{(i)}] There exists $0<C_{0}<\infty$ and $\beta>0$ such that for
all $t>0$, we have 
\[
\sup_{x\in X}p(t,x,x)\le\frac{C_{0}}{t^{1/\beta}}.
\]
\item [{(ii)}] There exists an increasing function $\phi:(0,\infty)\to(0,\infty)$
such that $\phi(2r)\le c_{\phi}\phi(r)$ for all $r>0$ and for all
$x\in X$, $r>0$, 
\begin{align*}
\sum_{y\in X}J(x,y)\mathbf{1}_{\{\rho(x,y)>r\}} & \le\frac{1}{\phi(r)},\\
\sum_{y\in X}\rho^{2}(x,y)J(x,y)\mathbf{1}_{\{\rho(x,y)\le r\}} & \le\frac{r^{2}}{\phi(r)}.
\end{align*}
\end{description}
Then there exists a constant $C=C\left(C_{0},c_{\phi},\beta\right)>0$
such that for any $x,y\in X$ and $t\le\phi(\rho(x,y))$,
\[
p(t,x,y)\le\frac{Ct}{\phi\left(\rho(x,y)\right)^{1+\frac{1}{\beta}}}+t\sup_{\rho(u,v)\ge\rho(x,y)}J(u,v).
\]

\end{prop}

Since this bound is crucial in the argument for non-triviality of
the Poisson boundary but the formulation of Proposition \ref{davies}
is not the same as in \cite{BGK}, we explain its proof for the reader's
convenience at the end of this subsection. To apply Proposition \ref{davies},
we verify the condition in (i) by Proposition \ref{eta3diag} and
conditions in (ii) by Proposition \ref{eta3kernel}. In this way we
obtain an off-diagonal upper bound for the transition probability
and the upper bound on the Green function stated in Proposition \ref{green}.
We remark that these bounds are not sharp, however it is sufficient
for our purposes. 

\begin{proof}[Proof of Proposition \ref{green}]

Let $J=P_{\mu_{\beta}}$ be the induced jumping kernel on the Schreier
graph and $p(t,x,y)$ be the heat kernel of the corresponding jumping
process in continuous time. By Proposition \ref{eta3diag}, the lower
bound on $\ell^{2}$-isoperimetric profile of $J$ implies
\[
\sup_{x,y\in X}p(t,x,y)\le\frac{C}{t^{1/\beta}}.
\]
Let 
\[
\phi(r)=r^{\beta}\left(\log_{2}r\right)^{-2A(\beta-\frac{1}{D})-\epsilon},
\]
then by Proposition \ref{eta3kernel}, we have that there exists a
constant $C=C(\beta,D,\epsilon)<\infty$ such that 
\begin{align*}
\sum_{y\in X}J(x,y)\mathbf{1}_{\{d_{\mathcal{S}}(x,y)>r\}} & \le\frac{C}{\phi(r)}.\\
\sum_{y\in X}\rho^{2}(x,y)J(x,y)\mathbf{1}_{\{d_{\mathcal{S}}(x,y)\le r\}} & \le\frac{Cr^{2}}{\phi(r)}.
\end{align*}
Plugging these bounds in Proposition \ref{davies}, we have that for
$t\le\phi(d_{\mathcal{S}}(x,y))$,
\[
p(t,x,y)\le\frac{Ct}{\phi\left(d_{\mathcal{S}}(x,y)\right)^{1+\frac{1}{\beta}}}+t\sup_{d_{\mathcal{S}}(u,v)\ge d_{\mathcal{S}}(x,y)}J(u,v).
\]
Note that
\[
\phi\left(d_{\mathcal{S}}(x,y)\right)^{-\left(1+\frac{1}{\beta}\right)}=\left(\frac{d_{\mathcal{S}}(x,y)}{\log_{2}^{2A}d_{\mathcal{S}}(x,y)}\right)^{-1-\beta}(\log_{2}d_{\mathcal{S}}(x,y))^{-\left(\frac{2A}{D}-\epsilon\right)\left(1+\frac{1}{\beta}\right)}.
\]
By Proposition \ref{eta3kernel}, there exists $C=C(\beta,\epsilon)>0$
\[
\sup_{d_{\mathcal{S}}(u,v)\ge d_{\mathcal{S}}(x,y)}J(u,v)\le C\left(\frac{d_{\mathcal{S}}(x,y)}{\log_{2}^{2A}d_{\mathcal{S}}(x,y)}\right)^{-1-\beta}(\log_{2}d_{\mathcal{S}}(x,y))^{-\left(\frac{2A}{D}-\epsilon\right)}.
\]
Combining these two parts, we obtain that for $t\le\phi(d_{\mathcal{S}}(x,y)),$
\[
p(t,x,y)\le Ct\left(\frac{d_{\mathcal{S}}(x,y)}{\log_{2}^{2A}d_{\mathcal{S}}(x,y)}\right)^{-1-\beta}(\log_{2}d_{\mathcal{S}}(x,y))^{-\left(\frac{2A}{D}-\epsilon\right)},
\]
where $C$ is a constant only depending on $\beta,D,\epsilon$. Compared
with the on-diagonal bound $p(t,x,y)\le Ct^{-1/\beta}$, note that
for $t$ in the interval $\phi_{1}(d_{\mathcal{S}}(x,y))\le t\le\phi(d_{\mathcal{S}}(x,y))$,
where 
\[
\phi_{1}(r)=\left(\frac{r}{\log_{2}^{2A}r}\right)^{\beta}(\log_{2}r)^{\left(\frac{2A}{D}-\epsilon\right)\frac{\beta}{1+\beta}},
\]
it is better to use the on-diagonal bound. We conclude that the heat
kernel satisfies that for all $t>0$,
\[
p(t,x,y)\le C\min\left\{ t^{-\frac{1}{\beta}},t\left(\frac{d_{\mathcal{S}}(x,y)}{\log_{2}^{2A}d_{\mathcal{S}}(x,y)}\right)^{-1-\beta}(\log_{2}d_{\mathcal{S}}(x,y))^{-\left(\frac{2A}{D}-\epsilon\right)}\right\} .
\]

It is elementary and well-known that when $J(x,x)\ge\alpha>0$, the
continuous time transition probability and discrete time transition
probabilities are comparable, because the former is a Poissonization
of the later, see for example \cite[Subsection 3.2]{delmotte}. By
\cite[Theorem 3.6]{delmotte}, the discrete time transition probability
$P_{\mu_{\beta}}^{t}(x,y)\le Cp(t,x,y)$ where $C>0$ is a constant
that only depends on $\mu_{\beta}(id)$. Therefore $P_{\mu_{\beta}}^{t}(x,y)$
admits the same upper bound as $p(t,x,y)$ with $C$ a larger constant.
To obtain the estimate on the Green function, sum up the transition
probability upper bound over $t\in\{0\}\cup\mathbb{N}$, 
\[
{\bf G}_{\mu_{\beta}}(x,y)=\sum_{t=0}^{\infty}P_{\mu_{\beta}}^{t}(x,y)=\sum_{t=0}^{\phi_{1}(d_{\mathcal{S}}(x,y))}P_{\mu_{\beta}}^{t}(x,y)+\sum_{t>\phi_{1}(d_{\mathcal{S}}(x,y))}P_{\mu_{\beta}}^{t}(x,y).
\]
For the first part, we have
\begin{align*}
\sum_{t=0}^{\phi_{1}(d_{\mathcal{S}}(x,y))}P_{\mu_{\beta}}^{t}(x,y) & \le C\sum_{t=0}^{\phi_{1}(d_{\mathcal{S}}(x,y))}t\left(\frac{d_{\mathcal{S}}(x,y)}{\log_{2}^{2A}d_{\mathcal{S}}(x,y)}\right)^{-1-\beta}(\log_{2}d_{\mathcal{S}}(x,y))^{-\left(\frac{2A}{D}-\epsilon\right)}\\
 & \le C\phi_{1}(d_{\mathcal{S}}(x,y))^{2}\left(\frac{d_{\mathcal{S}}(x,y)}{\log_{2}^{2A}d_{\mathcal{S}}(x,y)}\right)^{-1-\beta}(\log_{2}d_{\mathcal{S}}(x,y))^{-\left(\frac{2A}{D}-\epsilon\right)}\\
 & =C\left(\frac{d_{\mathcal{S}}(x,y)}{\log_{2}^{2A}d_{\mathcal{S}}(x,y)}\right)^{\beta-1}(\log_{2}d_{\mathcal{S}}(x,y))^{-\frac{1-\beta}{1+\beta}\left(\frac{2A}{D}-\epsilon\right)}.
\end{align*}
For the second part, 
\begin{align*}
\sum_{t>\phi_{1}(d_{\mathcal{S}}(x,y))}P_{\mu_{\beta}}^{t}(x,y) & \le\sum_{t>\phi_{1}(d_{\mathcal{S}}(x,y))}Ct^{-1/\beta}\\
 & \le\frac{C}{1-1/\beta}\phi_{1}(d_{\mathcal{S}}(x,y))^{1-\frac{1}{\beta}}\\
 & =\frac{C}{1-1/\beta}\left(\frac{d_{\mathcal{S}}(x,y)}{\log_{2}^{2A}d_{\mathcal{S}}(x,y)}\right)^{\beta-1}(\log_{2}d_{\mathcal{S}}(x,y))^{-\frac{1-\beta}{1+\beta}\left(\frac{2A}{D}-\epsilon\right)}.
\end{align*}
Combining these two parts, we obtain the stated upper bound on the
Green function. 

\end{proof}

\begin{proof}[Proof of Proposition \ref{davies} following \cite{BGK}]

First split the jumping kernel $J$ into two parts:
\[
J_{1}^{R}(x,y)=J(x,y)\mathbf{1}_{\{\rho(x,y)\le R\}}\ \mbox{ and \ }J_{2}^{R}(x,y)=J(x,y)\mathbf{1}_{\{\rho(x,y)>R\}}.
\]
Let $p_{R}(t,x,y)$ be transition density of the jumping process with
jump kernel $J_{1}^{R}$. Then as a consequence of the Meyer's construction,
see \cite[Lemma 3.1]{BGK}, we have
\begin{equation}
p(t,x,y)\le p_{R}(t,x,y)+t\left\Vert J_{2}^{R}\right\Vert _{\infty}.\label{eq:meyer}
\end{equation}

The Dirichlet forms satisfy
\[
\mathcal{E}_{J}(f,f)-\mathcal{E}_{J_{1}^{R}}(f,f)\le4\left\Vert f\right\Vert _{2}^{2}\sup_{x\in X}\left\{ \sum_{y\in X}J(x,y)\mathbf{1}_{\{\rho(x,y)>r\}}\right\} \le4\left\Vert f\right\Vert _{2}^{2}/\phi(R).
\]
It follows that 
\[
\sup_{x,y}p_{R}(t,x,y)\le e^{\frac{4t}{\phi(R)}}\sup_{x,y}p(t,x,y)\le Ce^{\frac{4t}{\phi(R)}}t^{-1/\beta}.
\]
By \cite{coulhon}, this bound turns into Nash inequality 
\[
\left\Vert f\right\Vert _{2}^{2(1+\beta)}\le C\left(\mathcal{E}_{J_{1}^{R}}(f,f)+\frac{4}{\phi(R)}\left\Vert f\right\Vert _{2}^{2}\right)\left\Vert f\right\Vert _{1}^{2\beta}
\]
Then by \cite[Theorem 3.25]{CKS}, for any function $\psi$ on $X$,
\begin{equation}
p_{R}(t,x,y)\le C_{2}t^{-\frac{1}{\beta}}\exp\left(\frac{4t}{\phi(R)}+72\Lambda_{R}(\psi)^{2}t-\psi(y)+\psi(x)\right),\label{eq:pR}
\end{equation}
where $C_{2}>0$ is a constant and the quantity $\Lambda_{R}(\psi)$
is defined as 
\begin{align*}
\Lambda_{R}(\psi) & ^{2}=\max\left\{ \left\Vert e^{-2\psi}\Gamma(e^{\psi},e^{\psi})\right\Vert _{\infty},\left\Vert e^{2\psi}\Gamma(e^{-\psi},e^{-\psi})\right\Vert _{\infty}\right\} ,\\
\Gamma_{R}(f,g) & (x)=\sum_{y\in X}(f(y)-f(x))(g(y)-g(x))J_{1}^{R}(x,y).
\end{align*}

Take a parameter $\lambda>0$ and for any two points $x_{0}\neq y_{0}$,
take the function 
\[
\psi(x)=\lambda\left(\rho(x,x_{0})-\rho(x,y_{0})\right)_{+}.
\]
Note that $\psi$ is $\lambda$-Lipschitz with respect to the metric
$\rho$. Then by the inequality $(e^{z}-1)^{2}\le z^{2}e^{2|z|}$,
$z\in\mathbb{R}$ and Assumption (ii), we have 
\begin{align*}
e^{-2\psi(x)}\Gamma(e^{\psi},e^{\psi})(x) & \le\lambda^{2}e^{2\lambda R}\sum_{y\in X}\rho^{2}(x,y)J(x,y)\mathbf{1}_{\{\rho(x,y)\le R\}}\\
 & \le\lambda^{2}e^{2\lambda R}R^{2}/\phi(R).
\end{align*}
The same calculation applies to $-\psi$, therefore
\[
\Lambda_{R}(\psi)^{2}\le\lambda^{2}e^{2\lambda R}R^{2}/\phi(R).
\]
The inequality (\ref{eq:pR}) evaluated at points $x_{0},y_{0}$ states
\begin{align*}
p_{R}(t,x_{0},y_{0}) & \le C_{2}t^{-1/\beta}\exp\left(\frac{4t}{\phi(R)}+72(\lambda R)^{2}e^{2\lambda R}t/\phi(R)-\lambda\rho(x_{0},y_{0})\right)\\
 & \le C_{2}t^{-1/\beta}\exp\left(\frac{4t}{\phi(R)}+72e^{3\lambda R}t/\phi(R)-\lambda\rho(x_{0},y_{0})\right).
\end{align*}
We now restrict to the case $t\le\phi(\rho(x_{0},y_{0}))$. Choose
$\lambda$ and $R$ to be 
\begin{align*}
R & =\frac{\beta}{3(1+\beta)}\rho(x_{0},y_{0}),\\
\lambda & =\frac{1}{3R}\log\left(\frac{\phi(R)}{t}\right).
\end{align*}
We conclude that for $t\le\phi(\rho(x,y))$, 
\[
p_{R}(t,x,y)\le\frac{Ct}{\phi(\rho(x,y))^{1+1/\beta}}.
\]
The statement is obtained by combining this bound with (\ref{eq:meyer}). 

\end{proof} 

\subsection{Proof of Proposition \ref{zetasum}\label{subsec:weights} }

In addition to the Green function estimates, we need to bound the
probability under $\mu_{\beta}$ that a point $x$ carries a germ
$(g,x)$ not in $\mathcal{H}^{b}$. We estimate contributions of each
$\upsilon_{n}$ and $\check{\upsilon}_{n}$ as in the statement of
Proposition \ref{zetasum}.

\begin{proof}[Proof of Proposition \ref{zetasum}]

To lighten notation, write $\mathcal{H}=\mathcal{H}^{b}$ for the
sub-groupoid of $\left\langle b\right\rangle $-germs, $o=1^{\infty}$
, $G=G_{\omega}$ and $d=d_{\mathcal{S}}$. Let $\left(\boldsymbol{\gamma},\boldsymbol{\epsilon}\right)$
be a random variable with uniform distribution on the set $\varLambda_{n}$
defined in (\ref{eq:lambdan}) and write $\boldsymbol{\gamma}^{\boldsymbol{\epsilon}}=\gamma_{n}^{\epsilon_{n}}\ldots\gamma_{1}^{\epsilon_{1}}$.
Then by definition of the measure $\upsilon_{n}$ in (\ref{eq:zetan}),
\[
\sum_{x\in o\cdot G}f\left(d(o,x)\right)\upsilon_{n}\left(\left\{ g\in G:\ (g,x)\notin\mathcal{H}\right\} \right)=\mathbb{E}\left(\sum_{x\in o\cdot G}f\left(d\left(o,x\right)\right)\mathbf{1}_{\left\{ \left(\boldsymbol{\gamma}^{\boldsymbol{\epsilon}},x\right)\notin\mathcal{H}\right\} }\right),
\]
where expectation is taken with respect to uniform distribution on
$\varLambda_{n}$. Recall the multiplication rule in the groupoid
of germs:
\[
\left(\boldsymbol{\gamma}^{\boldsymbol{\epsilon}},x\right)=\left(\gamma_{n}^{\epsilon_{n}},x\right)\left(\gamma_{n-1}^{\epsilon_{n-1}},x\cdot\gamma_{n}^{\epsilon_{n}}\right)\ldots\left(\gamma_{1}^{\epsilon_{1}},x\cdot\gamma_{n}^{\epsilon_{n}}\ldots\gamma_{2}^{\epsilon_{2}}\right).
\]
Since $\mathcal{H}$ is a sub-groupoid, it follows that the indicator
is bounded by
\[
\mathbf{1}_{\left\{ \left(\boldsymbol{\gamma}^{\boldsymbol{\epsilon}},x\right)\notin\mathcal{H}\right\} }\le\sum_{i=0}^{n-1}\mathbf{1}_{\left\{ \left(\gamma_{n-i}^{\epsilon_{n-i}},x\cdot\gamma_{n}^{\epsilon_{n}}\ldots\gamma_{n-i+1}^{\epsilon_{n-i+1}}\right)\notin\mathcal{H}\right\} }.
\]
Therefore 
\begin{align*}
 & \mathbb{E}\left(\sum_{x\in o\cdot G}f\left(d\left(o,x\right)\right)\mathbf{1}_{\left\{ \left(\boldsymbol{\gamma}^{\boldsymbol{\epsilon}},x\right)\notin\mathcal{H}\right\} }\right)\\
\le & \sum_{i=0}^{n-1}\mathbb{E}\left(\sum_{x\in o\cdot G}f\left(d\left(o,x\right)\right)\mathbf{1}_{\left\{ \left(\gamma_{n-i}^{\epsilon_{n-i}},x\cdot\gamma_{n}^{\epsilon_{n}}\ldots\gamma_{n-i+1}^{\epsilon_{n-i+1}}\right)\notin\mathcal{H}\right\} }\right)\\
= & \frac{1}{2}\sum_{i=0}^{n-1}\mathbb{E}\left(\sum_{x\in o\cdot G}f\left(d\left(o,x\cdot\left(\gamma_{n}^{\epsilon_{n}}\ldots\gamma_{n-i+1}^{\epsilon_{n-i+1}}\right)^{-1}\right)\right)\mathbf{1}_{\left\{ \left(\gamma_{n-i},x\right)\notin\mathcal{H}\right\} }\right)\\
= & \frac{1}{2}\sum_{j=1}^{n}\mathbb{E}\left(\sum_{x\in o\cdot G}f\left(d\left(o,x\cdot\gamma_{j+1}^{-\epsilon_{j+1}}\ldots\gamma_{n}^{-\epsilon_{n}}\right)\right)\mathbf{1}_{\left\{ \left(\gamma_{j},x\right)\notin\mathcal{H}\right\} }\right).
\end{align*}
Similarly, for $\check{\upsilon}_{n}$ we have 
\begin{align*}
 & \sum_{x\in o\cdot G}f\left(d(o,x)\right)\check{\upsilon}_{n}\left(\left\{ g\in G:\ (g,x)\notin\mathcal{H}_{x}\right\} \right)\\
= & \mathbb{E}\left(\sum_{x\in o\cdot G}f\left(d\left(o,x\right)\right)\mathbf{1}_{\left\{ \left(\gamma_{1}^{-\epsilon_{1}}\ldots\gamma_{n}^{-\epsilon_{n}},x\right)\notin\mathcal{H}\right\} }\right)\\
\le & \sum_{i=0}^{n-1}\mathbb{E}\left(\sum_{x\in o\cdot G}f\left(d\left(o,x\right)\right)\mathbf{1}_{\left\{ \left(\gamma_{i+1}^{-\epsilon_{i+1}},x\cdot\gamma_{1}^{-\epsilon_{1}}\ldots\gamma_{i}^{-\epsilon_{i}}\right)\notin\mathcal{H}\right\} }\right)\\
= & \frac{1}{2}\sum_{j=1}^{n}\mathbb{E}\left(\sum_{x\in o\cdot G}f\left(d\left(o,x\cdot\gamma_{j-1}^{\epsilon_{j-1}}\ldots\gamma_{1}^{\epsilon_{1}}\right)\right)\mathbf{1}_{\left\{ \left(\gamma_{j}^{-1},x\right)\notin\mathcal{H}\right\} }\right).
\end{align*}
Recall that by the definition of $\mathfrak{F}_{n}$ in (\ref{eq:ucube}),
only for those levels $j$ such that $\omega_{j-1}={\bf 2}$, an element
$\gamma_{j}\in\mathfrak{F}_{j,n}$ has germs not in $\mathcal{H}$.
Indeed for levels $j$ with $\omega_{j-1}\neq{\bf 2}$, $\mathfrak{F}_{j,n}$
consists of a single element $g_{j}$, which has germs in $\mathcal{H}$,
see the formula (\ref{eq:gngeneral}) for $g_{j}$ and the remark
following it. Therefore the summation is over $1\le j\le n$ with
$\omega_{j-1}={\bf 2}$. We split the sums into $j\le n-k_{n}$ and
$n-k_{n}<j\le n$. 

For $n-k_{n}<j\le n$ such that $\omega_{j-1}={\bf 2}$, by definition
of the set $\mathfrak{F}_{j,n}$ in (\ref{eq:Fjn}), an element $\gamma_{j}\in\mathfrak{F}_{j,n}$
is of the form 
\[
\gamma_{j}=\tilde{g}_{j}^{v}=\zeta_{\omega_{0}}\circ\ldots\circ\zeta_{\omega_{j-1}}\left(a\mathfrak{c}_{j}^{v}\right).
\]
where $v=1^{D-\bar{j}}u1^{m+1}0$, $u\in\mathsf{W}_{2k_{n}}^{j+D-\bar{j}}$,
$m=m_{\ell(j,2k_{n})}\in\{0,1,\ldots,D-3\}$ and $\left|1^{\infty}\wedge v\right|\ge n-j+D$.
It follows from Lemma \ref{ckgerm} that the collection of points
which carry germs not in $\mathcal{H}$ is exactly
\begin{align}
B(j,v): & =\left\{ x:\ \left(\tilde{g}_{j}^{v},x\right)\notin\mathcal{H}\right\} \nonumber \\
 & =\left\{ x:\ x=x_{1}\ldots x_{j+1}v_{2}\ldots v_{D-\bar{j}+2k_{n}}1^{m+2}01^{\infty}:\ j+1+\sum_{i=1}^{j+1}x_{i}\mbox{ is odd}\right\} .\label{eq:location}
\end{align}
In particular, the cardinality of this set is $2^{j}$ and on the
Schreier graph the distance from a point in this set to $1^{\infty}$
is comparable to $2^{j+2k_{n}}$. 

\begin{lem}\label{invertedist}

Let $n-k_{n}<j\le n$ be a level such that $\omega_{j-1}={\bf 2}$
and $\tilde{g}_{j}^{v}\in\mathfrak{F}_{j,n}$. Suppose $f:\mathbb{R}\to\mathbb{R}$
is a non-increasing function. Let $\left(\boldsymbol{\gamma},\boldsymbol{\epsilon}\right)$
be a random variable with uniform distribution on the set $\varLambda_{n}$
defined in (\ref{eq:lambdan}). 
\begin{description}
\item [{(i)}] Let $x$ be a point such that $\left(\tilde{g}_{j}^{v},x\right)\notin\mathcal{H}$,
then 
\[
\mathbb{E}\left(f\left(d\left(o,x\cdot\gamma_{j+1}^{-\epsilon_{j+1}}\ldots\gamma_{n}^{-\epsilon_{n}}\right)\right)\right)\le f\left(2^{j+2k_{n}}\right)+f\left(2^{n+D}\right)k_{n}^{1+\frac{2}{D}}2^{-\frac{1}{D}\left(j+2k_{n}-n\right)}.
\]
\item [{(ii)}] Let $x$ be a point such that $\left(\left(\tilde{g}_{j}^{v}\right)^{-1},x\right)\notin\mathcal{H}$,
then 
\[
\mathbb{E}\left(f\left(d\left(o,x\cdot\gamma_{j-1}^{\epsilon_{j-1}}\ldots\gamma_{1}^{\epsilon_{1}}\right)\right)\right)\le f\left(2^{j+2k_{n}}\right)+f\left(2^{n}\right)k_{n}^{1+\frac{2}{D}}2^{-\frac{1}{D}\left(j+2k_{n}-n\right)}.
\]
\end{description}
\end{lem}

\begin{proof}[Proof of Lemma \ref{invertedist}]

(i). Let $x$ be a point from the set $B(j,v)=\left\{ x:\ \left(\tilde{g}_{j}^{v},x\right)\notin\mathcal{H}\right\} $
as described in (\ref{eq:location}). Let $q_{0}$ be the position
of the first occurrence of $0$ in $\mathfrak{s}^{n+D}x$. Note that
by (\ref{eq:location}), $2\le q_{0}\le D-\bar{j}+2k_{n}+D$. 

We first claim that for all $m\in\{j+1,\ldots,n\}$, in $x\cdot\gamma_{j+1}^{-\epsilon_{j+1}}\ldots\gamma_{m}^{-\epsilon_{m}}$
the substring in positions $n+D+1$ to $n+D+q_{0}$ remains $1^{q_{0}-1}0$.
To see this, denote by $z$ the length $n+D+1$ prefix of $x$, $z=x_{1}\ldots x_{n+D+1}$,
and consider the section at $z$:
\begin{equation}
\left(\gamma_{j+1}^{-\epsilon_{j+1}}\ldots\gamma_{m}^{-\epsilon_{m}}\right)_{z}=\left(\gamma_{j+1}^{-\epsilon_{j+1}}\right)_{z}\left(\gamma_{j+2}^{-\epsilon_{j+2}}\right)_{z\cdot\gamma_{j+1}^{-\epsilon_{j+1}}}\ldots\left(\gamma_{m}^{-\epsilon_{m}}\right)_{z\cdot\gamma_{j+1}^{-\epsilon_{j+1}}\ldots\gamma_{m-1}^{-\epsilon_{m-1}}}.\label{eq:sectionpro}
\end{equation}
Note that by (\ref{eq:location}), $z$ has $1^{n+D-j}$ as suffix.
Recall that by Fact \ref{same} the action of $\gamma_{\ell}$ is
the same as $g_{\ell}$ on the finite level $\mathsf{L}_{n+D+1}$.
By definitions of the elements $g_{\ell}$ in (\ref{eq:gngeneral})
and Lemma \ref{translate} we have that $d(y,y\cdot g_{\ell})\le2^{\ell+2}$
for any $y\in\mathsf{L}_{n+D+1}$. It follows that 
\[
d\left(z,z\cdot\gamma_{j+1}^{-\epsilon_{j+1}}\ldots\gamma_{\ell}^{-\epsilon_{\ell}}\right)\le2^{j+3}+\ldots+2^{\ell+2}\le2^{\ell+3}-1.
\]
Therefore 
\begin{align*}
d\left(1^{n+D+1},z\cdot\gamma_{j+1}^{-\epsilon_{j+1}}\ldots\gamma_{\ell}^{-\epsilon_{\ell}}\right) & \le d\left(1^{n+D+1},z\right)+d\left(z,z\cdot\gamma_{j+1}^{-\epsilon_{j+1}}\ldots\gamma_{\ell}^{-\epsilon_{\ell}}\right)\\
 & \le2^{j+1}-1+2^{\ell+3}-1.
\end{align*}
It follows in particular since $D\ge3$, for $\ell\le n$, the $(n+D+1)$-th
digit of $z\cdot\gamma_{j+1}^{-\epsilon_{j+1}}\ldots\gamma_{\ell}^{-\epsilon_{\ell}}$
remains $1$. From the description of sections of $\gamma_{\ell}^{-1}$
we have that in the product (\ref{eq:sectionpro}) every factor belongs
to the set 
\[
\{id,b_{\mathfrak{s}^{n+D+1}\omega}\}\cup\left\{ \mathfrak{c}_{\mathfrak{s}^{n+D+1}\omega}^{u}\right\} ,
\]
where $\mathfrak{c}_{\mathfrak{s}^{n+D+1}\omega}^{u}$ is the section
at $1$ of $\mathfrak{c}_{\mathfrak{s}^{n+D}\omega}^{u}$, that is
$\mathfrak{c}_{\mathfrak{s}^{n+D}\omega}^{u}=(\omega_{n+D}(c),\mathfrak{c}_{\mathfrak{s}^{n+D+1}\omega}^{u})$,
$u\in\mathsf{W}_{k}^{n+D}$ for some $k$. Action of any element from
this set of possible sections fixes the vertex $1^{q_{0}-2}0$. The
claim follows. 

By the claim shown in the previous paragraph, the digit $0$ at the
$(n+D+q_{0})$-th position ensures that for any $(\boldsymbol{\gamma},\boldsymbol{\epsilon})$
and any point $x$ such that $\left(\tilde{g}_{j}^{v},x\right)\notin\mathcal{H},$
\[
d\left(o,x\cdot\gamma_{j+1}^{-\epsilon_{j+1}}\ldots\gamma_{n}^{-\epsilon_{n}}\right)\ge2^{n+D+1}.
\]
Denote by $A_{j}(x)$ the event
\[
A_{j}(x)=\left\{ (\mathbf{\boldsymbol{\gamma}},\boldsymbol{\epsilon}):\ d\left(x,x\cdot\gamma_{j+1}^{-\epsilon_{j+1}}\ldots\gamma_{n}^{-\epsilon_{n}}\right)\ge\frac{1}{2}d(o,x)\right\} .
\]
Since $f$ is assumed to be non-increasing, we have
\begin{align*}
 & \mathbb{E}\left(f\left(d\left(o,x\cdot\gamma_{j+1}^{-\epsilon_{j+1}}\ldots\gamma_{n}^{-\epsilon_{n}}\right)\right)\right)\\
= & \mathbb{E}\left(f\left(d\left(o,x\cdot\gamma_{j+1}^{-\epsilon_{j+1}}\ldots\gamma_{n}^{-\epsilon_{n}}\right)\right)\mathbf{1}_{A_{j}(x)^{c}}\right)+\mathbb{E}\left(f\left(d\left(o,x\cdot\gamma_{j+1}^{-\epsilon_{j+1}}\ldots\gamma_{n}^{-\epsilon_{n}}\right)\right)\mathbf{1}_{A_{j}(x)}\right)\\
\le & f\left(\frac{1}{2}d\left(o,x\right)\right)+f\left(2^{n+D}\right)\mathbb{P}\left(A_{j}(x)\right).
\end{align*}
Since $d(o,x)\ge2^{j+2k_{n}+2}$, we have
\begin{align*}
\mathbb{P}\left(A_{j}(x)\right) & \le\mathbb{P}\left(d\left(x,x\cdot\gamma_{j+1}^{-\epsilon_{j+1}}\ldots\gamma_{n}^{-\epsilon_{n}}\right)\ge2^{j+2k_{n}}\right)\\
 & \le\sum_{\ell=j+1}^{n}\mathbb{P}\left(d\left(x\cdot\gamma_{j+1}^{-\epsilon_{j+1}}\ldots\gamma_{\ell}^{-\epsilon_{\ell}},x\cdot\gamma_{j+1}^{-\epsilon_{j+1}}\ldots\gamma_{\ell+1}^{-\epsilon_{\ell+1}}\right)\ge\frac{1}{n-j}2^{j+2k_{n}}\right)\\
 & \le(n-j)2^{-\frac{1}{D}\left(j+2k_{n}-n-\log_{2}(n-j)-\log_{2}k_{n}\right)}.
\end{align*}
In the last step we applied Lemma \ref{modifiedtail}. We have proved
(i).

(ii). The proof of (ii) is similar, with $\gamma_{j+1}^{-\epsilon_{j+1}}\ldots\gamma_{m}^{-\epsilon_{m}}$
replaced by $\gamma_{m}^{\epsilon_{m}}\ldots\gamma_{1}^{\epsilon_{1}}$.
We omit the repetition. 

\end{proof}

We now return to the proof of Proposition \ref{zetasum}. For $1\le j\le n-k_{n}$
and $\omega_{j-1}={\bf 2}$, $\gamma_{j}=g_{j}$, we have
\[
B_{j}=\{x:\ (g_{j},x)\notin\mathcal{H}\}=\left\{ x:\ x_{1}\ldots x_{j+1}1^{\infty},\ j+1+\sum_{i=1}^{j+1}x_{i}\mbox{ is odd}\right\} .
\]
Since $\gamma_{j}=g_{j}$ in this case, we can write
\begin{align*}
 & \mathbb{E}\left(\sum_{x\in o\cdot G}f\left(d\left(o,x\cdot\gamma_{j+1}^{-\epsilon_{j+1}}\ldots\gamma_{n}^{-\epsilon_{n}}\right)\right)\mathbf{1}_{\left\{ \left(\gamma_{j},x\right)\notin\mathcal{H}\right\} }\right)\\
= & \sum_{x:(g_{j},x)\notin\mathcal{H}}\mathbb{E}\left(f\left(d\left(o,x\cdot\gamma_{j+1}^{-\epsilon_{j+1}}\ldots\gamma_{n}^{-\epsilon_{n}}\right)\right)\right).
\end{align*}
Note that if we choose ${\bf x}$ uniform from $B_{j}$ and $(\boldsymbol{\gamma},\boldsymbol{\epsilon})$
uniform from $\varLambda_{n}$, then in ${\bf x}\cdot\gamma_{j+1}^{-\epsilon_{j+1}}\ldots\gamma_{n}^{-\epsilon_{n}}$
the distribution of the first $n$ digits is uniform on $\{0,1\}^{n}$.
Therefore 
\begin{align*}
 & \sum_{x:(g_{j},x)\notin\mathcal{H}}\mathbb{E}\left(f\left(d\left(o,x\cdot\gamma_{j+1}^{-\epsilon_{j+1}}\ldots\gamma_{n}^{-\epsilon_{n}}\right)\right)\right)\\
\le & 2^{j}\mathbb{E}_{{\bf u}_{\mathsf{L}_{n}}}\left(f\left(d(1^{n},{\bf z})\right)\right)\le2^{j}\sum_{i=0}^{n}f(2^{n-i})2^{-i}.
\end{align*}
For $n-k_{n}<j\le n$, use (i) in Lemma \ref{invertedist}. Combining
these two parts, we obtain 
\begin{align*}
 & \sum_{j=1}^{n}\mathbb{E}\left(\sum_{x\in o\cdot G}f\left(d\left(o,x\cdot\gamma_{j+1}^{-\epsilon_{j+1}}\ldots\gamma_{n}^{-\epsilon_{n}}\right)\right)\mathbf{1}_{\left\{ \left(\gamma_{j},x\right)\notin\mathcal{H}\right\} }\right)\\
\le & \sum_{j=1}^{n-k_{n}}2^{j}\sum_{i=0}^{n}f(2^{n-i})2^{-i}+\sum_{j=n-k_{n}+1}^{n}2^{j}\left(f\left(2^{j+2k_{n}}\right)+f\left(2^{n}\right)k_{n}^{1+\frac{2}{D}}2^{-\frac{1}{D}\left(j+2k_{n}-n\right)}\right)\\
\le & C2^{n-k_{n}}f(2^{n})+C\sum_{j=n-k_{n}+1}^{n}2^{j}f\left(2^{j+2k_{n}}\right)\le C'2^{n}f\left(2^{n+2k_{n}}\right).
\end{align*}
In the last line we applied the assumption on $f$ that $f(2x)\ge2^{-1/D'}f(x)$
and $D'>D\ge3$. 

The calculation for $\check{\upsilon}_{n}$ is similar. For $1\le j\le n-k_{n}$,
use the fact that the first $(j-1)$-digits of $x\cdot\gamma_{1}^{\epsilon_{1}}\ldots\gamma_{j-1}^{\epsilon_{j-1}}$
is uniform, we have 
\begin{align*}
 & \mathbb{E}\left(\sum_{x\in o\cdot G}f\left(d\left(o,x\cdot\gamma_{1}^{\epsilon_{1}}\ldots\gamma_{j-1}^{-\epsilon_{j-1}}\right)\right)\mathbf{1}_{\left\{ \left(\gamma_{j}^{-1},x\right)\notin\mathcal{H}\right\} }\right)\\
\le2^{j} & \mathbb{E}_{{\bf u}_{\mathsf{L}_{j}}}\left(f\left(d\left(o,{\bf z}\right)\right)\right)\le2^{j}\sum_{i=0}^{j}f(2^{j-i})2^{-i}.
\end{align*}
Combine with (ii) in Lemma \ref{invertedist}, we conclude that 

\begin{align*}
 & \sum_{j=1}^{n}\mathbb{E}\left(\sum_{x\in o\cdot G}f\left(d_{\mathcal{S}}\left(o,x\cdot\gamma_{j-1}^{\epsilon_{j-1}}\ldots\gamma_{1}^{\epsilon_{1}}\right)\right)\mathbf{1}_{\left\{ \left(\gamma_{j}^{-1},x\right)\notin\mathcal{H}\right\} }\right).\\
\le & \sum_{j=1}^{n-k_{n}-1}2^{j}\sum_{i=0}^{j}f(2^{j-i})2^{-i}++\sum_{j=n-k_{n}}^{n}2^{j}\left(f\left(2^{j+2k_{n}}\right)+f\left(2^{n}\right)k_{n}^{1+\frac{2}{D}}2^{-\frac{1}{D}\left(j+2k_{n}-n\right)}\right)\\
\le & C2^{n-k_{n}}f(2^{n-k_{n}})+C\sum_{j=n-k_{n}}^{n}2^{j}f\left(2^{j+2k_{n}}\right)\le C'2^{n}f\left(2^{n+2k_{n}}\right).
\end{align*}
In the last line we applied the assumption on $f$ that $f(2x)\ge2^{-1/D'}f(x)$
and $D'>D\ge3$. 

\end{proof}

\section{Applications to volume lower estimates of Grigorchuk groups\label{sec:maingrowth}}

In this section we derive volume lower estimate for $G_{\omega}$
from Theorem \ref{stabilization}, where $\omega$ satisfies Assumption
$({\rm Fr}(D))$ as defined in Notation \ref{D}. Throughout the section
we use notations for $G_{\omega}$ as reviewed in Subsection \ref{subsec:Grigorchuk-def}.

Let $\omega$ be a string that satisfies Assumption $({\rm Fr}(D))$
and $\mu_{\beta}$ be a measure with non-trivial Poisson boundary
as in Theorem \ref{stabilization}. The construction of $\mu_{\beta}$
is in Subsection \ref{subsec:elements}. Recall that $\mu_{\beta}$
is built from measures $\upsilon_{n}$, $n\in D\mathbb{N}$. We bound
the length of elements in the support of $\upsilon_{n}$. Let $M_{i}$
be the matrix associated with the substitution $\zeta_{i}$, $i\in\{{\bf 0},{\bf 1},{\bf 2}\}$,
such that ${\bf l}(\zeta_{i}(\omega))=M_{i}{\bf l}(\omega)$ where
${\bf l}(\omega)$ is the column vector that records occurrence of
$ab,ac,ad$ in $\omega$. Explicitly
\[
M_{{\bf 0}}=\left[\begin{array}{ccc}
2 & 0 & 1\\
0 & 2 & 1\\
0 & 0 & 1
\end{array}\right],M_{{\bf 1}}=\left[\begin{array}{ccc}
2 & 1 & 0\\
0 & 1 & 0\\
0 & 1 & 2
\end{array}\right],M_{{\bf 2}}=\left[\begin{array}{ccc}
1 & 0 & 0\\
1 & 2 & 0\\
1 & 0 & 2
\end{array}\right].
\]
Define the number $L_{n}^{\omega}$ to be
\begin{equation}
L_{n}^{\omega}:=\left(\begin{array}{ccc}
1 & 1 & 1\end{array}\right)M_{\omega_{0}}\ldots M_{\omega_{n-1}}\left(\begin{array}{c}
1\\
1\\
1
\end{array}\right).\label{eq:Ln}
\end{equation}

\begin{lem}\label{lengthv}

Let $\upsilon_{n}$ be defined as in (\ref{eq:zetan}) and $L_{n}^{\omega}$
be as in (\ref{eq:Ln}). Then 
\[
\sup\{|g|:g\in{\rm supp}\upsilon_{n}\}\le2^{2k_{n}+2D+4}L_{n}^{\omega}.
\]

\end{lem}

\begin{proof}

An element in the support of $\upsilon_{n}$ is of the form $g=\gamma_{n}^{\epsilon_{n}}\ldots\gamma_{1}^{\epsilon_{1}}$,
where $\gamma_{j}\in\mathfrak{F}_{j,n}$. For those indices $j$ such
that $\mathfrak{F}_{j,n}=\{g_{j}\}$, we have 
\[
\left|\gamma_{j}\right|\le\left(\begin{array}{ccc}
2 & 2 & 2\end{array}\right)M_{\omega_{0}}\ldots M_{\omega_{j-1}}\left(\begin{array}{c}
1\\
1\\
1
\end{array}\right).
\]
Otherwise $\gamma_{j}=\tilde{g}_{j}^{v}$ for some $v$ as in (\ref{eq:gjv}).
By definition of $\mathfrak{c}_{j}^{v}$ in (\ref{eq:ckv}) we have
\[
\left|\mathfrak{c}_{j}^{v}\right|{}_{G_{\mathfrak{s}^{j}\omega}}\le1+2\sum_{i=1}^{2k_{n}+2D}\left|h_{i}^{v}\right|{}_{G_{\mathfrak{s}^{j}\omega}},
\]
where $h_{j}^{v}$ is defined in (\ref{eq:hv}). The element $\iota\left(\left[b_{\mathfrak{s}^{j+i-2}\omega},a\right],v_{1}\ldots v_{i-2}\right)$
can be obtained by applying the substitutions $\sigma_{i}$. We have
$\left|h_{i}^{v}\right|{}_{G_{\mathfrak{s}^{j}\omega}}\le2^{i+1}$.
Therefore 
\[
\left|a\mathfrak{c}_{j}^{v}\right|{}_{G_{\mathfrak{s}^{j}\omega}}\le2+2\sum_{i=1}^{2k_{n}+2D}\left|h_{i}^{v}\right|{}_{G_{\mathfrak{s}^{j}\omega}}\le2^{2k_{n}+2D+2}.
\]
Since $\tilde{g}_{j}^{v}$ is obtained by applying substitutions $\zeta_{i}$
to $a\mathfrak{c}_{j}^{v}$, 
\[
\left|\tilde{g}_{j}^{v}\right|\le2^{2k_{n}+2D+2}\left(\begin{array}{ccc}
2 & 2 & 2\end{array}\right)M_{\omega_{0}}\ldots M_{\omega_{j-1}}\left(\begin{array}{c}
1\\
1\\
1
\end{array}\right).
\]
From the triangle inequality $|g|\le\sum_{j=1}^{n}|\gamma_{j}|$,
we have for $g\in{\rm supp}\upsilon_{n}$, 
\[
|g|\le\sum_{j<n-k_{n}}2L_{j}^{\omega}+\sum_{n-k_{n}\le j\le n}2^{2k_{n}+D+2}2L_{j}^{\omega}\le2^{2k_{n}+D+4}L_{n}^{\omega}.
\]

\end{proof}

We obtain a tail estimate of $\mu_{\beta}$ as an immediate corollary
of Lemma \ref{lengthv}.

\begin{cor}\label{betatail}

Let $\mu_{\beta}$ be defined as in (\ref{eq:eta3}) with parameters
$\beta>0$ and $(k_{n})$, $k_{n}\le n$. Then $\mu_{\beta}$ is of
finite entropy and there exists constant $C>0$ depending on $\beta,D$
such that
\[
\mu_{\beta}\left(B\left(id,2^{2k_{n}}L_{n}^{\omega}\right)^{c}\right)\le\frac{C}{2^{n\beta}}.
\]

\end{cor}

\begin{proof}

Since the size of the support of $\upsilon_{n}$ is bounded by $\left|\varLambda_{n}\right|\le2^{n}\left(2^{2k_{n}}\right)^{k_{n}}$,
the entropy of $\mu_{\beta}$ is bounded by
\[
H(\mu_{\beta})\le C\sum_{D|n}2^{-n\beta}H(\upsilon_{n})\le C\sum_{D|n}2^{-n\beta}\left(n+2k_{n}^{2}\right)<\infty.
\]
By Lemma \ref{lengthv}, $\sup\{|g|:g\in{\rm supp}\upsilon_{n}\}\le2^{2k_{n}+2D+4}L_{n}$,
thus 
\[
\mu_{\beta}\left(B\left(id,2^{2k_{n}}L_{n}^{\omega}\right)^{c}\right)\le C_{\beta}\sum_{j\ge n-2D}\frac{1}{2^{j\beta}}\le C2^{-n\beta}.
\]

\end{proof}

Combine Theorem \ref{stabilization} with the tail estimate in Corollary
\ref{betatail} we deduce the following.

\begin{thm}\label{periodic1}

Let $\omega$ be a string satisfying Assumption $({\rm Fr}(D))$.
For any $\epsilon>0$, there exists a constant $C=C(D,\epsilon)>0$
and a non-degenerate symmetric probability measure $\mu$ on $G_{\omega}$
of finite entropy and nontrivial Poisson boundary with tail decay
\[
\mu\left(B(id,L_{n}^{\omega})^{c}\right)\le Cn^{-1+\epsilon},
\]
where the number $L_{n}^{\omega}$ is defined in (\ref{eq:Ln}). As
a consequence, there exists constant $c=c(\omega,\epsilon)>0$ such
that for all $n\ge1$,
\[
v_{G_{\omega},S}(L_{n}^{\omega})\ge\exp\left(cn^{1-\epsilon}\right).
\]

\end{thm}

\begin{proof}

Given $\epsilon>0$, take $\max\{1-\epsilon,1-1/D\}<\beta<1$. By
Theorem \ref{stabilization} there exists $A>0$ such that the measure
$\mu_{\beta}$ with $k_{n}=A\left\lfloor \log_{2}n\right\rfloor $
has non-trivial Poisson boundary on $G_{\omega}$. Note that $L_{n+1}^{\omega}\ge2L_{n}^{\omega}$
for any $n$ and $\omega$. By Corollary \ref{betatail}, 
\[
\mu_{\beta}\left(B\left(id,n^{2A}L_{n}^{\omega}\right)^{c}\right)\le\frac{C}{2^{n\beta}}.
\]
Then the tail estimate stated follows. The volume lower bound follows
from Lemma \ref{moment_growth}.

\end{proof}

Denote by $W_{\omega}$ the permutation wreath extension $(\mathbb{Z}/2\mathbb{Z})\wr_{\mathcal{S}}G_{\omega}$
and $T$ the generating set $\left\{ a,b_{\omega},c_{\omega},d_{\omega},\left(\delta_{1^{\infty}}^{1},id_{G_{\omega}}\right)\right\} $.
By \cite{BE2}, for any $\omega\in\{{\bf 0,1,2}\}^{\infty}$, the
growth function of $W_{\omega}$ satisfies 
\[
v_{W_{\omega},T}(L_{n}^{\omega})\le\exp\left(Cn\right)
\]
for some absolute constant $C>0$. Note that $v_{G_{\omega},S}\le v_{W_{\omega},T}$,
thus combined with Theorem \ref{periodic1}, for $\omega$ satisfying
Assumption $({\rm Fr}(D))$, for any $\epsilon>0$, we have that 
\begin{equation}
\exp\left(cn^{1-\epsilon}\right)\le v_{G_{\omega},S}(L_{n}^{\omega})\le\exp(Cn),\label{eq:twosided}
\end{equation}
where the constant $c$ depends on $\omega,\epsilon$ and the constant
$C$ is absolute. 

We note the following relation between the critical constant of recurrence
$c_{{\rm rt}}\left(G_{\omega},{\rm Stab}_{G_{\omega}}(1^{\infty})\right)$
as in Subsection \ref{subsec:critical} and the volume exponent of
$G_{\omega}$. We say the volume exponent of $G$ exists if the limit
$\lim_{r\to\infty}\log\log v_{G,S}(r)/\log r$ exists and the limit
value, if exists, is called the volume exponent of $G$. 

\begin{cor}

Let $\omega\in\{{\bf 0,1,2}\}^{\infty}$ be a string that satisfies
Assumption $({\rm Fr}(D))$. Assume that $\frac{1}{n}\log L_{n}^{\omega}$
converges when $n\to\infty$, and let $\lambda_{\omega}=\lim_{n\to\infty}\frac{1}{n}\log L_{n}^{\omega}$.
 Then the volume exponent $\alpha_{\omega}$ of $G_{\omega}$ exists
and it is related to $\lambda_{\omega}$ through
\[
\alpha_{\omega}=\frac{1}{\log_{2}\lambda_{\omega}}.
\]
Moreover, the critical constant of recurrence of $\left(G_{\omega},{\rm Stab}_{G_{\omega}}(1^{\infty})\right)$
satisfies
\[
c_{{\rm rt}}\left(G_{\omega},{\rm Stab}_{G_{\omega}}(1^{\infty})\right)=\alpha_{\omega}.
\]

\end{cor}

\begin{proof}

The statement on the volume exponent $\alpha_{\omega}$ follows from
(\ref{eq:twosided}) as explained above. The proof of the claim on
the critical exponent is the same as in Theorem \ref{critical}: suppose
it is not true, then there exists $\delta>0$ and a measure $\mu$
on $G_{\omega}$ with finite $\left(\alpha_{\omega}+\delta\right)$-moment
such that the induced random walk $P_{\mu}$ on the orbit of $1^{\infty}$
is transient. On the permutation wreath extension $W_{\omega}$, consider
the measure $\zeta=\frac{1}{2}$$(\mu+\nu),$where $\nu$ is uniform
on $\left\{ id,\delta_{1^{\infty}}^{1}\right\} $. Since the $P_{\mu}$-random
walk on $1^{\infty}\cdot G$ is transient, the measure $\frac{1}{2}(\mu+\nu)$
has non-trivial Poisson boundary by \cite[Proposition 3.3]{BE3}.
Since $\zeta$ has finite $(\alpha_{\omega}+\delta)$-moment, the
upper bound $v_{W_{\omega}}(n)\lesssim\exp(n^{\alpha_{\omega}})$
and Corollary \ref{trivialmoment} implies $(W_{\omega},\zeta)$ has
trivial Poisson boundary, which is a contradiction.

\end{proof}

Theorem A stated in the Introduction is the special case of Theorem
\ref{periodic1} applied to the periodic sequence $({\bf 012})^{\infty}$.

\begin{proof}[Proof of Theorem A]

The first Grigorchuk group is indexed by the string $\omega=({\bf 012})^{\infty}$,
which satisfies Assumption ${\rm Fr}(D)$ with $D=3$. Explicit calculation
of eigenvalues the substitution matrix shows that $L_{n}=L_{n}^{({\bf 012})}\le3\lambda_{0}^{n}$,
where $\lambda_{0}$ is the positive root of the polynomial $X^{3}-X^{2}-2X-4$.
The statement of Theorem \ref{bound1st} follows from Theorem \ref{periodic1}.

\end{proof}

More generally, apply Theorem \ref{periodic1} to periodic strings
which contain all three letters, we obtain Theorem B stated in the
Introduction. By \cite[Proposition 4.6]{BE2}, the set of quantities
$\frac{q\log2}{\log\lambda_{\omega}}$ , where $\lambda_{\omega}$
is the spectral radius of $M_{\omega_{0}}\ldots M_{\omega_{q-1}}$,
coming from periodic $\omega$ containing all three letters is dense
in $[\alpha_{0},1]$. Therefore the volume exponents of the collection
$\{G_{\omega}\}$, where $\omega$ is a periodic string with all three
letters, form a dense subset of $[\alpha_{0},1]$. 

\begin{proof}[Proof of Theorem B]

Up to a renaming of letters $\{{\bf 0},{\bf 1},{\bf 2}\}$, a periodic
string $\omega$ containing all three letters satisfies Assumption
$({\rm Fr}(D))$ with $D=2q$. For the tail estimate in Theorem \ref{periodic1},
note that the length $L_{n}^{\omega}$ satisfies $L_{n}^{\omega}\le3\lambda^{n/q}$
where $\lambda=\lambda_{\omega_{0}\ldots\omega_{q-1}}$ is the spectral
radius of $M_{\omega_{0}}\ldots M_{\omega_{q-1}}$. Therefore by Theorem
\ref{periodic1}, for any $\epsilon>0$, there exists a constant $c>0$
such that for all $n\ge1$,
\[
v_{G_{\omega}}(n)\ge\exp(cn^{\alpha-\epsilon}),
\]
where $\alpha=\frac{q\log2}{\log\lambda}$. By \cite[Proposition 4.4]{BE2},
for such a period string $\omega$, the permutation wreath extension
$W_{\omega}=(\mathbb{Z}/2\mathbb{Z})\wr_{\mathcal{S}}G_{\omega}$
has growth
\[
v_{W_{\omega}}(n)\simeq\exp\left(n^{\frac{q\log2}{\log\lambda}}\right).
\]
Therefore $v_{G_{\omega}}(n)\lesssim v_{W_{\omega}}(n)\lesssim\exp\left(n^{\alpha}\right)$.
The statement follows.

\end{proof}

We now consider strings $\omega$ in $\{{\bf 201},{\bf 211}\}^{\ast}$
that are not necessarily periodic. The main result of \cite{BE2}
states that any sub-multiplicative function that grows uniformly faster
than $\exp\left(n^{\alpha_{0}}\right)$ can be realized as the growth
function of some permutational wreath extension $W_{\omega}=(\mathbb{Z}/2\mathbb{Z})\wr_{\mathcal{S}}G_{\omega}$,
$\omega\in\{{\bf 0},{\bf 1},{\bf 2}\}^{\infty}$. The strategy in
\cite{BE2} to choose a string to approximate a prescribed volume
function can be implemented in $\{{\bf 201},{\bf 211}\}^{\ast}$.
Roughly speaking, the procedure relies on the fact that the string
${\bf 201}$ exhibits the sum contracting coefficient $\eta_{0}$
(as in the first Grigorchuk group) while for the string ${\bf 211}$
the coefficient is $1$ (no sum-contraction). Careful choice of length
$(i_{k},j_{k})$ in the product $({\bf 201})^{i_{1}}({\bf 211})^{j_{1}}({\bf 201})^{i_{2}}({\bf 211})^{j_{2}}\ldots$
of these two types of strings allows to approximate prescribed volume
function as in \cite[Section 5]{BE2}. 

\begin{thm}\label{prescribe}

Let $f:\mathbb{R}_{+}\to\mathbb{R}_{+}$ be a function such that $f(2R)\le f(R)^{2}\le f(2R/\eta_{0})$,
where $\eta_{0}$ is the real root of $X^{3}+X^{2}+X-2$. Then there
exists a string $\omega$ in $\{{\bf 201},{\bf 211}\}^{\ast}$ such
that the volume growth of $G_{\omega}$ satisfies that for any $\epsilon>0$,
\[
\exp\left(\frac{c_{\epsilon}}{n^{\epsilon}}\log f(n)\right)\le v_{G_{\omega},S}\left(n\right)\le f(Cn),
\]
where $C>0$ is a constant independent of $\epsilon$ and $c_{\epsilon}>0$
is a constant depending on $\epsilon$. 

\end{thm}

\begin{proof}

By the volume estimate \cite[Corollary 4.2]{BE2} and \cite[Lemma 5.1, Corollary 5.2]{BE2},
we have that for any string $\omega$ in $\{{\bf 0,1,2}\}^{\infty}$,
there exists a constants $C,c>0$ such that 
\begin{equation}
\exp(cn)\le v_{W_{\omega},T}\left(L_{n}^{\omega}\right)\le\exp(Cn),\label{eq:equiv}
\end{equation}
where $W_{\omega}$ is the permutation wreath product $(\mathbb{Z}/2\mathbb{Z})\wr_{\mathcal{S}}G_{\omega}$
and $L_{n}^{\omega}$ is defined in (\ref{eq:Ln}).

Let $f:\mathbb{R}_{+}\to\mathbb{R}_{+}$ be a function such that $f(2R)\le f(R)^{2}\le f(2R/\eta_{0})$,
where $\eta_{0}$ is the real root of $X^{3}+X^{2}+X-2$. For such
a function $f$, the proof in \cite[Section 5]{BE2} produces a string
$\omega$ is the form $({\bf 012})^{i_{1}}{\bf 2}^{j_{1}}({\bf 012})^{i_{2}}{\bf 2}^{j_{3}}\ldots$
such that the volume growth of $W_{\omega}$ is equivalent to $f$.
For this result, one can use a string of the form $({\bf 201})^{i_{1}}({\bf 211}){}^{j_{1}}({\bf 201})^{i_{2}}({\bf 211})^{j_{2}}\ldots$
instead. This is because the properties of $({\bf 012})^{n}$ and
${\bf 2}^{n}$ in \cite[Lemma 5.3 ]{BE2} are satisfied by ${\bf 201}$
and ${\bf 211}$. Following the notations of \cite{BE2}, similar
to \cite[Lemma 5.3 ]{BE2}, there exists constants $A'\le1$, $B'\ge1$
such that for all $V\in\Delta'$ and $n\in\mathbb{N}$,
\begin{align*}
\eta\left(V,({\bf 201})^{n}\right) & \ge\eta_{0}^{3n}A',\\
\eta(V,({\bf 211})^{n}) & \le2^{3n}B'.
\end{align*}

Using this property of $({\bf 201})^{n}$ instead of $({\bf 012})^{n}$
and $({\bf 211})^{n}$ instead of ${\bf 2}^{n}$, the proof in \cite[Section 5]{BE2}
carries over verbatim and shows the following: for a prescribed function
$f:\mathbb{R}_{+}\to\mathbb{R}_{+}$ be a function such that $f(2R)\le f(R)^{2}\le f(2R/\eta_{0})$,
there exists a string $\omega\in\{{\bf 201},{\bf 211}\}^{\infty}$
such that the volume growth of $W_{\omega}$ is equivalent to $f$,
that is, there exists constants $C,c>0$ such that $f(cn)\le v_{W_{\omega},T}(n)\le f(Cn)$. 

By Theorem \ref{periodic1}, since $\omega\in\{{\bf 201},{\bf 211}\}^{\infty}$
satisfies Assumption$(D)$ with $D=3$, we have that for any $\epsilon>0$,
there exists a constant $c>0$ such that for all $n\ge1$, 
\[
v_{G_{\omega},S}(L_{n}^{\omega})\ge\exp\left(cn^{1-\epsilon}\right).
\]
Combined with the upper estimates on $v_{W_{\omega}}(n)$ in (\ref{eq:equiv})
we have that if the volume of $W_{\omega}$ is equivalent to $f$,
then for any $\epsilon>0$ there exists $c_{\epsilon}>0$ such that
\[
\exp\left(\frac{c_{\epsilon}}{n^{\epsilon}}\log f(n)\right)\le v_{G_{\omega},S}\left(n\right)\le f(n).
\]

\end{proof}

\begin{proof}[Proof of Theorem \ref{liminf}]

Let $\alpha,\beta$ be as given in the statement. Take a function
$f$ satisfying $f(2r)\le f(r)\le f(2r/\eta_{0})$, 
\[
\liminf_{r\to\infty}\frac{\log\log f(r)}{\log r}=\alpha\mbox{ and }\limsup\frac{\log\log f(r)}{\log r}=\beta.
\]
Then by Theorem \ref{prescribe}, there exists a string $\omega\in\{{\bf 201},{\bf 211}\}^{\infty}$
such that for any $\epsilon>0$, $\exp\left(\frac{c_{\epsilon}}{n^{\epsilon}}\log f(n)\right)\le v_{G_{\omega},S}\left(n\right)\le f(Cn)$
for some constants $c_{\epsilon},C>0$. The statement follows. 

\end{proof}

\section{Final remarks and questions }

\subsection{On $G_{\omega}$ without Assumption $({\rm Fr}(D))$}

When the string $\omega$ does not satisfy Assumption $({\rm Fr}(D))$
in Notation \ref{D}, the construction in Section \ref{sec:Main}
does not apply. For example, consider strings of the form ${\bf 0}^{i_{1}}{\bf 1}^{i_{1}}{\bf 2}^{i_{1}}{\bf 0}^{i_{2}}{\bf 1}^{i_{2}}{\bf 2}^{i_{2}}\ldots$
where $i_{n}\to\infty$ as $n\to\infty$. It is known from \cite{Grigorchuk84}
that if the sequence $(i_{n})$ grows rapidly, the growth of $G_{\omega}$
for such $\omega$ is close to exponential along a subsequence. However
different ideas are needed to obtain good volume lower estimates for
these $\omega$'s. 

For any $\omega$ that is not eventually constant, one can apply the
construction that takes a measure with restricted germs as in Section
\ref{sec: growth-first} to obtain measures with non-trivial Poisson
boundary. The same argument that proves Corollary \ref{Hbound} extends
to general $\omega$, we omit the detail here.

For the periodic string $\omega_{t}=({\bf 0}^{t-2}{\bf 1}{\bf 2})^{\infty}$,
by Theorem \ref{periodic} explicit calculation shows that when $t\to\infty$,
the the volume exponent $\alpha_{\omega_{t}}=\lim_{n\to\infty}\frac{\log\log v_{G_{{\bf 0}^{t-2}{\bf 1}{\bf 2}}}(n)}{\log n}$
satisfies 
\[
1-\alpha_{\omega_{t}}\sim\frac{c}{t}\mbox{ \ \ as }t\to\infty,
\]
for some number $c>0$. This kind of dependence on frequencies of
the letters can be observed in general strings. Given $\omega\in\{{\bf 0},{\bf 1},{\bf 2}\}^{\infty}$,
let $\ell_{n}(x)$ count the number of occurrences of the symbol $x$
in the first $n$ digits, $x\in\{{\bf 0},{\bf 1},{\bf 2}\}$. Let
\[
\theta_{n}=\frac{1}{n}\min\left\{ \ell_{n}({\bf 0}),\ell_{n}({\bf 1}),\ell_{n}({\bf 2})\right\} .
\]
Then there exists absolute constants $c,C>0$ such that 
\[
n^{1-C\theta_{n}}/\log^{1+\epsilon}n\lesssim\log v_{G_{\omega},S}(n)\lesssim n^{1-c\theta_{n}}.
\]

\subsection{Critical constant of the Liouville property}

Kaimanovich and Vershik conjectured in \cite{KV83} that any group
of exponential growth admits a probability measure with non-trivial
Poisson boundary. In a very recent paper \cite{FHTV}, Frisch, Hartman,
Tamuz and Ferdowsi answered this conjecture positively. In \cite{FHTV}
the authors completely characterize countable groups that can admit
random walks with non-trivial Poisson boundary: a countable group
$G$ admits a probability measure with non-trivial Poisson boundary
if and only if it admits a quotient with infinite conjugacy class
property. In particular, for a finitely generated group $G$, there
exists a symmetric probability measure $\mu$ of finite entropy and
non-trivial Poisson boundary on $G$ if and only if $G$ is not virtually
nilpotent. 

On a finitely generated group that is not virtually nilpotent one
can then ask whether tail decay of measures with non-trivial Poisson
boundary provides useful information about the growth of the group.
Similar in flavor to the critical constant of recurrence (see Subsection
\ref{subsec:critical}), given a countable group $\Gamma$ equipped
with a length function $l$, one can define the \emph{critical constant
of the Liouville property }of $\Gamma$ with respect to $l$ to be
$\sup\beta$ where the $\sup$ is over all $\beta\ge0$ such that
there exists a non-degenerate symmetric probability measure $\mu$
on $\Gamma$ with finite entropy and non-trivial Poisson boundary
such that $\mu$ has finite $\beta$-moment, that is $\sum_{g\in\Gamma}l(g)^{\beta}\mu(g)<\infty.$
When $\Gamma$ is finitely generated and $l$ is the word length we
omit reference to $l$. 

For example, for the lamplighter group $(\mathbb{Z}/2\mathbb{Z})\wr\mathbb{Z}$
and polycyclic groups that are not virtually nilpotent, the critical
constant of the Liouville property is $1$; while for $(\mathbb{Z}/2\mathbb{Z})\wr\mathbb{Z}^{2}$
it is $2$ and for $(\mathbb{Z}/2\mathbb{Z})\wr\mathbb{Z}^{3}$, $d\ge3$,
it is $\infty$. These examples are of exponential growth. For intermediate
growth groups, the critical constant of the Liouville property is
bounded from above by the upper growth exponent. Results in this paper
imply that for $\omega$ periodic containing all letters ${\bf 0},{\bf 1},{\bf 2}$,
the volume exponent of the Grigorchuk group $G_{\omega}$, the critical
constant of recurrence of $(G_{\omega},{\rm St}_{G_{\omega}}(1^{\infty}))$
and the critical constant of Liouville property of $G_{\omega}$ all
coincide. By \cite{Erschler04}, if $\omega$ contains exactly two
letters infinitely often, then these three exponents are all equal
to $1$ on $G_{\omega}$. 

Recall that the FC-center of a group $\Gamma$, $Z_{{\rm FC}}(\Gamma)$,
consists of elements with finite conjugacy classes. By \cite[Theorem IV. 1]{azencott},
for any non-degenerate measure $\mu$ on $\Gamma$, bounded $\mu$-harmonic
functions are constant on $Z_{{\rm FC}}(\Gamma)$. As a consequence,
the critical constant of the Liouville property for $\Gamma$ with
respect to $l$ is the same as its quotient $\Gamma/Z_{{\rm FC}}(\Gamma)$
with respect to the quotient length function $\bar{l}$. On the other
hand, the volume growth of $\Gamma$ can be much faster than that
of $\Gamma/Z_{{\rm FC}}(\Gamma)$. For example, it is known that FC-central
extensions of the first Grigorchuk group can have growth arbitrarily
close to exponential and have arbitrarily large F{\o}lner function,
see the remarks below. Therefore one cannot expect random walks with
non-trivial Poisson boundary to provide near optimal lower bound on
general intermediate growth groups. 

The Gap Conjecture (strong version) of Grigorchuk asks if the growth
of a finitely generated group being strictly smaller than $e^{\sqrt{n}}$
implies that it is polynomial. A weaker formulation of the conjecture,
called the Gap Conjecture with parameter $\beta$, asks if growth
strictly smaller than $e^{n^{\beta}}$ implies polynomial growth.
There are related gap conjecture type questions regarding various
asymptotic characteristics of the groups, see the survey of Grigorchuk
\cite{GrigorchukGap} for more information. If the Gap Conjectures
of Grigorchuk, strong or weak, hold true, one can ask (having in mind
the results of this paper) whether stronger statements hold true for
any group $\Gamma$ of super-polynomial growth: 

\begin{ques}

Let $\Gamma$ be a finitely generated group of super-polynomial growth,
is it true that there exists a finite entropy measure $\mu$ with
power-law tail decay bound, that is for some $C,\beta>0$, $\mu\left(\{g:|g|_{S}\ge r\}\right)\le Cr^{-\beta}$,
such that the Poisson boundary of $(\Gamma,\mu)$ is non-trivial?
Can one always choose such $\mu$ with $\beta\ge1/2-\epsilon$? 

\end{ques}

\begin{rem}\label{FCextension}

The following construction provides FC-central extensions. Let ${\bf F}$
be a group equipped with a finite generating set $S$ (typically a
free group or free product). Suppose we have a sequence of quotients
${\bf F}\to\Gamma_{k}$, $k\in\mathbb{N}$, and ${\bf F}\to G$. Each
quotient group is marked with the generating set $S$. Let $R_{k}$
be the largest radius $r$ such that in the Cayley graph of $\left(\Gamma_{k},S\right)$,
the ball of radius $r$ around $id$ is the same as the ball of same
radius around $id$ in the Cayley graph $(G,S)$. Take the the universal
group of the sequence $(\Gamma_{k})$, 
\[
\Gamma={\bf F}/\left(\cap_{k=1}^{\infty}\ker\left({\bf F}\to\Gamma_{k}\right)\right).
\]
Suppose $(\Gamma_{k},S)$ is a sequence of marked\emph{ finite} groups
such that with respect to $(G,S)$, $R_{k}\to\infty$ as $k\to\infty$.
Then the group $\Gamma$ is an FC-central extension of $G$. 

To see this, first note that $G$ is a quotient of $\Gamma$. Suppose
$w\in\ker({\bf F}\to\Gamma)$, then it is in $\ker({\bf F}\to\Gamma_{k})$
for every $k$. It follows that $w$ is in $\ker({\bf F}\to G)$ because
$R_{k}\to\infty$ as $k\to\infty$. We need to show any element $\gamma\in\ker\left(\Gamma\to G\right)$
has finite conjugacy class in $\Gamma$. Let $w$ be a word in $S\cup S^{-1}$
which represents $\gamma$, where $\gamma\in\ker(\Gamma\to G)$. Since
$\lim R_{k}=\infty$, there is a finite $k_{0}$ such that for $k>k_{0}$,
the ball of radius $|w|$ around $id$ in $\Gamma_{k}$ agrees with
the ball of same radius in $G$. In particular the image of $w$ in
$\Gamma_{k}$ with $k>k_{0}$ is identity. Therefore the conjugacy
class of $\gamma$ is contained in the product of $\Gamma_{k}$, $1\le k\le k_{0}$.
Since each $\Gamma_{k}$ is assumed to be finite, we conclude that
the conjugacy class of $\gamma$ is finite. 

\end{rem}

\begin{rem}\label{piecewise}

In \cite{Erschler06} it is shown that intermediate growth groups
can have arbitrarily large F{\o}lner functions. The groups are obtained
as direct products of piecewise automatic groups with returns. Each
factor group $\Gamma_{k}$ (represented as a piecewise automatic group
with return ${\rm PA}(\tau_{1},\tau_{2},i_{k},j_{k})$) is a finite
quotient of the free product ${\bf F}=\mathbb{Z}/2\mathbb{Z}\ast(\mathbb{Z}/2\mathbb{Z}\times\mathbb{Z}/2\mathbb{Z})$.
Let $S=\{a,b,c,d\}$ be a generating set of the free product, $\mathbb{Z}/2\mathbb{Z}=\left\langle a\right\rangle $
and $\mathbb{Z}/2\mathbb{Z}\times\mathbb{Z}/2\mathbb{Z}=\{id,b,c,d\}$.
Because the action of the first Grigorchuk group $G$ on the tree
is contracting, this sequence $(\Gamma_{k},S)$ satisfies the property
that with respect to $(G,S)$, $R_{k}\to\infty$ as $k\to\infty$.
In other words $(\Gamma_{k},S)$ converges to $(G,S)$ as $k\to\infty$
in the Chabauty-Grigorchuk topology, which is also called the Cayley
topology. Then the universal group of the sequence $(\Gamma_{k})$
(referred to as direct product of piecewise automatic groups in \cite{Erschler06}),
$\Gamma={\bf F}/\left(\cap_{k=1}^{\infty}\ker\left({\bf F}\to\Gamma_{k}\right)\right)$,
is an FC-central extension of $G$. By \cite{Erschler06}, $\Gamma$
is of intermediate growth and its F{\o}lner function can be arbitrarily
large with suitable choices of parameters. In particular the growth
of $\Gamma$ can be arbitrarily close to exponential. Since the group
$\Gamma$ is torsion, by \cite[Theorem IV. 1]{azencott}, for any
measure $\mu$ on $\Gamma$, bounded $\mu$-harmonic functions factor
through the quotient $\Gamma/Z_{{\rm FC}}(\Gamma)=G$. Since the growth
of $G$ is bounded by $\exp\left(n^{\alpha_{0}}\right)$, it follows
that any measure $\mu$ on $\Gamma$ with finite $\alpha_{0}$-moment
has trivial Poisson boundary. 

\end{rem}

\bibliographystyle{plain}
\bibliography{Boundary2}

\textsc{\newline Anna Erschler ---  D\'{e}partement de math\'{e}matiques et applications, \'{E}cole normale sup\'{e}rieure, CNRS, PSL Research University, 45 rue d'Ulm, 75005 Paris } --- anna.erschler@ens.fr

\textsc{\newline Tianyi Zheng --- Department of Mathematics, UC San Diego, 9500 Giman Dr. La Jolla, CA 92093 } --- tzheng2@math.ucsd.edu 
\end{document}